\documentclass[10pt,a4paper]{article}
\usepackage[latin1]{inputenc}
\usepackage{amsfonts,amssymb,latexsym}
\usepackage{amsmath,amsthm}
\usepackage[Sonny]{fncychap}
\usepackage{enumerate}
\usepackage[dvips]{color,graphicx} 
\usepackage{pst-all}
\usepackage{pstcol}
\usepackage{amsbsy,amssymb}
\usepackage{amsmath}
\usepackage{amsfonts}
\usepackage{amssymb}
\usepackage{graphicx}
\usepackage{amscd}
\usepackage{amsmath, upgreek}
\usepackage{graphicx}
\usepackage{amsfonts}
\usepackage{amssymb, wasysym}
\usepackage{amssymb}
 \usepackage{pst-all}
  \usepackage{nonfloat}
   \usepackage[bf,footnotesize]{caption2}
   \usepackage{soul} 
\usepackage{cancel}

\newtheorem{theorem}{Theorem}[section]
\newtheorem{lemma}[theorem]{Lemma}

\newtheorem{definition}[theorem]{Definition}
\newtheorem{remark}[theorem]{Remark}

\numberwithin{equation}{section}

\newcommand{\CFF}{\textbf{\textit{F}}}
\newcommand{\CF}{\mathcal{F}}
\newcommand{\CJ}{\mathcal{J}}

\newcommand{\CS}{\mathcal{S}}

\begin{document}

\title{A Boundary Control Problem Associated to the Rayleigh-Bénard-Marangoni System}
\author{{E.J. Villamizar-Roa}{\thanks{Corresponding author: jvillami@uis.edu.co. }\ \ and D.A. Rueda-G\'omez} \\
{\small Universidad Industrial de Santander, Escuela de Matem\'{a}ticas}\\
{\small A.A. 678, Bucaramanga, Colombia.} \\
}
\date{\empty}
\maketitle
\begin{abstract}
In this paper, we study a boundary control problem associated to the
stationary Rayleigh-B\'enard-Marangoni (RBM) system in presence of
controls for the velocity and the temperature on parts of the
boundary. We analyze the existence, uniqueness and regularity of
weak solutions for the stationary RBM system in polyhedral domains
of $\mathbb{R}^3,$ and then, we prove the existence of the optimal
solution. By using the Theorem of Lagrange multipliers, we derive an
optimality system. We also give a second-order sufficient optimality
condition and we establish a result of uniqueness of the optimal
solution.\bigskip

\textbf{AMS subject classification: }35Q55, 35A01, 35A02, 35C06.

\medskip\textbf{Keywords:} Rayleigh-B\'enard-Marangoni System,
Boundary control problems, Optimality conditions.
\end{abstract}

\pagestyle{myheadings} \markright{On the Rayleigh-B\'enard-Marangoni system}
\section{Introduction} \label{introd}

\hspace{0.5cm}Fluid movement by temperature gradients, also called thermal
convection, is an important process in nature. Its main applications
appear in industry, as for instance, in the growth of semiconductor
crystals, but also,  thermal convection is the basis for the
interpretation of several natural phenomena such as the movement of
the earth's plates, the solar activity, large scale circulations of
the oceans, movement in the atmosphere, among others. A model of
particular interest consists of a horizontal layer of a fluid in a
container heated uniformly from below, with the bottom surface and
the lateral walls rigid and the upper surface open to the
atmosphere. Due to heating, the fluid in the bottom surface expands
and it becomes lighter than the fluid in the upper surface, so that,
by effect of the buoyancy, the liquid is potentially unstable.
Because of the instability, the fluid tends to redistribute.
However, this natural tendency will be controlled by its own
viscosity. On the other hand, the upper surface, which is free to
the atmosphere, experiences changes in its surface tension as a
result of the temperature gradients in the surface. Then, it is
expected that
the temperature gradient exceeds a critical value, above which the instability can manifest.\\

The first experiments to demonstrate the beginning of
Thermal instability in fluids were developed by
Henri B\'enard in 1900 (see \cite{Benard}). In his experiments, B\'enard considered
a very thin layer of liquid, around $1\: mm$  of depth,
in a metal plate maintained at a constant temperature. The upper surface was usually free and it was in contact with the air, which
was at a lower temperature. B\'enard experimented with a variety of
liquids with different physical characteristics, mainly interested
in the effect of viscosity on the convection, using liquids of
high viscosity like wax whale melted and paraffin. In all
these cases, B\'enard found that when the temperature of the plate
gradually increased, at a certain moment, the layer lost
stability and formed patterns of hexagonal cells, all alike
and correctly aligned.\\


A first theoretical interpretation of thermal convection was provided
by Lord Rayleigh in 1916 (see \cite{Rayleigh}), whose analysis was inspired
by B\'enard's experiment. Rayleigh assumed that the fluid was confined between two horizontal thermally conductive plates and the fluid
was being heated from below. Rayleigh considered that
the effect of buoyancy is the only one responsible of the beginning of the
instability, and theoretically, the results coincided with the
reported by B\'enard, giving the impression that his model was correct. However, it is known now that Rayleigh's theory is not
adequate for explaining the convective mechanism observed by B\'enard.
In fact, in B\'enard's experiments, the free surface was
in contact with the atmosphere which generates a surface tension,
and Rayleigh, using a plate in the upper surface, eliminated the surface tension's effects.\\

It should be noted that the surface tension is not constant and
it may depend on the temperature or contaminants in the
surface. This dependence is called {\it capillarity} or
{\it Marangoni effect} \cite{lappa,Marangoni}. The importance
of the Marangoni effect in B\'enard's experiments was established by Block in
\cite{Block} from an experimental point of view, and by
Pearson \cite{Pearson} from a theoretical point of view. Now is recognized that the Marangoni effect is
the main cause of instability and convection in the
B\'enard original experiments.\\

For the foregoing reasons, we consider the physical situation of a horizontal layer of a fluid in a cubic container of height $d$ ($x_{3}$-coordinate), of length $L_1$ ($x_{2}$-coordinate) and width $l_1$ ($x_{1}$-coordinate). The bottom surface of the container and the lateral walls are rigid and the upper surface is open to the atmosphere. In order to describe the system, we use the Oberbeck-Boussinesq approximation \cite{Chandra}, which assumes that the thermodynamical coefficients are constant, except in the case of the density in the buoyancy term, which is considered as being $\rho_{0}\left[1-\alpha\left(\theta-\theta_{a}\right)\right]$. Here $\rho_{0}$ is the mean density, $\theta_{a}$ is the temperature of the environment and $\alpha$ is the thermal expansion coefficient, which is positive for most liquids. Moreover, we assume that the surface tension is a function of the temperature, and it is approximated by $\sigma=\sigma_{0}-\gamma\left(\theta-\theta_{a}\right)$. Here, $\sigma_{0}$ is the surface tension at temperature $\theta_{a}$, and $\gamma$ is the ratio of change of the surface tension with the temperature ($\gamma$ is positive for the more commonly used liquids). Also, the
free surface is presumed not to be distorted, that is, the vertical component of the velocity in the free surface
always will be zero. Then, we consider that the domain, in which the fluid is confined, is given by $\Omega=\left(0,\,
l_{1}\right)\times\left(0,\, L_{1}\right)\times\left(0,\, d\right)$. However, the analysis developed in this paper allows us to consider a domain $\Omega$ with more general geometries, specifically, we can consider $\Omega=\hat{\Omega}\times (0,d)$, being $\hat{\Omega}$ a Lipschitz bounded domain of $\mathbb{R}^{2}.$\\


In stationary regime, the RBM system
is given by the following coupling between the Navier-Stokes equations and heat equation:
\begin{equation}\label{modnoest1}
\left\{
\begin{array}
[c]{rcl}%
\rho_{0}\left(\textbf{u}\cdot\nabla\right) \textbf{u}+\nabla p-\mu\Delta\textbf{u}=\rho_{0}\left[1-\alpha\left(\theta-\theta_{a}\right)\right]\vec{g}\; \text{ in }\; \Omega,\\
\rho_{0}\hat{C}_{p}\left(\textbf{u}\cdot\nabla \right)\theta=K\Delta \theta \;\text{ in }\; \Omega,\\
\text{div }\textbf{u}=0  \;\text{ in }\; \Omega,
\end{array}
\right.
\end{equation}
where the unknowns are
${\textbf{u}(\textbf{x})=(u_{1}(\textbf{x}),u_{2}(\textbf{x}),u_{3}(\textbf{x}))}\in
\mathbb{R}^{3}$, $\theta(\textbf{x})\in \mathbb{R}$ and
$p(\textbf{x})\in \mathbb{R},$  which represent the velocity field,
the temperature and the hydrostatic pressure of the fluid,
respectively, at a point $\textbf{x}\in \Omega.$ The constant
$\hat{C}_{p}$ is the heat capacity per unit mass of the fluid, $\mu$
its viscosity, $K$ its thermal conductivity and the field $\vec{g}$
is the acceleration due to gravity.\\

In order to express the system in adimensional form, we make the following
changes of variables:
\begin{eqnarray*}
x'_{1}=\frac{x_{1}}{d},\ x'_{2}=\frac{x_{2}}{d},\ x'_{3}=\frac{x_{3}}{d},\ u'_{1}=\frac{du_{1}}{\kappa},\ u'_{2}=\frac{du_{2}}{\kappa},\
u'_{3}=\frac{du_{3}}{\kappa},\ \theta'=\frac{\theta-\theta_{a}}{\theta_{u}},\ p'=\frac{d^{2}p}{\rho_{0}\nu\kappa},
\end{eqnarray*}
where $\theta_{u}=\theta_{c}-\theta_{a}$ with $\theta_{c}$ the
temperature at the bottom plate, ${\displaystyle
\kappa=\frac{K}{\rho_{0}\hat{C}_{p}}}$ and ${\displaystyle
\nu=\frac{\mu}{\rho_{0}}}$. Thus, removing the primes to simplify
the notation, from (\ref{modnoest1}) we get
\begin{equation}\label{modest1}
\left\{
\begin{array}
[c]{rcl}%
(\textbf{u}\cdot\nabla) \textbf{u}=Pr\left[\left(b+R\theta\right)\vec{e}_{3}-\nabla p+\Delta\textbf{u}\right]\ \text{ in }\Omega,\\
(\textbf{u}\cdot\nabla) \theta=\Delta \theta\ \text{ in } \Omega,\\
 \text{div }\textbf{u}=0\ \text{ in } \Omega,
\end{array}
\right.
\end{equation}
with $\Omega=\left(0,\, l\right)\times\left(0,\,
L\right)\times\left(0,\,1\right)$, where $l=\frac{l_{1}}{d}$ and
$L=\frac{L_{1}}{d}$. Moreover, ${\displaystyle
Pr=\frac{\nu}{\kappa}}$, ${\displaystyle R=\frac{\vert
\vec{g}\vert\alpha \theta_{u}d^{3}}{\kappa\nu}}$ and ${\displaystyle
b=-\frac{\vert \vec{g}\vert d^{3}}{\kappa\nu}}$. The number $R$ is
known as the Rayleigh number and it measures the effect of buoyancy; $Pr$ is known as the Prandtl number and it represents the
relationship between the speed of diffusion of momentum and the rate
of diffusion of
heat in the fluid, and $\vec{e}_{3}$ is the unit vector in the third direction, that is, $\vec{e}_{3}=(0,0,1)$. \\



Let us denote by $\partial\Omega$ the boundary of $\Omega$ and let $\Gamma_{1}:=\partial\Omega\cap\left\{ x_{3}=1\right\}$ and $\Gamma_{0}:=\partial\Omega\setminus \Gamma_{1}$. Then, the following boundary conditions are imposed:
\begin{equation}\label{condfro3}
\left\{
\begin{array}
[c]{rcl}%
\mathbf{u}={\boldsymbol 0} \ \ \text{on}\,\ \Gamma_{0}, \qquad u_{3}=0 \ \ \text{on}\,\ \Gamma_1,\\
\frac{\partial u_{i}}{\partial\textbf{n}}+M\frac{\partial
\theta}{\partial x_{i}}=0,\,\,\, i=1,2, \ \ \text{on}\,\
\Gamma_1,\\
\frac{\partial \theta}{\partial\textbf{n}}+B\theta=0 \ \
\text{on}\,\ \Gamma_1,\qquad
\theta =\theta_{c}  \ \ \text{on}\,\ \{ x_{3}=0\},\\
\frac{\partial \theta}{\partial\textbf{n}} =0 \ \ \text{on}\,\
\Gamma_{0}\setminus\left\{ x_{3}=0\right\},
\end{array}
\right.
\end{equation}
where $\textbf{n}=(n_{1},n_{2},n_{3})$ is the normal vector pointing outward, $M=\frac{\gamma \theta_{u}d}{\rho_{0}\nu k}$, $B=\frac{hd}{K}>0$, and $h$ is the heat exchange coefficient of the surface with the atmosphere.\\


The boundary conditions for the velocity in (\ref{condfro3})$_{1}$ are no slip conditions on the rigid and free surface. The condition (\ref{condfro3})$_{2}$ takes into account the Marangoni effect, which represents the mass transfer at an interface between two fluids due to a surface tension gradient.
Conditions (\ref{condfro3})$_{3}$ say that on the lateral surfaces there is not heat flow (adiabatics), that
on the free surface is allowed the heat flow, and  that the bottom surface is maintained at temperature $\theta_{c}$ (isothermal).\\

From the point of view of the existence of solution of RBM problem,
recently in \cite{Pardo} was discussed a bifurcation problem in
which, considering either the Rayleigh number or the Prandtl number
as bifurcation parameters. By using the local bifurcation theory
due to Crandall and Rabinowitz \cite{Crandall}, the authors showed the
existence of stationary solutions to the problem
(\ref{modest1})-(\ref{condfro3}), which bifurcate from a basic state
of heat conduction. For basic state we refer to the exact solution
of the problem (\ref{modest1})-(\ref{condfro3}), which is given by
\begin{eqnarray}
\textbf{u}_{b}=\textbf{0},\;\;\; \theta_{b}=\theta_{c}-\frac{\theta_{c}B}{1+B}x_{3}\ \text{ and }\ p_{b}=p_{1}x_{3}+p_{2}x_{3}^{2}. \label{estadobasico}
\end{eqnarray}

Previously to \cite{Pardo}, in \cite{Dauby,Hoyas,Hoyas2,Arne} were
obtained numerical results on the existence of solutions that
bifurcate of the basic stationary states, instability and patter
formation problems, as well as a validation of initial and boundary
conditions. However, from a theoretical point of view, no more
results are available in the literature. The main difficulty in the
treating the RBM problem (\ref{modest1})-(\ref{condfro3}), beyond
the coupling between Navier-Stokes system and heat equation, are the
crossed boundary conditions (\ref{condfro3}) involving tangential
derivatives of the temperature and normal derivatives of the
velocity field; in fact, in order to define tangential derivatives at
the boundary, intended in the trace sense, it is necessary
regularity of the weak solutions; this fact involves the geometry of
the domain in order to use elliptic regularity in Sobolev spaces
$W^{k,p}$ for the Laplace and Stokes equations, in polyhedral
domains (see \cite{Dauge,Dauge2,Grisvard} for elliptic regularity
results associated to Laplace equation and
\cite{Dauge1,Kellogg,Majda,Majda2} for elliptic regularity  results
related to the Stokes equation; see also \cite{Rodriguez} for
related problems associated to the
Boussinesq system).\\

From the point of view of optimal control theory, unlike to the Navier-Stokes stationary
equations, results on boundary control problems in which the cost
functional is subject to state equations governed by RBM system are
not known. Some optimal control results associated with the
Boussinesq equations are known, see for example
\cite{Abergel,Alekseev,Alekseev2,Tereshko,Lee2,Lee3,Lee4}; however, the mathematical formulation and the boundary control problem for Boussinesq equations differ from the RBM model in the type of boundary conditions, principally the condition (\ref{condfro3})$_2$ which takes into account the Marangoni effect, as well as the dimension of the domain. More exactly, in \cite{Abergel} the authors studied an optimal control problem minimizing the turbulence caused by the heat convection. The states are given by 3D-Boussinesq equations with Neumann control on the temperature. In \cite{Alekseev,Alekseev2,Tereshko}, the authors analyzed optimal control problems for the 3D-Boussinesq equations with Neumann and Dirichlet boundary controls. The results of  \cite{Pardo} do not include control theory. In \cite{Lee2,Lee4} the authors analyzed optimal control problems associated to the 2D-Boussinesq equations. The controls considered may be of either the distributed or the Neumann type. In \cite{Lee4} the author considers the approximation, by finite element methods, of the optimality system and derive optimal error estimates. The convergence of a simple gradient method is proved and some numerical results are given. In \cite{Lee3} the authors studied an optimal control problem for the Boussinesq equations, also in 2D, with Dirichlet control on the temperature. A gradient method for the solution of the discrete optimal control problem is presented and analyzed. Finally, the results of some computational experiments are presented. In the previous references, sufficient optimality conditions were not analyzed. Optimal control problems for the Navier-Stokes equations through the action of Dirichlet boundary conditions have been analyzed (see for instance in \cite{delosReyes,Gunzburger,Heinkenschloss}). In some cases, numerical results, either by solving the optimality system or by optimization methods, have been obtained.\\

In this paper, we analyze an optimal control problem for which the
velocity and the temperature of the fluid are controlled by
boundary data along portions of the boundary; the cost
functional is given by a sum of functionals which measure, in the
$L^2$-norm, the difference between the velocity (respectively, the
temperature) and a given prescribed velocity (respectively, a
prescribed temperature). The cost functional also measures the vorticity of the flow. The fluid motion is constrained to satisfy
the stationary system of RBM. The exact mathematical formulation
will be given in Section 2. We will prove the solvability of the
optimal control problem and, by using the Lagrange multiplier
method, we state the first-order optimality conditions; we derive an
optimality system and give a second-order sufficient optimality
condition. Moreover, we also study the uniqueness of the optimal solution. Beside to the solvability of the optimal control problem,
we first prove the existence of weak solutions for RBM system with
nonhomogeneous boundary data, as well as the uniqueness and
regularity properties. It is worthwhile to remark that the proof of
existence and regularity of weak solutions for RBM system
 is not a simple generalization of
the similar ones to deal with Navier-Stokes or related models in
fluid mechanics \cite{Elder2}. If fact, we are considering non
homogeneous crossed boundary conditions involving tangential
derivatives of the temperature and normal derivatives of the
velocity field, which permit to deal with pointwise constrained
boundary optimal control of Dirichlet and Neumann type. In
\cite{Pardo}, the boundary conditions are homogeneous, and thus,
boundary control problems are not considered. The non homogeneous
boundary conditions are assumed in spaces of kind
$H^{1/2}_{00}(\Gamma), \Gamma\subseteq \partial\Omega,$ which are
natural from the variational point of view; these space, which are
used as control spaces, are closed subspaces of  $H^{1/2}(\Gamma)$
and satisfy the embeddings $H^{1}(\Gamma)\hookrightarrow
H^{1/2}_{00}(\Gamma)\hookrightarrow L^2(\Gamma).$ On the other hand,
to define tangential derivatives at the boundary, intended in the
trace sense, it is necessary to analyze the regularity of the weak
solutions, in particular, it is required the $H^2$-regularity for
the temperature (cf. (\ref{condfro3})$_{2}$ and Lemma \ref{lempardo}
below). However, due the geometry of the domain, the regularity of
the weak solutions, when non homogeneous boundary conditions are
assumed, is a nontrivial subject. For that, we adapt regularity
results for the Laplace equation with Dirichlet-Neumann boundary
homogeneous conditions in corner domains of
\cite{Dauge,Dauge2,Grisvard}, and some ideas of \cite{Ziane} to
treat the Robin and Neumann nonhomogeneous boundary conditions. On
the other hand, from the point of view of the control theory, as far
as we known, our results are the first ones dealing with with
pointwise constrained boundary optimal control of the RBM system, by
using spaces $H_{00}^{1/2}(\Gamma)$ as the control spaces. We give
necessary and sufficient optimality conditions which are a
significant advance in the analysis of controlling these equations.
In order to obtain necessary optimality conditions we use an approach which differs from
the other ones in the case of 3D-Boussinesq equations (cf.
\cite{Abergel,Alekseev}). In fact, in order to derive the optimality
conditions, in \cite{Alekseev} the author used a theorem of Ioffe
and Tikhomorov and also he assumed a property, called \textit{Property
C}, whereas that in
 \cite{Abergel}, the authors used a penalization method because in that case the relation control-state is multivalued. It is worthwhile to remark that in the previous references related to convection problems, sufficient optimality conditions were not analyzed. Finally, from the point of view of numerical results, since the analysis of the control problem yields variational inequalities as optimality conditions, the numerical analysis offers new challenges, for instance, the applicability of the semi-smooth Newton method in order to obtain a numerical solution (cf. \cite{delosReyes} for numerical results in the context of Navier-Stokes model).\\

The outline of this paper is as follows: In Section 2, we give a
precise definition of the optimal control problem to be studied and,
in Section 3, we prove the existence and uniqueness of weak
solutions, as well as we show regularity properties. In Section 4,
we prove the existence of the optimal solution. In section 5, we
obtain the first-order optimality conditions, and by using the
Lagrange multipliers method we derive an optimality system. In
Section 6, we give a second-order sufficient optimality condition.
In Section 7, we establish a result of uniqueness of the optimal
solution.

\section{Statement of the boundary control problem}
\hspace{0.4cm}Throughout this paper we use the Sobolev space $H^m(\Omega),$ and
$L^p(\Omega),$ $1\leq p\leq \infty,$ with the usual notations for
norms $\Vert \cdot\Vert_{H^{m}(\Omega)}$ and $\Vert\cdot
\Vert_{L^{p}(\Omega)}$ respectively. If $H$ is a Hilbert space we
denote its inner product by $( \cdot,\cdot )_{H}$; in particular,
the $L^2(\Omega)$-inner product will be represented by
$(\cdot,\cdot)_{L^{2}(\Omega)}$. If $X$ is a general Banach space,
its topological dual will be denoted by $X'$ and the duality product
by $\langle \cdot,\cdot\rangle_{X',X}$. Corresponding Sobolev spaces
of vector valued functions will be denoted by ${\bf H}^1(\Omega),$
${\bf H}^2(\Omega),$ ${\bf L}^2(\Omega)$, and so on. If $\Gamma$ is
a connected subset of the boundary $\partial \Omega$, we define the
trace space
\begin{equation*}
H^{1/2}_{00}(\Gamma):=\{v\in L^{2}(\Gamma): \mbox{ there exists } g\in H^{\frac{1}{2}}(\partial\Omega), \ \left. g\right|_{\Gamma}=v, \ \left. g\right|_{\partial\Omega \setminus \Gamma}=0 \},
\end{equation*}
which is a closed subspace of $H^{\frac{1}{2}}(\Gamma)$ (see \cite{LionsDau}, p. 397), where $H^{\frac{1}{2}}(\Gamma)$ is the restriction of the elements of $H^{\frac{1}{2}}(\partial\Omega)$ to $\Gamma$.  We also will use the following Banach spaces
\begin{eqnarray*}
&&\textbf{X}:=\{\textbf{u}=(u_{1},u_{2},u_{3})\in \textbf{H}^1 (\Omega):\text { div }\textbf{u}=0,\, u_{3}=0\,\,\text { on }\,\,\Gamma_{1}\},\\
&&\textbf{X}_{0}:=\{\textbf{u}\in \textbf{H}^1 (\Omega):\text { div }\textbf{u}=0,\textbf{u}=0\,\,\text { on }\,\,\Gamma_{0},\,\, u_{3}=0\,\,\text { on }\,\,\Gamma_{1} \},\\
&&Y:=\{S\in H^1(\Omega): S=0\,\,\text { on }\,\,\{x_{3}=0\}\},\\
&&\widetilde{\textbf{X}}:=\left\{\textbf{u}\in \textbf{X}: \textbf{u}\cdot \textbf{n}=0 \text{ on } \Gamma_{0}\setminus\left\{x_{3}=0\right\}\right\},\\
&& \widetilde{\mathbf{H}}^{1/2}_{00}(\Gamma):=\left\{\mathbf{v}\in \mathbf{H}^{1/2}_{00}(\Gamma): \int_{\Gamma} \mathbf{v}\cdot \textbf{n} = 0, \ \mathbf{v}\cdot \textbf{n}=0 \ \mbox{on} \ \Gamma\setminus \{x_{3}=0\} \right\}.
\end{eqnarray*}
Moreover, if $\Gamma$ is an arbitrary subset of $\partial\Omega$, we
use the notation $\langle f,g \rangle_{\Gamma}$ to represent the
integral $\int_\Gamma fg\; dS$. In the paper, the letter $C$ will
denote diverse positive constants which may change from line to line
or even within a same line.

In order to establish the boundary control problem, we consider the
following stationary model related to
(\ref{modest1})-(\ref{condfro3}) with nonhomogeneous boundary data:

\begin{equation}\label{model001}
\left\{
\begin{array}
[c]{rcl}%
(\textbf{u}\cdot\nabla) \textbf{u}=Pr\left[\left(b+R\theta\right)\vec{e}_{3}-\nabla p+\Delta\textbf{u}\right]\ \ \ \text{in}\,\ \Omega,\\
(\textbf{u}\cdot\nabla) \theta=\Delta \theta\ \ \ \text{in}\,\  \Omega,\\
\text{div}\:\textbf{u}=0\ \ \ \text{in}\,\ \Omega,\\
\mathbf{u}=\mathbf{g} \ \ \text{on}\,\ \Gamma^1_{0}, \qquad \mathbf{u}=\mathbf{u}^0\ \  \text{on} \,\ \Gamma^2_{0},\\
u_{3}=0 \ \ \text{on}\,\ \Gamma_1,\\
\frac{\partial u_{i}}{\partial\textbf{n}}+M\frac{\partial
\theta}{\partial x_{i}}=0,\,\,\, i=1,2, \ \ \text{on}\,\
\Gamma_1,\\
\frac{\partial \theta}{\partial\textbf{n}}+B\theta=0 \ \
\text{on}\,\ \Gamma_1,\qquad
\theta =\phi_2  \ \ \text{on}\,\ \{ x_{3}=0\},\\
\frac{\partial \theta}{\partial\textbf{n}} =\phi_1 \ \ \text{on}\,\
\Gamma_{0}\setminus\left\{ x_{3}=0\right\},
\end{array}
\right.
\end{equation}
where $\Gamma_{0}=\Gamma^1_{0}\cup \Gamma^2_{0}$ with
$\Gamma^1_{0}\cap \Gamma^2_{0}=\emptyset$, the vector
$\mathbf{u}^{0}=(u^{0}_{1},u^{0}_{2},u^{0}_{3})\in
\widetilde{\mathbf{H}}^{1/2}_{00}(\Gamma_{0}^{2})$ is a Dirichlet
condition for the velocity $\mathbf{u}$ on $\Gamma_{0}^{2}$; the
field $\mathbf{g}=(g_{1},g_{2},g_{3})\in
\widetilde{\mathbf{H}}^{1/2}_{00}(\Gamma_{0}^{1})$ is given and
denotes a control for $\mathbf{u}$ on $\Gamma_{0}^{1}$;
additionally, $\phi_{1}\in
H^{\frac{1}{2}}(\Gamma_{0}\setminus\left\{ x_{3}=0\right\})$ is a
given function which denotes a Neumman control to temperature
$\theta$ on $\Gamma_{0}\setminus\left\{ x_{3}=0\right\}$, and
$\phi_{2}\in H^{1/2}_{00}(\{ x_{3}=0\})$ is a Dirichlet control to
temperature $\theta$ on $\left\{ x_{3}=0\right\}$.

Suppose that $\mathcal{\textbf{U}}_{1}\subset
\widetilde{\textbf{H}}^{1/2}_{00}(\Gamma_{0}^{1})$,
$\mathcal{U}_{2}\subset H^{\frac{1}{2}}(\Gamma_{0}\setminus\left\{
x_{3}=0\right\})$ and $\mathcal{U}_{3}\subset H^{1/2}_{00}(\{
x_{3}=0\})$ are nonempty sets, and $\gamma_{i}$, $i=1,...,6$, are
constants. Assume one of the following conditions:
    \begin{enumerate}
    \item[(i)]$\gamma_{i}\geq 0\text{ for } i=1,2,...,6$, with $\gamma_1,\gamma_2,\gamma_3$ not simultaneously zero, and $\mathcal{\textbf{U}}_{1}$, $\mathcal{U}_{2}$ and $\mathcal{U}_{3}$ are bounded closed convex sets;
    \item[(ii)] $\gamma_{i}\geq 0 \text{ for } i=1,2,3$, with $\gamma_1,\gamma_2,\gamma_3$ not simultaneously zero, $\gamma_{i}>0\text{ for } i=4,5,6$ and $\mathcal{\textbf{U}}_{1}$, $\mathcal{U}_{2}$ and $\mathcal{U}_{3}$ are closed convex sets.
 \end{enumerate}

We study the following constrained minimization problem on weak
solutions to problem (\ref{model001}), for fixed data
$\mathbf{u}^{0}\in \widetilde{\mathbf{H}}^{1/2}_{00}(\Gamma_{0}^{2})$:
 \begin{eqnarray}
\left\{
\begin{array}[c]{l}
\text{Find } [\textbf{u},\theta,\textbf{g},\phi_{1},\phi_{2}]\in \widetilde{\textbf{X}} \times H^1(\Omega) \times\mathcal{\textbf{U}}_{1} \times\mathcal{U}_{2} \times\mathcal{U}_{3} \text{ such that the functional}\vspace{0,4 cm}\\
\CJ
[\textbf{u},\theta,\textbf{g},\phi_{1},\phi_{2}]=\frac{\gamma_1}{2}\|\text{rot
}
\mathbf{u}\|_{L^2(\Omega)}^2+\frac{\gamma_2}{2}\|\mathbf{u}-\mathbf{u}_d\|_{L^2(\Omega)}^2
+\frac{\gamma_3}{2}\|\theta-\theta_d\|_{L^2(\Omega)}^2 +\frac{\gamma_{4}}{2}\|\mathbf{g}\|^2_{H^{\frac{1}{2}}(\Gamma_{0}^{1})}\\
\hspace{3 cm}
+\frac{\gamma_{5}}{2}\|\phi_1\|^2_{H^{\frac{1}{2}}({\Gamma_{0}\setminus\left\{
x_{3}=0\right\} })}
+\frac{\gamma_6}{2}\|\phi_2\|^2_{H^{\frac{1}{2}}(\{x_3=0\})}, \vspace{0,4 cm} \\
\text{is minimized subject to the constraint that }
[\textbf{u},\theta] \text{ is a weak solution of (\ref{model001})}. \text{ Here } \mathbf{u}_{d}\in \mathbf{L}^2(\Omega)\\
\text{and } \theta_{d}\in L^2(\Omega) \text{ are given.}
\end{array}
\right. \label{eq:funcional}
\end{eqnarray}

\vspace{0,3 cm}

\section{Well-posedness and Regularity of Solutions for (\ref{model001})}
\hspace{0.4cm}In this section we analyze the existence, uniqueness and regularity
of weak solution for system (\ref{model001}), which, as was said in
Section \ref{introd}, it is not a simple generalization of the
similar ones to deal with Navier-Stokes or related models in fluid
mechanics \cite{Elder2}.

\subsection{Weak Solutions for (\ref{model001})}
\hspace{0.4cm}We introduce the bilinear and trilinear forms $a:\textbf{X}\times
\textbf{X}\rightarrow\mathbb{R}$, $c:\textbf{X}\times
\textbf{X}\times \textbf{X}\rightarrow\mathbb{R}$,
$a_{1}:H^{1}(\Omega)\times H^{1}(\Omega)\rightarrow\mathbb{R}$,
$b_{1}:H^{1}(\Omega)\times \textbf{X}\rightarrow\mathbb{R}$ and
$c_{1}:\textbf{X}\times H^{1}(\Omega)\times
H^{1}(\Omega)\rightarrow\mathbb{R}$, for the velocity and
temperature:

\begin{center}
${\displaystyle a(\textbf{u},\textbf{v})=\int_\Omega
\nabla\textbf{u}\cdot\nabla\textbf{v}\: d\Omega}
 ,\qquad
  {\displaystyle c(\textbf{u},\textbf{v},\textbf{z})=\int_\Omega [(\textbf{u}\cdot\nabla)\textbf{v}]\cdot\textbf{z}\: d\Omega},$

${\displaystyle a_{1}(\theta,W)=\int_\Omega \nabla\theta\cdot\nabla
W\: d\Omega},\qquad {\displaystyle
b_{1}(\theta,\textbf{v})=\int_\Omega
\nabla\theta\cdot\frac{\partial\textbf{v}}{\partial x_{3}}\:
d\Omega}$,

${\displaystyle c_{1}(\textbf{u},\theta,W)=\int_\Omega
[(\textbf{u}\cdot\nabla)\theta] W\: d\Omega}$.

\end{center}

\begin{lemma} \label{lemalt}
The following relations hold for $c$ and $c_1$:
\begin{equation}
c(\mathbf{u},\mathbf{v},\mathbf{z})=-c(\mathbf{u},\mathbf{z},\mathbf{v}),
\:\: c(\mathbf{u},\mathbf{v},\mathbf{v})=0,\:\:\:
\forall\mathbf{u}\in \mathbf{X}_{0},\:\:\forall
\mathbf{v},\mathbf{z}\in \mathbf{H}^1(\Omega),\label{eq:altc}
\end{equation}
\begin{equation}
c_{1}(\mathbf{u},\theta,W)=-c_{1}(\mathbf{u},W,\theta), \:\:
c_{1}(\mathbf{u},\theta,\theta)=0,\:\:\: \forall \mathbf{u}\in
\mathbf{X}_{0},\:\: \forall\theta,W\in H^1(\Omega).\label{eq:altc1}
\end{equation}
\end{lemma}

\begin{proof}
Considering that $\textbf{u}\in \textbf{X}_{0}$, i.e. $\textbf{u}=0$
on $\Gamma_{0}$, $u_{3}=0$ on $\Gamma_{1}$ and $\text{ div
}\textbf{u}=0$, and the normal vector $\textbf{n}$ on $\Gamma_{1}$
is $\textbf{n}=(0,0,1)$, we obtain that $\textbf{u}\cdot
\textbf{n}=0$ on $\partial \Omega$. Therefore, the proof follows as
in Lemma 2.2 in \cite{Girault}, p. 285.
\end{proof}

\begin{lemma} (\cite{Pardo}) \label{lempardo}
Assume that $\Omega$ is a bounded domain with boundary $\partial \Omega$ Lipschitz, and $\partial \Omega=\Gamma_{0}\cup \Gamma_{1}$ with $\Gamma_{1}\subseteq\{x_{3}=C\}$, $C$ a constant. If $\theta\in H^2(\Omega)$ then
\begin{center}
${\displaystyle \int_{\Gamma_{1}}\frac{\partial\theta}{\partial
x_{1}} v_{1}+\frac{\partial\theta}{\partial x_{2}} v_{2}\:
dS=\int_\Omega\nabla\theta\cdot\frac{\partial\mathbf{v}}{\partial
x_{3}}\: d\Omega},\,\,\, \forall\mathbf{v}\in \mathbf{X}_{0}$.
\end{center}
\end{lemma}
Motivated by the formula of integration by parts and using Lemma \ref{lempardo}, we obtain the following weak formulation of System
(\ref{model001}).
\begin{definition}\label{defw1}
A pair $[\mathbf{u},\theta] \in \mathbf{X} \times H^1(\Omega)$ is said a weak solution of (\ref{model001}) if
\begin{equation}
Pr\, a(\mathbf{u},\mathbf{v})+Pr\, M\, b_{1}(\theta,\mathbf{v})+c(\mathbf{u},\mathbf{u},\mathbf{v})=\left\langle f(\theta),\mathbf{v}\right\rangle ,\;\:\forall\mathbf{v}\in \mathbf{X}_{0},\label{eq:020}
\end{equation}
\begin{equation}
c_{1}(\mathbf{u},\theta,W)+a_{1}(\theta,W)+\left\langle B\theta,W\right\rangle _{\Gamma_{1}}=\left\langle \phi_{1},W\right\rangle _{\Gamma_{0}\backslash\{x_{3}=0\}},\;\:\forall W\in Y, \label{eq:021}
\end{equation}
\begin{equation}
\mathbf{u}=\mathbf{g}\: \text{ on } \:\Gamma_{0}^{1}, \ \
\mathbf{u}=\mathbf{u}^{0}\: \text{ on } \:\Gamma_{0}^{2} \ \  \text{
and }\quad\theta=\phi_{2}\: \text{ on } \:\{x_{3}=0\},
\label{eq:022}
\end{equation}
where $\displaystyle\left\langle f(\theta),\mathbf{v}\right\rangle
=\int_\Omega Pr\left(b+R\theta\right)\textbf{e}_{3}\cdot\mathbf{v}\:
d\Omega$.
\end{definition}

\subsection{Existence of Weak Solutions}\label{solweak}
\hspace{0.4cm}In order to prove the existence of a solution to the problem
(\ref{eq:020})-(\ref{eq:022}) we reduce the problem to an auxiliary
problem with homogeneous conditions for the velocity $\textbf{u}$ on
$\Gamma_{0}$ and the temperature $\theta$ on $\{x_{3}=0\}$. For that,
we will use the Hopf Lemma (see Lemma 4.2 of \cite{Galdi}, p. 28). Notice that if $\mathbf{u}^{0}\in
\widetilde{\mathbf{H}}^{1/2}_{00}(\Gamma_{0}^{2})$ and
$\mathbf{g}\in \widetilde{\mathbf{H}}^{1/2}_{00}(\Gamma_{0}^{1})$,
then there exist $\widetilde{\mathbf{u}}^{0}\in
\mathbf{H}^{\frac{1}{2}}(\partial\Omega)$ and
$\widetilde{\mathbf{g}}\in \mathbf{H}^{\frac{1}{2}}(\partial\Omega)$
such that
\begin{equation*}
\widetilde{\mathbf{u}}^{0} \mid_{\Gamma^2_{0}}  = \mathbf{u}^{0}, \ \ \widetilde{\mathbf{u}}^{0}\mid_{\partial\Omega\setminus\Gamma^2_{0}} = 0, \ \ \int_{\Gamma_{0}^{2}} \mathbf{u}^{0} \cdot \mathbf{n}=0,\ \ \mathbf{u}^{0} \cdot \mathbf{n}=0 \ \mbox{ on } \Gamma_{0}^{2}\setminus \{x_{3}=0\},
\end{equation*}
\begin{equation*}
\widetilde{\mathbf{g}}\mid_{\Gamma^1_{0}}  = \mathbf{g}, \ \ \widetilde{\mathbf{g}}\mid_{\partial\Omega\setminus\Gamma^1_{0}} = 0, \ \ \int_{\Gamma_{0}^{1}} \mathbf{g} \cdot \mathbf{n}=0,\ \ \mathbf{g} \cdot \mathbf{n}=0 \ \mbox{ on } \Gamma_{0}^{1}\setminus \{x_{3}=0\}.
\end{equation*}
Thus, the function $\widetilde{\mathbf{u}}^{0} + \widetilde{\mathbf{g}}\in \mathbf{H}^{\frac{1}{2}}(\partial\Omega)$ and
\begin{eqnarray*}
\int_{\partial\Omega} (\widetilde{\mathbf{u}}^{0} + \widetilde{\mathbf{g}})\cdot \mathbf{n}= \int_{\Gamma_{0}^{1}} \widetilde{\mathbf{g}} \cdot \mathbf{n} +\int_{\Gamma_{0}^{2}} \widetilde{\mathbf{u}}^{0} \cdot \mathbf{n}= \int_{\Gamma_{0}^{1}} \mathbf{g} \cdot \mathbf{n}+\int_{\Gamma_{0}^{2}} \mathbf{u}^{0} \cdot \mathbf{n}=0.
\end{eqnarray*}
Therefore, by the Hopf Lemma, there exists a function $\textbf{u}_{\varepsilon}=(u_{\varepsilon_{1}},u_{\varepsilon_{2}},u_{\varepsilon_{3}})$ which satisfies the conditions
\begin{equation*}
\textbf{u}_{\varepsilon}\in \textbf{H}^{1}(\Omega),\qquad \text{ div }\textbf{u}_{\varepsilon}=0\; \text { in  } \Omega,\qquad \textbf{u}_{\varepsilon}=\widetilde{\mathbf{u}}^{0} + \widetilde{\mathbf{g}}\; \text { on  } \partial\Omega,
\end{equation*}
\begin{equation*}\label{exthop}
\left\Vert \mathbf{u}_{\varepsilon}\right\Vert_{H^1(\Omega)}\leq C
\left(\Vert \mathbf{u}^{0}
\Vert_{H^{\frac{1}{2}}(\Gamma_{0}^{2})}+\Vert \mathbf{g}
\Vert_{H^{\frac{1}{2}}(\Gamma_{0}^{1})}\right), \qquad
\left|c(\textbf{v},\textbf{u}_{\varepsilon},\textbf{v})\right|\leq\varepsilon\left\Vert
\textbf{v}\right\Vert _{H^{1}(\Omega)}^{2},\;\;\forall\textbf{v}\in
\mathbf{X}_{0},
\end{equation*}
where $C=C(n,\Omega)$ and
$\varepsilon>0$ is a real number arbitrarily small. Notice that
$\mathbf{u}_{\varepsilon}\mid_{\Gamma^2_{0}}  = \mathbf{u}^{0}$ and
$\mathbf{u}_{\varepsilon}\mid_{\Gamma^1_{0}}  = \mathbf{g}$.
Moreover, proceeding as in Lemma \ref{lemalt}, we can easily prove
that the following relations hold:
\begin{equation}
c(\mathbf{u}_{\varepsilon},\mathbf{u},\mathbf{v})=-c(\mathbf{u}_{\varepsilon},\mathbf{v},\mathbf{u}),
\:\: c(\mathbf{u}_{\varepsilon},\mathbf{u},\mathbf{u})=0,\:\:\:
\forall\mathbf{u},\mathbf{v}\in \mathbf{X}_{0},\label{hopf1}
\end{equation}
\begin{equation}
c_{1}(\mathbf{u}_{\varepsilon},\theta,W)=-c_{1}(\mathbf{u}_{\varepsilon},W,\theta),
\:\: c_{1}(\mathbf{u}_{\varepsilon},\theta,\theta)=0,\:\:\:
\forall\theta,W\in Y.\label{hopf2}
\end{equation}

On the other hand, arguing as in \cite{Casas}, we can construct a
function $\theta_{\delta}\in H^{1}(\Omega)$ such that
\begin{equation*}
\theta_{\delta}=\phi_{2}\; \text{ on }\;\{x_{3}=0\},\:\:\:\;\frac{\partial\theta_{\delta}}{\partial \textbf{n}}+B\theta_{\delta}=0\; \text{ on } \;\Gamma_{1},\:\:\:\;\frac{\partial\theta_{\delta}}{\partial\textbf{n}}=\phi_{1}\;\text{ on }\;\Gamma_{0}\backslash\{x_{3}=0\},
\end{equation*}
\begin{equation}\label{thetadeltaest}
\left\Vert \theta_{\delta} \right\Vert_{L^4(\Omega)}\leq \delta,\:\:\:\; \left\Vert \theta_{\delta} \right\Vert_{H^1(\Omega)}\leq C\left(\left\Vert \phi_{1} \right\Vert_{H^{\frac{1}{2}}(\Gamma_{0}\backslash\{x_{3}=0\})}+\left\Vert \phi_{2}\right\Vert_{H^{\frac{1}{2}}(\{x_{3}=0\})}\right).
\end{equation}
Here $\delta$ is an arbitrarily small number and the constant $C$ depends on $\delta$.\\

Rewriting $[\textbf{u},\theta] \in \textbf{X}\times H^{1}(\Omega)$
in the form
$\textbf{u}=\textbf{u}_{\varepsilon}+\widehat{\textbf{u}}$ and
$\theta=\theta_{\delta}+\widehat{\theta}$ with
$\widehat{\textbf{u}}\in \textbf{X}_{0}$ and $\widehat{\theta}\in Y$
new unknown functions, from Definition \ref{defw1} we obtain the
following nonlinear problem: Find
$[\widehat{\textbf{u}},\widehat{\theta}] \in \mathbf{X}_0\times Y$
such that
\begin{eqnarray}
&& Pr\, a(\widehat{\textbf{u}},\mathbf{v})+Pr\, M\, b_{1}(\widehat{\theta},\mathbf{v})+c(\widehat{\mathbf{u}},\widehat{\mathbf{u}},\mathbf{v})+c(\widehat{\mathbf{u}},{\mathbf{u}_{\varepsilon}},\mathbf{v})+c({\mathbf{u}_{\varepsilon}},\widehat{\mathbf{u}},\mathbf{v})=\left\langle f(\widehat{\theta}+\theta_\delta),\mathbf{v}\right\rangle\nonumber\\
&&\hspace{0.5cm} -Pr\, a({\textbf{u}_{\varepsilon}},\mathbf{v})-Pr\, M\, b_{1}({\theta}_{\delta},\mathbf{v})-c({\mathbf{u}_{\varepsilon}},{\mathbf{u}_{\varepsilon}},\mathbf{v}),\;\:\forall\mathbf{v}\in \mathbf{X}_{0},\label{eq:020a}\\
&&c_{1}(\widehat{\mathbf{u}},\widehat{\theta},W)+c_{1}({\mathbf{u}_{\varepsilon}},\widehat{\theta},W)+c_{1}(\widehat{\mathbf{u}},{\theta_\delta},W)+a_{1}(\widehat{\theta},W)+\left\langle B\widehat{\theta},W\right\rangle _{\Gamma_{1}}=\left\langle \phi_{1},W\right\rangle _{\Gamma_{0}\backslash\{x_{3}=0\}}\nonumber\\
&&\hspace{0.5cm}-c_{1}({\mathbf{u}_{\varepsilon}},{\theta_\delta},W)-a_{1}({\theta_\delta},W)-\left\langle
B{\theta_\delta},W\right\rangle _{\Gamma_{1}},\:\forall W\in Y.
\label{eq:021a}
\end{eqnarray}
Here $\left\langle f(\theta),\mathbf{v}\right\rangle$ is as in
Definition \ref{defw1}.

In order to prove existence of a solution
$[\widehat{\textbf{u}},\widehat{\theta}]\in \textbf{X}_{0}\times Y$
of (\ref{eq:020a})-(\ref{eq:021a}), we introduce the mapping
$F:\textbf{X}_{0}\rightarrow \textbf{X}_{0}$ defined by
$F(\bar{\mathbf{u}})=\widehat{\textbf{u}}$, $\bar{\mathbf{u}}\in
\textbf{X}_{0},$ such that
$[\widehat{\textbf{u}},\widehat{\theta}]\in \textbf{X}_{0}\times Y$
is the solution of the following linearized problem
\begin{eqnarray}
&& Pr\, a(\widehat{\textbf{u}},\mathbf{v})+Pr\, M\, b_{1}(\widehat{\theta},\mathbf{v})+c(\bar{\mathbf{u}},\widehat{\mathbf{u}},\mathbf{v})+c(\bar{\mathbf{u}},{\mathbf{u}_{\varepsilon}},\mathbf{v})+c({\mathbf{u}_{\varepsilon}},\widehat{\mathbf{u}},\mathbf{v})=\left\langle f(\widehat{\theta}+\theta_\delta),\mathbf{v}\right\rangle\nonumber\\
&&\hspace{0.5cm} -Pr\, a({\textbf{u}_{\varepsilon}},\mathbf{v})-Pr\, M\, b_{1}({\theta}_{\delta},\mathbf{v})-c({\mathbf{u}_{\varepsilon}},{\mathbf{u}_{\varepsilon}},\mathbf{v}),\;\:\forall\mathbf{v}\in \mathbf{X}_{0},\label{eq:020al}\\
&&c_{1}(\bar{\mathbf{u}},\widehat{\theta},W)+c_{1}({\mathbf{u}_{\varepsilon}},\widehat{\theta},W)+c_{1}(\bar{\mathbf{u}},{\theta_\delta},W)+a_{1}(\widehat{\theta},W)+\left\langle B\widehat{\theta},W\right\rangle _{\Gamma_{1}}=\left\langle \phi_{1},W\right\rangle _{\Gamma_{0}\backslash\{x_{3}=0\}}\nonumber\\
&&\hspace{0.5cm}-c_{1}({\mathbf{u}_{\varepsilon}},{\theta_\delta},W)-a_{1}({\theta_\delta},W)-\left\langle
B{\theta_\delta},W\right\rangle _{\Gamma_{1}},\:\forall W\in Y.
\label{eq:021al}
\end{eqnarray}
In next lemma, we shall show that the operator
$F:\textbf{X}_{0}\rightarrow \textbf{X}_{0}$ is well-defined.

\begin{lemma} \label{lemexissolpl}
Let $\bar{\mathbf{u}}\in \mathbf{X}_{0}$ and ${\phi}_{1}\in
H^{\frac{1}{2}}(\Gamma_{0}\setminus\left\{x_{3}=0\right\})$. Then
there exists a unique weak solution
$[\widehat{\mathbf{u}},\widehat{\theta}]\in \mathbf{X}_{0}\times Y$
of problem (\ref{eq:020al})-(\ref{eq:021al}). Moreover, the
following estimates hold
\begin{eqnarray}
&\left\Vert \widehat{\mathbf{u}} \right\Vert _{
H^1(\Omega)}&\!\!\!\!\leq  C\left(\vert b\vert+(R+M)\left(\Vert
\widehat{\theta}\Vert _{H^{1}(\Omega)}+\Vert \theta_\delta\Vert
_{H^{1}(\Omega)}\right)+\frac{1}{Pr}\Vert \bar{\mathbf{u}}\Vert
_{L^{4}(\Omega)}\Vert \mathbf{u}_{\varepsilon}\Vert
_{ L^{4}(\Omega)}\right)\nonumber\\
&& + C\left(\Vert \nabla\mathbf{u}_{\varepsilon}\Vert
_{L^{2}(\Omega)}+\frac{1}{Pr}\Vert \mathbf{u}_{\varepsilon}\Vert
_{H^{1}(\Omega)}^{2}\right),\label{eq:estu}
\end{eqnarray}

\vspace{-0,6 cm}

\begin{eqnarray}
&\Vert \widehat{\theta}\Vert _{H^1(\Omega)}&\!\!\!\!\leq
C\left(\left\Vert \phi_{1}\right\Vert
_{H^{\frac{1}{2}}(\Gamma_{0}\backslash\{x_{3}=0\})}+\Vert
\bar{\mathbf{u}}\Vert_{L^{4}(\Omega)}\Vert
\theta_\delta\Vert_{L^{4}(\Omega)}+\left(\Vert
\mathbf{u}_{\varepsilon}\Vert_{L^{4}(\Omega)}+1+B\right)\Vert
\theta_\delta\Vert_{H^{1}(\Omega)}\right),\label{eq:027}
\end{eqnarray}
where $C$ is a constant independent of $\bar{\mathbf{u}}$,
$\widehat{\mathbf{u}}$, $\phi_{1}$ and $\widehat{\theta}$.
\end{lemma}
\begin{proof}
We consider the bilinear continuous mappings
$\hat{a}:\textbf{X}_{0}\times \textbf{X}_{0}\rightarrow\mathbb{R}$
and $\hat{a}_{1}:Y\times Y\rightarrow\mathbb{R}$ given by
\begin{eqnarray*}
&&\displaystyle \hat{a}(\widehat{\textbf{u}},\textbf{v})=Pr\, a(\widehat{\textbf{u}},\textbf{v})+c(\bar{\textbf{u}},\widehat{\textbf{u}},\textbf{v})+c({\mathbf{u}_{\varepsilon}},\widehat{\mathbf{u}},\mathbf{v}), \:\:\forall \widehat{\textbf{u}},\textbf{v}\in \textbf{X}_{0},\\
&&\displaystyle\hat{a}_{1}(\widehat{\theta},W)=c_{1}(\bar{\textbf{u}},\widehat{\theta},W)+c_{1}({\mathbf{u}_{\varepsilon}},\widehat{\theta},W)+a_{1}(\widehat{\theta},W)+\left\langle
B\widehat{\theta},W\right\rangle _{\Gamma_{1}},
\:\:\forall\widehat{\theta},W \in Y.
\end{eqnarray*}
Consequently, we rewrite (\ref{eq:020al}) and (\ref{eq:021al}) as
\begin{eqnarray}
\hat{a}(\widehat{\textbf{u}},\textbf{v}) =\left\langle l_{\widehat{\theta}},\textbf{v}\right\rangle,\;\forall\textbf{v}\in \textbf{X}_{0},\label{eq:025a}\\
\hat{a}_{1}(\widehat{\theta},W)=\left\langle
\widetilde{\phi}_{1},W\right\rangle,\;\forall W\in Y,\label{eq:026a}
\end{eqnarray}
where
$$\left\langle
l_{\widehat{\theta}},\textbf{v}\right\rangle= \left\langle
f(\widehat{\theta}+\theta_\delta),\textbf{v}\right\rangle -Pr\, M\,
b_{1}(\widehat{\theta}+{\theta}_{\delta},\textbf{v})-c(\bar{\mathbf{u}},{\mathbf{u}_{\varepsilon}},\mathbf{v})
-Pr\, a({\textbf{u}_{\varepsilon}},\mathbf{v})-c({\mathbf{u}_{\varepsilon}},{\mathbf{u}_{\varepsilon}},\mathbf{v}),\;\forall\textbf{v}\in
\textbf{X}_{0},
$$
$$\left\langle \widetilde{\phi}_{1},W\right\rangle=
\left\langle
\phi_{1},W\right\rangle_{\Gamma_{0}\backslash\{x_{3}=0\}}-c_{1}(\bar{\mathbf{u}},\theta_\delta,W)-c_{1}({\mathbf{u}_{\varepsilon}},{\theta_\delta},W)
-a_{1}({\theta_\delta},W)-\left\langle
B{\theta_\delta},W\right\rangle _{\Gamma_{1}},\;\forall W\in Y.$$
 We
can verify that the operator bilinear $\hat{a}_{1}$ is continuous
and coercive on $Y$ and $\widetilde{\phi}_{1}\in Y'$. Indeed, the
continuity of $\hat{a}_{1}$ and $\widetilde{\phi}_{1}$ it follows
from the H\"older inequality and Sobolev embeddings. Moreover, the
coercivity of $\hat{a}_{1}$ follows from (\ref{eq:altc1}),
(\ref{hopf2}) and the following generalized Poincar\'e inequality:
\begin{eqnarray}\label{poincare1}
\Vert u\Vert_{L^2(\Omega)}\leq C\left(\Vert \nabla
u\Vert_{L^2(\Omega)}+\int_\Sigma\vert u\vert\right),\ \forall u\in
H^1(\Omega),
\end{eqnarray}
where $C=C(n,\Omega,\Sigma)$ and $\Sigma$ is an arbitrary portion of $\partial\Omega$ of positive measure (cf. Lemma 10.9 in \cite{Feireisl}, p. 327; see also \cite{Galdi1}, p. 56).
 Therefore, by the Lax-Milgram theorem, there exists a unique $\widehat{\theta}\in Y$ which satisfies equation (\ref{eq:026a}). Knowing $\widehat{\theta}$ and inserting it in the equation
(\ref{eq:025a}), by using the H\"older inequality and Sobolev
embeddings we can verify that the operator bilinear $\hat{a}$ is
continuous on $\textbf{X}_{0}$ and $l_{\theta}\in \textbf{X}'_{0}$.
Moreover, from (\ref{eq:altc}), (\ref{hopf1}) and using the
generalized Poincar\'e inequality (\ref{poincare1}) we have that
$\hat{a}$ is coercive. Therefore, by the Lax-Milgram theorem, there
exists a unique $\widehat{\textbf{u}}\in \textbf{X}_{0}$ which
satisfies equation (\ref{eq:025a}). Finally, setting
$\mathbf{v}=\widehat{\mathbf{u}}$ in (\ref{eq:025a}),
$W=\widehat{\theta}$ in
 (\ref{eq:026a}) and using the generalized Poincar\'e inequality (\ref{poincare1}), we easily obtain (\ref{eq:estu}) and (\ref{eq:027}).
 \end{proof}

Now, using the Schauder Fixed Point Theorem, we will prove the
existence of a fixed point of $F$ which yields a solution of
(\ref{eq:020a})-(\ref{eq:021a}). For that, we consider the ball
$B_{r}=\{\widehat{\mathbf{u}}\in \textbf{X}_{0}:\left\Vert
\widehat{\mathbf{u}}\right\Vert _{\textbf{H}^{1}(\Omega)}\leq r
\}\subseteq \textbf{X}_{0}$, where $r$ is a positive constant such
that
$$r>C\!\left( \vert b\vert +(R+M)\!\left[\Vert \phi_{1}\Vert
_{H^{\frac{1}{2}}(\Gamma_{0}\backslash\{x_{3}=0\})}\!+\!\left(\Vert
\mathbf{u}_{\varepsilon}\Vert_{L^{4}(\Omega)}+1+B\right)\!\Vert
\theta_\delta\Vert_{H^{1}(\Omega)}\right]\!+\!\Vert
\nabla\mathbf{u}_{\varepsilon}\Vert
_{L^{2}(\Omega)}\!+\frac{1}{Pr}\Vert \mathbf{u}_{\varepsilon}\Vert
_{H^{1}(\Omega)}^{2}\!\right)\!.$$ It follows from
(\ref{eq:estu})-(\ref{eq:027}) that $F(B_{r})\subseteq B_{r},$
provided $\delta$ be small enough and $Pr$ large enough. Moreover
$F$ is completely continuous. This follows from the next inequality
\begin{equation}\label{tcomcon1}
\Vert F(\bar{\mathbf{u}}_{1}) - F(\bar{\mathbf{u}}_{2})
\Vert_{H^{1}(\Omega)} \leq K\Vert \bar{\mathbf{u}}_{1} -
\bar{\mathbf{u}}_{2}\Vert_{L^{4}(\Omega)},
\end{equation}
and from the compact embedding of
$\mathbf{H}^{1}(\Omega)\hookrightarrow \mathbf{L}^{4}(\Omega)$,
where
\begin{eqnarray*}
&K&\!\!\!\!\!=\frac{C}{Pr}\left(\Vert \widehat{\textbf{u}}_{2}\Vert
_{L^{4}(\Omega)}\!+\Vert \textbf{u}_{\varepsilon}\Vert
_{L^{4}(\Omega)}\right)\nonumber\\
&&+ C( M+ R) \! \left(\!\left\Vert \phi_{1}\right\Vert
_{H^{\frac{1}{2}}(\Gamma_{0}\backslash\{x_{3}=0\})}\!\!+\delta\Vert
\bar{\mathbf{u}}_{2}\Vert_{L^{4}(\Omega)}\!+\!(\Vert
\mathbf{u}_{\varepsilon}\Vert_{L^{4}(\Omega)}\!+\!1\!+\!B)\Vert
\theta_\delta\Vert_{H^{1}(\Omega)}\!\right),
 \end{eqnarray*}
and $C$ is a constant independent of $\bar{\mathbf{u}}_{1}$ and
$\bar{\mathbf{u}}_{2}$.

Let us prove (\ref{tcomcon1}). Let $\widehat{\theta}_{i}\in Y$ be
the solution of equation (\ref{eq:021al}) corresponding to
$\bar{\mathbf{u}}_{i} \in \mathbf{X}_{0}$ and set
$\widehat{\mathbf{u}}_{i}=F(\bar{\mathbf{u}}_{i})$, for $i=1,2$.
From (\ref{eq:020al}) and (\ref{eq:021al}) we obtain
\begin{eqnarray}\label{equ1}
&Pr\,a(\widehat{\textbf{u}}_{1}-\widehat{\textbf{u}}_{2},\textbf{v})&\!\!\!\!\!+c(\bar{\textbf{u}}_{1},\widehat{\textbf{u}}_{1}-\widehat{\textbf{u}}_{2},\textbf{v})+c(\bar{\textbf{u}}_{1}-\bar{\textbf{u}}_{2},\widehat{\textbf{u}}_{2},\textbf{v})=-c(\bar{\textbf{u}}_{1}-\bar{\textbf{u}}_{2},{\textbf{u}_\epsilon},\textbf{v})\nonumber\\
&&\!\!\!\!-c({\textbf{u}_\epsilon},\widehat{\textbf{u}}_{1}-\widehat{\textbf{u}}_{2},\textbf{v})-Pr\,
M\,
b_{1}(\widehat{\theta}_{1}-\widehat{\theta}_{2},\textbf{v})+\int_\Omega
Pr
R(\widehat{\theta}_{1}-\widehat{\theta}_{2})v_{3},\;\:\forall\textbf{v}\in
\mathbf{X}_{0},
\end{eqnarray}

\vspace{-0,7 cm}

\begin{eqnarray}\label{eqtheta1}
&a_{1}(\widehat{\theta}_{1}-\widehat{\theta}_{2},W)&\!\!\!\!\!+\left\langle
B(\widehat{\theta}_{1}-\widehat{\theta}_{2}),W\right\rangle
_{\Gamma_{1}}=-c_{1}(\bar{\textbf{u}}_{1},\widehat{\theta}_{1}-\widehat{\theta}_{2},W)\nonumber\\
&&-c_{1}(\bar{\textbf{u}}_{1}-\bar{\textbf{u}}_{2},\widehat{\theta}_{2},W)-c_{1}(\textbf{u}_{\varepsilon},\widehat{\theta}_{1}-\widehat{\theta}_{2},W)-c_{1}(\bar{\textbf{u}}_{1}-\bar{\textbf{u}}_{2},\theta_{\delta},W),\;\forall
W\in Y.
\end{eqnarray}
Setting  $W=\widehat{\theta}_{1}-\widehat{\theta}_{2}$ in
(\ref{eqtheta1}) and using (\ref{eq:altc1}), (\ref{hopf2}), the
H\"older inequality, the continuous embedding
$H^1(\Omega)\hookrightarrow L^4(\Omega)$ and the Poincar\'e
inequality, it is not difficult to obtain
\begin{equation}\label{eqtheta2}
\Vert \nabla(\widehat{\theta}_{1}-\widehat{\theta}_{2})\Vert
_{L^{2}(\Omega)} \leq C \left(\Vert \widehat{\theta}_{2}\Vert
_{H^{1}(\Omega)}+\Vert \theta_{\delta}\Vert
_{H^{1}(\Omega)}\right)\Vert
\bar{\textbf{u}}_{1}-\bar{\textbf{u}}_{2}\Vert _{L^{4}(\Omega)} .
 \end{equation}
Now, using (\ref{eq:027}) and the Poincar\'e inequality
(\ref{poincare1}), from (\ref{eqtheta2}) we obtain
\begin{equation}
\Vert \widehat{\theta}_{1}\!-\widehat{\theta}_{2}\Vert
_{H^{1}(\Omega)}\! \leq C \! \left(\!\left\Vert \phi_{1}\right\Vert
_{H^{\frac{1}{2}}(\Gamma_{0}\backslash\{x_{3}=0\})}\!\!+\delta\Vert
\bar{\mathbf{u}}_{2}\Vert_{L^{4}(\Omega)}\!+\!(\Vert
\mathbf{u}_{\varepsilon}\Vert_{L^{4}(\Omega)}\!+\!1\!+\!B)\Vert
\theta_\delta\Vert_{H^{1}(\Omega)}\!\right)\left\Vert
\bar{\textbf{u}}_{1}\!-\bar{\textbf{u}}_{2}\right\Vert
_{L^{4}(\Omega)}\!\!.\label{eq:036}
 \end{equation}
Setting
$\textbf{v}=\widehat{\textbf{u}}_{1}-\widehat{\textbf{u}}_{2}$ in
(\ref{equ1}) and using (\ref{eq:altc}), (\ref{hopf1}), the H\"older
inequality, the continuous embedding $H^1(\Omega)\hookrightarrow
L^2(\Omega)$ and the Poincar\'e inequality, we obtain
\begin{eqnarray}
Pr\left\Vert
\nabla(\widehat{\textbf{u}}_{1}\!-\widehat{\textbf{u}}_{2})\right\Vert
_{L^{2}(\Omega)}\! \leq (\Vert \widehat{\textbf{u}}_{2}\Vert
_{L^{4}(\Omega)}\!+\Vert \textbf{u}_{\varepsilon}\Vert
_{L^{4}(\Omega)})\!\left\Vert
\bar{\textbf{u}}_{1}\!-\bar{\textbf{u}}_{2}\right\Vert
_{L^{4}(\Omega)}\! + C(Pr M\!+\!Pr R)\Vert
\widehat{\theta}_{1}\!-\widehat{\theta}_{2}\Vert
_{H^{1}(\Omega)}.\label{mar1a}
 \end{eqnarray}
Then, using the Poincar\'e inequality, from (\ref{mar1a}) we get
\begin{equation}
\left\Vert
\widehat{\textbf{u}}_{1}-\widehat{\textbf{u}}_{2}\right\Vert
_{H^{1}(\Omega)}\leq C \left(\frac{1}{Pr}(\Vert
\widehat{\textbf{u}}_{2}\Vert _{L^{4}(\Omega)}+\Vert
\textbf{u}_{\varepsilon}\Vert _{L^{4}(\Omega)})\!\left\Vert
\bar{\textbf{u}}_{1}\!-\bar{\textbf{u}}_{2}\right\Vert
_{L^{4}(\Omega)} + (M+ R)\Vert
\widehat{\theta}_{1}\!-\widehat{\theta}_{2}\Vert
_{H^{1}(\Omega)}\right). \label{schu}
\end{equation}
Thus, (\ref{tcomcon1}) follows from (\ref{schu}) and (\ref{eq:036}).
Therefore, the Schauder Theorem implies that $F$ has a fixed point
$\widehat{\mathbf{u}}=F(\widehat{\textbf{u}})$. The field
$\widehat{\textbf{u}}$, together with the corresponding function
$\widehat{\theta}=\theta_{\widehat{\textbf{u}}}\in Y$ solving the
problem (\ref{eq:021al}) for
$\bar{\textbf{u}}=\widehat{\textbf{u}}$, is a solution to the
problem (\ref{eq:020a})-(\ref{eq:021a}). We collect this result in
the following theorem:
\begin{theorem} \label{tmaexistenciasolucion2}
Let ${\phi}_{1}\in{
H}^{\frac{1}{2}}(\Gamma_{0}\setminus\left\{x_{3}=0\right\})$,
$\phi_{2}\in H^{1/2}_{00}(\left\{x_{3}=0\right\})$,
$\mathbf{u}^{0}\in
\widetilde{\mathbf{H}}^{1/2}_{00}(\Gamma_{0}^{2})$ and
$\mathbf{g}\in \widetilde{\mathbf{H}}^{1/2}_{00}(\Gamma_{0}^{1})$.
Then there exists at least one solution $[\mathbf{u},\theta] \in
\widetilde{\mathbf{X}}\times H^1(\Omega)$ of problem
(\ref{eq:020})-(\ref{eq:022}) provided $Pr$ be large enough, and the
following estimate holds
\begin{equation}
\left\Vert \mathbf{u}\right\Vert_{H^1(\Omega)}+ \left\Vert
\theta\right\Vert_{H^1(\Omega)} \leq C\left(\Vert \mathbf{u}^{0}
\Vert_{H^{\frac{1}{2}}(\Gamma_{0}^{2})}+\Vert \mathbf{g}
\Vert_{H^{\frac{1}{2}}(\Gamma_{0}^{1})}+\left\Vert \phi_{1}
\right\Vert_{H^{\frac{1}{2}}(\Gamma_{0}\backslash\{x_{3}=0\})}+\left\Vert
\phi_{2}\right\Vert_{H^{\frac{1}{2}}(\{x_{3}=0\})}\right),\label{eq:estutheta}
\end{equation}
where the constant $C$ depends linearly of the parameters $M$, $B$
and $R$.
\end{theorem}

\subsection{Uniqueness of the Weak Solutions}

\hspace{0.4cm}The purpose of this section is to determine conditions on the
boundary data and parameters which guarantee the uniqueness of the
weak solution $[\textbf{u},\theta]\in \widetilde{\textbf{X}}\times
H^1(\Omega)$ to the problem (\ref{eq:020})-(\ref{eq:022}). For that,
suppose that there exist
$[\textbf{u}_{1},\theta_{1}],[\textbf{u}_{2},\theta_{2}]\in
\widetilde{\textbf{X}}\times H^1(\Omega)$ weak solutions of system
(\ref{eq:020})-(\ref{eq:022}). Then, defining
$\textbf{u}=\textbf{u}_{1}-\textbf{u}_{2}$ and
$\theta=\theta_{1}-\theta_{2}$, we obtain that
$[\textbf{u},\theta]\in \textbf{X}_{0}\times Y$ solves the system
\begin{eqnarray}
Pr\, a(\textbf{u},\textbf{v})+Pr\, M\, b_{1}(\theta,\textbf{v})+c(\textbf{u},\textbf{u}_{1},\textbf{v})+c(\textbf{u}_{2},\textbf{u},\textbf{v})= \int_\Omega Pr\; R\; \theta v_{3}\: d\Omega ,\;\forall\textbf{v}\in \mathbf{X}_{0},\label{unsol7}\\
c_{1}(\textbf{u},\theta_{1},W)+c_{1}(\textbf{u}_{2},\theta,W)+a_{1}(\theta,W)+\left\langle
B\theta,W\right\rangle _{\Gamma_{1}}=0,\;\forall W\in Y.
\label{unsol8}
\end{eqnarray}
Proceeding as in Lemma \ref{lemalt}, we can easily prove that if
$\textbf{u}_{2}\in \widetilde{\textbf{X}}$, $\textbf{u}\in
\textbf{X}_{0}$ and $\theta\in Y$, then
$c(\textbf{u}_{2},\textbf{u},\textbf{u})=0$ and
$c_{1}(\textbf{u}_{2},\theta,\theta)=0$. Thus, setting
$\textbf{v}=\textbf{u}$ in (\ref{unsol7}), $W=\theta$ in
(\ref{unsol8}), and using the H\"older inequality, Sobolev
embeddings and the Poincar\'e inequality (\ref{poincare1}), we
deduce
\begin{equation} \label{unsol12}
Pr \Vert \nabla\textbf{u} \Vert_{L^2(\Omega)}\leq  Pr M \Vert \nabla \theta\Vert_{L^2(\Omega)} + C\Vert \nabla\textbf{u} \Vert_{L^2(\Omega)}\Vert  \nabla\textbf{u}_{1} \Vert_{L^2(\Omega)}+ C Pr R  \Vert  \nabla\theta\Vert_{L^2(\Omega)},
 \end{equation}
\begin{equation} \label{unsol13}
 \Vert \nabla\theta \Vert_{L^2(\Omega)}
  \leq C\Vert \nabla\textbf{u} \Vert_{L^2(\Omega)}\Vert \nabla \theta_{1}\Vert_{L^2(\Omega)}.
 \end{equation}
Using (\ref{unsol13}) in (\ref{unsol12}), we find
\begin{equation*} \label{unsol14}
Pr \Vert \nabla\textbf{u} \Vert_{L^2(\Omega)} \leq C(\Vert
\nabla\textbf{u}_{1} \Vert_{L^2(\Omega)}+ (Pr M+Pr R) \Vert \nabla
\theta_{1}\Vert_{L^2(\Omega)} )\Vert \nabla\textbf{u}
\Vert_{L^2(\Omega)}.
 \end{equation*}
Now, taking into account that $[\textbf{u}_{1},\theta_{1}]$ and
$[\textbf{u}_{2},\theta_{2}]$ are weak solutions to the problem
(\ref{eq:020})-(\ref{eq:022}), then from Theorem
\ref{tmaexistenciasolucion2}, we have that
$[\textbf{u}_{1},\theta_{1}]$ and $[\textbf{u}_{2},\theta_{2}]$
satisfy the estimate (\ref{eq:estutheta}), which imply that
\begin{equation*}
Pr \Vert \nabla\textbf{u} \Vert_{L^2(\Omega)}\! \leq C(Pr (M +
R)+1)\!\!\left[\!\Vert \mathbf{u}^{0}
\Vert_{H^{\frac{1}{2}}(\Gamma_{0}^{2})}\!\!+\!\Vert \mathbf{g}
\Vert_{H^{\frac{1}{2}}(\Gamma_{0}^{1})}\!\!+\!\left\Vert \phi_{1}
\right\Vert_{H^{\frac{1}{2}}(\Gamma_{0}\backslash\{x_{3}=0\})}\!\!+\!\left\Vert
\phi_{2}\right\Vert_{H^{\frac{1}{2}}(\{x_{3}=0\})}\!\right]\!\!\Vert
\nabla\textbf{u} \Vert_{L^2(\Omega)},
 \end{equation*}
where the constant $C$ depends almost linearly of the parameters
$M$, $B$ y $R$. Therefore, if the condition
\begin{equation}\label{condpeq}
Pr - C(Pr (M + R) + 1 )\left[\Vert \mathbf{u}^{0}
\Vert_{H^{\frac{1}{2}}(\Gamma_{0}^{2})}\!+\!\Vert \mathbf{g}
\Vert_{H^{\frac{1}{2}}(\Gamma_{0}^{1})}\!+\!\left\Vert \phi_{1}
\right\Vert_{H^{\frac{1}{2}}(\Gamma_{0}\backslash\{x_{3}=0\})}\!+\!\left\Vert
\phi_{2}\right\Vert_{H^{\frac{1}{2}}(\{x_{3}=0\})}\right]> 0
 \end{equation}
is satisfied, we conclude that $\Vert \nabla\textbf{u} \Vert_{L^2(\Omega)}=0$, and consequently $\textbf{u}=0$, which implies that $\textbf{u}_{1}=\textbf{u}_{2}$. Moreover, using this fact in (\ref{unsol13}), we obtain that $\Vert \nabla\theta \Vert_{L^2(\Omega)}=0$, and consequently $\theta=0$, which implies that $\theta_{1}=\theta_{2}$. Thus we have proved the following theorem:

\begin{theorem}\label{tmasolunica}
Let ${\phi}_{1}\in H^{\frac{1}{2}}(\Gamma_{0}\setminus\left\{x_{3}=0\right\})$, $\phi_{2}\in H^{1/2}_{00}(\left\{x_{3}=0\right\})$  $\mathbf{u}^{0}\in \widetilde{\mathbf{H}}^{1/2}_{00}(\Gamma_{0}^{2})$ and $\mathbf{g}\in \widetilde{\mathbf{H}}^{1/2}_{00}(\Gamma_{0}^{1})$. If the condition (\ref{condpeq}) is satisfied, then the problem  (\ref{eq:020})-(\ref{eq:022}) has a unique solution $[\mathbf{u},\theta]\in \widetilde{\mathbf{X}}\times H^1(\Omega)$. Moreover, the solution $[\mathbf{u},\theta]$ satisfies the estimate (\ref{eq:estutheta}).
\end{theorem}

\begin{remark}\label{obssolun}
Observe that the condition
\begin{equation*}
Pr - C(Pr (M +R) +1)\left[\Vert \mathbf{u}^{0}
\Vert_{H^{\frac{1}{2}}(\Gamma_{0}^{2})}\!+\!\Vert \mathbf{g}
\Vert_{H^{\frac{1}{2}}(\Gamma_{0}^{1})}\!+\!\left\Vert \phi_{1}
\right\Vert_{H^{\frac{1}{2}}(\Gamma_{0}\backslash\{x_{3}=0\})}+\!\left\Vert
\phi_{2}\right\Vert_{H^{\frac{1}{2}}(\{x_{3}=0\})}\right]> 0
\end{equation*}
is verified if either the functions $\mathbf{u}^0$, $\mathbf{g}$,
$\phi_{1}$ and $\phi_{2}$ are small, or if the coefficients $M$, $R$
and $B$ are small. In particular, for small values of $M$, $R$ and
$B$ and boundary data $\mathbf{u}^{0}=0$, $\mathbf{g}=0$,
$\phi_{1}=0$ and $\phi_{2}=\theta_{c}$, the basic solution
$[\mathbf{u}_{b},\theta_{b},p_{b}]$ given by (\ref{estadobasico}) is
unique.
\end{remark}

\subsection{Regularity}

\hspace{0.4cm}In Subsection \ref{solweak} was demonstrated the existence of a weak
solution $[\textbf{u},\theta]\in \widetilde{\mathbf{X}}\times
H^1(\Omega)$ to the problem (\ref{eq:020})-(\ref{eq:022}); however,
taking into account the tangential and normal derivatives of the
temperature at the boundary, we need to prove that $\theta\in
H^2(\Omega)$ (see Lemma \ref{lempardo}). In this subsection we
analyze the following regularity problem for the weak solution
$\theta\in H^1(\Omega)$: Given $\textbf{u}\in
\widetilde{\textbf{X}}$, find $\theta \in H^2(\Omega)$ such that
\begin{align} \label{problema1}
 \begin{cases}
 -\Delta\theta=- (\textbf{u}\cdot\nabla)\theta  & \text{ in  } \Omega,\\
  \frac{\partial\theta}{\partial\textbf{n}}+B\theta=0 & \text{ on } \Gamma_{1},\\
  \frac{\partial\theta}{\partial\textbf{n} }=\phi_{1} & \text{ on } \Gamma_{2},\\
 \theta=\phi_{2} & \text{ on } \Gamma_{3},
 \end{cases}
\end{align}
where $\Gamma_{1}:=\{x_{3}=1\}$, $\Gamma_{3}:=\{x_{3}=0\}$ and $\Gamma_{2}:=\partial\Omega\setminus \{\Gamma_{1}\cup\Gamma_{3}\}$ (see Figure \ref{fronteraomega1}).
\begin{center}
\psset{unit=0.60cm,linewidth=0.8pt}
\begin{pspicture}(0,-0.75)(6.5,6.0)
\pspolygon[linestyle=dashed,dash=5pt 5pt,linecolor=black,fillcolor=gray,fillstyle=solid,opacity=1](0,0)(4,0)(6.22,1.6)(2.22,1.6)
\pspolygon[linecolor=black,fillcolor=gray,fillstyle=solid,opacity=0.3](0,4)(4,4)(4,0)(0,0)
\pspolygon[linecolor=black,fillcolor=gray,fillstyle=solid,opacity=0.3](4,4)(6.22,5.6)(6.22,1.6)(4,0)
\pspolygon[linecolor=black,fillcolor=gray,fillstyle=solid,opacity=0.6](2.22,5.62)(6.22,5.6)(4,4)(0,4)
\psline[linestyle=dashed,dash=5pt 5pt](2.22,1.6)(2.22,5.6)
\rput[tl](3,5.24){$\Gamma_1$}
\rput[tl](4.86,3.14){$\Gamma_2$}
\rput[tl](4.32,-0.12){$\Gamma_3$}
\psarc[linewidth=0.3pt,arrowsize=2.5pt 2.0]{<-}(4.25,1){1.25}{180}{270}
\end{pspicture}
\\[-2mm]
\figcaption{Representation of $\partial\Omega$.} \label{fronteraomega1}
\end{center}

In this subsection. we will use the following space
\begin{equation*}
H^{3/2}_{00}(\Gamma)=\{v\in L^{2}(\Gamma): \mbox{ there exists } g\in H^{\frac{3}{2}}(\partial\Omega), \ \left. g\right|_{\Gamma}=v, \ \left. g\right|_{\partial\Omega \setminus \Gamma}=0 \}.
\end{equation*}

\begin{theorem} \label{lemregularidad1}
Let $\phi_{1} \in H^{\frac{1}{2}}(\Gamma_{2})$, $\phi_{2} \in H^{3/2}_{00}(\Gamma_{3})$ and $f \in L^p(\Omega)$ with $\frac{6}{5}<p\leq 2$. Then, the system
\begin{align} \label{problema1.2}
 \begin{cases}
 -\Delta\theta=f  & \text{ in  } \Omega,\\
 \frac{\partial\theta}{\partial\mathbf{n}}+B\theta=0 & \text{ on } \Gamma_{1},\\
  \frac{\partial\theta}{\partial\mathbf{n} }=\phi_{1} & \text{ on } \Gamma_{2},\\
 \theta=\phi_{2} & \text{ on } \Gamma_{3},
 \end{cases}
\end{align}
has a solution $\theta \in W^{2,p}(\Omega).$
\end{theorem}
\begin{proof}
We first convert the problem (\ref{problema1.2}) with Robin, Neumann
and Dirichlet conditions, in a boundary problem with only Dirichlet
and Neumann conditions. For that, we will adapt the ideas of
\cite{Ziane}, Section 2. First, we consider the functions
$\eta(x_{3})$ and $\tilde{\theta}$ defined by: \vspace{-0,1 cm}
\begin{equation*}
\eta(x_{3})=\text{exp}[B(2x_{3}-\frac{1}{2}x_{3}^2)] \qquad \text{ and } \qquad \tilde{\theta}=\eta\theta.
\end{equation*}
Since $B$ is constant, it is easy to check that problem
(\ref{problema1.2}) is equivalent to find $\tilde{\theta} \in
W^{2,p}(\Omega)$, such that
\begin{align} \label{problema2}
 \begin{cases}
 -\Delta\tilde{\theta}=\tilde{f}  & \text{ in  } \Omega,\\
  \frac{\partial\tilde{\theta}}{\partial\textbf{n}}=0 & \text{ on } \Gamma_{1},\\
  \frac{\partial\tilde{\theta}}{\partial\textbf{n}}=\tilde{\phi}_{1} & \text{ on } \Gamma_{2},\\
 \tilde{\theta}=\phi_{2} & \text{ on } \Gamma_{3},
 \end{cases}
\end{align}
where $\tilde{f}=-\theta\eta''-2\frac{\partial\theta}{\partial
x_{3}}\eta'+\eta f $ and $\tilde{\phi}_{1}=\eta\phi_{1}$. Taking
into account that $\phi_{2}\in H^{3/2}_{00}(\Gamma_{3})$, we
consider the function $\widetilde{\phi}_{2}\in
H^{\frac{3}{2}}(\partial\Omega)$ such that
$\widetilde{\phi}_{2}=\phi_{2}$ on $\Gamma_{3}$ and
$\widetilde{\phi}_{2}=0$ on $\partial\Omega\backslash\Gamma_{3}$. By
the lifting Theorem, we have that exists $\Phi_{2} \in H^2(\Omega)$
such that
$\left.\Phi_{2}\right|_{\partial\Omega}=\widetilde{\phi}_{2}$.
Considering $\hat{\theta}=\tilde{\theta}-\Phi_{2}$, it is not
difficult to verify that problem (\ref{problema2}) is equivalent to
find $\hat{\theta} \in W^{2,p}(\Omega)$, such that
\begin{align} \label{problema3}
 \begin{cases}
 -\Delta\hat{\theta}=\hat{f}  & \text{ in  } \Omega,\\
 \frac{\partial\hat{\theta}}{\partial\textbf{n}}=\phi_{3} & \text{ on } \Gamma_{1},\\
  \frac{\partial\hat{\theta}}{\partial\textbf{n}}=\hat{\phi}_{1} & \text{ on } \Gamma_{2},\\
 \hat{\theta}=0 & \text{ on } \Gamma_{3},
 \end{cases}
\end{align}
where $\hat{f}=\tilde{f}+\Delta\Phi_{2}$, $\phi_{3}=-\left.\frac{\partial\Phi_{2}}{\partial\textbf{n}}\right|_{\Gamma_{1}}$ and $\hat{\phi}_{1}=\tilde{\phi}_{1}-\left.\frac{\partial\Phi_{2}}{\partial\textbf{n}}\right|_{\Gamma_{2}}$. In order to find the solution  $\hat{\theta}$ of problem (\ref{problema3}), we decompose $\hat{\theta}$ as the sum $\hat{\theta}=\theta_{1}+\theta_{2}+\theta_{3}$, where $\theta_{1}$, $\theta_{2}$ and $\theta_{3}$ solve respectively the following problems:
\begin{align} \label{problema3.1}
 \begin{cases}
 -\Delta\theta_{1}=\hat{f}  & \text{ in  } \Omega,\\
 \frac{\partial\theta_{1}}{\partial\textbf{n}}=0 & \text{ on } \Gamma_{1},\\
  \frac{\partial\theta_{1}}{\partial\textbf{n}}=0 & \text{ on } \Gamma_{2},\\
 \theta_{1}=0 & \text{ on } \Gamma_{3},
 \end{cases}
\end{align}
\begin{align} \label{problema3.2}
 \begin{cases}
 -\Delta\theta_{2}=0  & \text{ in  } \Omega,\\
 \frac{\partial\theta_{2}}{\partial\textbf{n}}=0 & \text{ on } \Gamma_{1},\\
  \frac{\partial\theta_{2}}{\partial\textbf{n}}=\hat{\phi}_{1} & \text{ on } \Gamma_{2},\\
 \theta_{2}=0 & \text{ on } \Gamma_{3},
 \end{cases}
\end{align}
\begin{align} \label{problema3.3}
 \begin{cases}
 -\Delta\theta_{3}=0  & \text{ in  } \Omega,\\
  \frac{\partial\theta_{3}}{\partial\textbf{n}}=\phi_{3} & \text{ on } \Gamma_{1},\\
  \frac{\partial\theta_{3}}{\partial\textbf{n}}=0 & \text{ on } \Gamma_{2},\\
 \theta_{3}=0 & \text{ on } \Gamma_{3}.
 \end{cases}
\end{align}
In order to prove the existence of
$\theta_{1},\theta_{2},\theta_{3}\in W^{2,p}(\Omega)$, we require
the following preliminary result whose proof follows from Theorem
$1$ in \cite{Dauge} (see also \cite{Dauge2}).
\begin{theorem} \label{tmadauge1}
If $F\in L^q(\Omega)$ with $\frac{6}{5}<q<\infty$, then the weak
solution to the problem
\begin{align}
 \begin{cases}
 -\Delta\omega=F  & \text{ in  } \Omega,\nonumber\\
 \frac{\partial\omega}{\partial\mathbf{n}}=0 & \text{ on } \Gamma_{1}\cup \Gamma_{2},\nonumber\\
 \omega=0 & \text{ on } \Gamma_{3},\nonumber
 \end{cases}
\end{align}
belongs to the space $W^{2,q}(\Omega)$.
\end{theorem}
Thus, by Theorem \ref{tmadauge1}, if $\hat{f}\in L^p(\Omega)$ with $\frac{6}{5}<p\leq\ 2$, then the system (\ref{problema3.1}) has solution $\theta_{1}\in W^{2,p}(\Omega)$. We remember that  $\hat{f}=-\theta\eta''-2\frac{\partial\theta}{\partial x_{3}}\eta'+\eta f+\Delta\Phi_{2}$. Observe that as $\theta \in H^1(\Omega)\hookrightarrow L^2(\Omega)$ then $\frac{\partial\theta}{\partial x_{3}} \in L^2(\Omega)$. Moreover since $\eta(x_{3})=\text{exp}[B(2x_{3}-\frac{1}{2}x_{3}^2)]$, then $\eta'=B(2-x_{3})\eta$ and $\eta''=-B\eta +B^{2}(2-x_{3})^{2}\eta$. Recalling that $0\leq x_{3}\leq 1$, we deduce that $\eta$, $\eta'$ y $\eta''$ belong to $L^2(\Omega)$. Finally, as $\Phi_{2}\in H^2(\Omega)$, $\Delta\Phi_{2} \in L^2(\Omega)$, and as by initial hypothesis $f\in L^p(\Omega)$ with $\frac{6}{5}<p\leq 2$, we conclude that $\hat{f}\in L^p(\Omega)$ with $\frac{6}{5}<p\leq 2$. Thus, the system (\ref{problema3.1}) has solution $\theta_{1}\in W^{2,p}(\Omega)$.\\

On the other hand, observe that for finding
$\theta_{2},\theta_{3}\in W^{2,p}(\Omega)$ solutions of
(\ref{problema3.2}) and (\ref{problema3.3}) respectively, we can not
use directly Theorem \ref{tmadauge1}, because these systems have
nonhomogeneous boundary conditions. Therefore, for solving the
problem (\ref{problema3.2}), we first divide $\Gamma_{2}$ in four
parts $\Gamma_{2}^{i}$ with $i=1,2,3,4$, as showed in Figure
\ref{fronteraomega2}, and then we decompose the solution
$\theta_{2}$ as the sum
$\theta_{2}=\theta_{2}^{1}+\theta_{2}^{2}+\theta_{2}^{3}+\theta_{2}^{4}$,
where $\theta_{2}^{i}$ ($i=1,2,3,4$) solve:
\begin{align} \label{problema3.2.i}
 \begin{cases}
 -\Delta\theta_{2}^{i}=0  & \text{ in  } \Omega,\\
 \frac{\partial\theta_{2}^{i}}{\partial\textbf{n}}=\hat{\phi}_{1}^{i} & \text{ on } \Gamma_{2}^{i},\\
 \frac{\partial\theta_{2}^{i}}{\partial\textbf{n}}=0 & \text{ on } \Gamma_{1}\cup (\Gamma_{2}\setminus\Gamma_{2}^{i}),\\
 \theta_{2}^{i}=0 & \text{ on } \Gamma_{3},
 \end{cases}
\end{align}
where $\hat{\phi}_{1}^{i}$ is defined by $\hat{\phi}_{1}^{i}= \left.
\hat{\phi}_{1} \right|_{\Gamma_{2}^{i}}$ on $\Gamma_{2}^{i}$, and
$\hat{\phi}_{1}^{i}=0$ on $\Gamma_{2}\backslash\Gamma_{2}^{i}$,
$i=1,2,3,4$.
\begin{center}
\psset{unit=0.60cm,linewidth=0.8pt}
\begin{pspicture}(0,-0.75)(6.5,6.0)
\pspolygon[linestyle=dashed,dash=5pt 5pt,linewidth=0.3pt](0,0)(4,0)(6.22,1.6)(2.22,1.6)
\pspolygon[linestyle=dashed,dash=5pt 5pt,linewidth=0.3pt](0,4)(4,4)(4,0)(0,0)
\pspolygon[linestyle=dashed,dash=5pt 5pt,linewidth=0.3pt](4,4)(6.22,5.6)(6.22,1.6)(4,0)
\pspolygon[linestyle=dashed,dash=5pt 5pt,linewidth=0.3pt](2.22,5.62)(6.22,5.6)(4,4)(0,4)
\psline[linestyle=dashed,dash=5pt 5pt](2.22,1.6)(2.22,5.6)
\pspolygon[linewidth=1.3pt,linecolor=black,fillcolor=gray,fillstyle=solid,opacity=0.3](2.22,1.6)(6.22,1.6)(6.22,5.6)(2.22,5.62)
\pspolygon[linewidth=1.3pt,linecolor=black,fillcolor=gray,fillstyle=solid,opacity=0.6](0,0)(0,4)(2.22,5.62)(2.22,1.6)
\rput[tl](4.6,3.8){$\Gamma_2^2$}
\rput[tl](0.9,3.2){$\Gamma_2^1$}
\end{pspicture}
\qquad\qquad\qquad
\begin{pspicture}(0,-0.75)(6.5,6.0)
\pspolygon[linestyle=dashed,dash=5pt 5pt,linewidth=0.3pt](0,0)(4,0)(6.22,1.6)(2.22,1.6)
\pspolygon[linewidth=1.3pt,linecolor=black,fillcolor=gray,fillstyle=solid,opacity=0.3](0,4)(4,4)(4,0)(0,0)
\pspolygon[linewidth=1.3pt,linecolor=black,fillcolor=gray,fillstyle=solid,opacity=0.6](4,4)(6.22,5.6)(6.22,1.6)(4,0)
\pspolygon[linestyle=dashed,dash=5pt 5pt,linewidth=0.3pt](2.22,5.62)(6.22,5.6)(4,4)(0,4)
\psline[linestyle=dashed,dash=5pt 5pt](2.22,1.6)(2.22,5.6)
\rput[tl](4.86,3.14){$\Gamma_2^3$}
\rput[tl](1.2,2.5){$\Gamma_2^4$}
\end{pspicture}
\\[-2mm]
\figcaption{Division of $\Gamma_{2}$.} \label{fronteraomega2}
\end{center}
For solving problems (\ref{problema3.2.i}), we will adapt the ideas
of \cite{Ziane}, Section 2. In the case $i=1$, we divide the
boundary of $\Gamma_{2}^{1}$, denoted by $\partial\Gamma_{2}^{1}$,
as follows:
$\partial\Gamma_{2}^{1}=\Gamma_{2}^{11}\cup\Gamma_{2}^{12}\cup\Gamma_{2}^{13}\cup\Gamma_{2}^{14}$
(see Figure \ref{fronteraomega3}), and we construct a function
$\psi_{1}$ as a solution of the heat equation:
\begin{align} \label{problema3.3.1.1}
 \begin{cases}
 \frac{\partial\psi_{1}}{\partial x_{2}}=\Delta\psi_{1}  & \text{ in  } \Gamma_{2}^1\times (0,\infty),\\
 \frac{\partial\psi_{1}}{\partial\textbf{n}}=0 & \text{ on } \Gamma_{2}^{1i}\times (0,\infty), \;\;i=1,2,4,\\
 \psi_{1}=0 & \text{ on } \Gamma_{2}^{13}\times (0,\infty),\\
 \psi_{1}(x_{1},0,x_{3})=\hat{\phi}_{1}(x_{1},0,x_{3}) & \text{ on } \Gamma_{2}^1.
 \end{cases}
\end{align}
\begin{center}
\psset{unit=0.60cm,linewidth=0.8pt}
\begin{pspicture}(0,-0.75)(6.5,6.0)
\pspolygon[linewidth=1.3pt,linecolor=black,fillcolor=gray,fillstyle=solid,opacity=0.3](0,0)(2.22,1.6)(2.22,5.6)(0,4)
\pspolygon[linestyle=dashed,dash=5pt 5pt,linewidth=0.3pt](0,0)(4,0)(6.22,1.6)(2.22,1.6)
\pspolygon[linestyle=dashed,dash=5pt 5pt,linewidth=0.3pt](0,4)(4,4)(4,0)(0,0)
\pspolygon[linestyle=dashed,dash=5pt 5pt,linewidth=0.3pt](2.22,5.62)(6.22,5.6)(4,4)(0,4)
\psline[linestyle=dashed,dash=5pt 5pt,linewidth=0.3pt](6.22,1.6)(6.22,5.6)
\uput[u](1,4.8){$\Gamma_2^{11}$}
\uput[r](2.1,3.0){$\Gamma_2^{12}$}
\uput[d](1.85,1.25){$\Gamma_2^{13}$}
\uput[l](0,2){$\Gamma_2^{14}$}
\end{pspicture}
\\[-4mm]
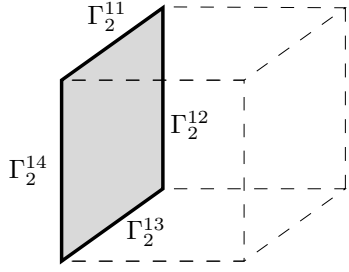
\figcaption{Division of $\partial\Gamma_{2}^{1}$.} \label{fronteraomega3}
\end{center}
By standard methods (cf. \cite{Feireisl}, Ch. 10) we can verify that there exists a solution
$\psi_{1}\in H^2(\Omega)$ for problem (\ref{problema3.3.1.1}).
Moreover, considering the following function
\begin{equation*} \label{problema3.3.1.2}
T_{1}(x_{1},x_{2},x_{3})=(1-x_{2}^{2})\overset{L}{\underset{x_{2}}{\int}}\psi_{1}(x_{1},z,x_{3})\;dz-x_{2}^{2}\frac{(L^{2}-1)}{2L}\psi_{1}(x_{1},L,x_{3}),\;\: (x_{1},x_{2},x_{3})\in \Gamma_{2}^1\times (0,\infty),
\end{equation*}
we can easily see that $T_{1}$ satisfies the boundary conditions in
(\ref{problema3.2.i}) (for $i=1$). Moreover, taking into account
that $\psi_{1}\in H^2(\Omega)$, we deduce that $T_{1}\in
H^2(\Omega)$. Additionally, as $T_{1}\in H^2(\Omega)$ then $-\Delta
T_{1}\in L^2(\Omega),$ and consequently $-\Delta T_{1}\in
L^p(\Omega)$ for $\frac{6}{5}<p\leq2$. Thus, by Theorem
\ref{tmadauge1}, the solution $\tilde{T}_{1}$ of the system
\begin{align*} \label{problema3.3.1.3}
 \begin{cases}
 \Delta\tilde{T}_{1}=-\Delta T_{1} & \text{ in } \Omega,\\
 \frac{\partial\tilde{T}_{1}}{\partial\textbf{n}}=0 & \text{ on }  \Gamma_{1}\cup \Gamma_{2},\\
 \tilde{T}_{1}=0 & \text{ on } \Gamma_{3},
 \end{cases}
\end{align*}
belongs to $W^{2,p}(\Omega)$. In conclusion, considering
$\theta_{2}^{1}=T_{1}+\tilde{T}_{1}$, we obtain that
$\theta_{2}^{1}\in W^{2,p}(\Omega)$ satisfies the system
(\ref{problema3.2.i}) for $i=1$. Analogously, we can find solutions
$\theta_{2}^{2}$, $\theta_{2}^{3}$, $\theta_{2}^{4}$ and
$\theta_{3}$ in $W^{2,p}(\Omega)$, for problems
(\ref{problema3.2.i}) (for $i=2,3,4$) and (\ref{problema3.3})
respectively. Thus, considering
$\theta_{2}=\theta_{2}^{1}+\theta_{2}^{2}+\theta_{2}^{3}+\theta_{2}^{4}$
we deduce that $\theta_{2}\in W^{2,p}(\Omega)$ is a solution to the
system (\ref{problema3.2}). Therefore, it was verified the existence
of $\theta_{1},\theta_{2},\theta_{3}\in W^{2,p}(\Omega)$ solutions
of (\ref{problema3.1}), (\ref{problema3.2}) and (\ref{problema3.3})
respectively, and the theorem is proven.
\end{proof}

Now, taking into account Theorem \ref{lemregularidad1}, we prove the
following theorem which guarantees the existence of solution of
problem (\ref{problema1}).
\begin{theorem}
Let $\phi_{1} \in H^{\frac{1}{2}}(\Gamma_{2})$, $\phi_{2} \in H^{3/2}_{00}(\Gamma_{3})$, $\mathbf{u}\in  \widetilde{\mathbf{X}}$ and $\theta \in H^1(\Omega)$ weak solution of system (\ref{problema1}). Then, the solution $\theta$ belongs to the space $H^2(\Omega)$.
\end{theorem}
\begin{proof}
First, observe that as $\textbf{u}\in \widetilde{\textbf{X}}\subset \textbf{H}^1(\Omega)$ then using Sobolev embeddings we obtain that $\textbf{u}\in \textbf{L}^6(\Omega)$. Moreover, as $\theta\in H^1(\Omega)$, $\nabla\theta\in L^2(\Omega)$ and consequently  $-(\textbf{u}\cdot\nabla) \theta\in L^{\frac{3}{2}}(\Omega)$. Thus, by Theorem \ref{lemregularidad1} we conclude that the problem (\ref{problema1}) has solution $\theta\in W^{2,\frac{3}{2}}(\Omega)$. Analogously, since $\theta\in W^{2,\frac{3}{2}}(\Omega)$, $\nabla\theta\in W^{1,\frac{3}{2}}(\Omega)$, and consequently, using the Sobolev embedding $W^{1,\frac{3}{2}}(\Omega) \hookrightarrow L^3(\Omega)$ we deduce that $\nabla\theta\in L^3(\Omega)$. Therefore, $-(\textbf{u}\cdot\nabla) \theta\in L^{2}(\Omega)$ and from Theorem  \ref{lemregularidad1}, we conclude that the solution $\theta$ of problem (\ref{problema1}) belongs to $H^2(\Omega)$.
\end{proof}
\begin{remark}
Taking into account that the geometry of $\Omega$ corresponds with a cube, we are able to obtain $\mathbf{H}^2$-regularity for the velocity $\mathbf{u}.$ For that, we can apply the results of $L^p$-regularity for the Stokes problem in polyhedral domains (see \cite{Dauge1,Kellogg,Majda,Majda2}).
\end{remark}

\section{Existence of Optimal Solutions}
\hspace{0.4cm}In this section we will prove the existence of an optimal solution
for Problem (\ref{eq:funcional}). We define the set of admissible
solutions of Problem (\ref{eq:funcional}) as follows:
\begin{eqnarray*} \label{conjsolad}
&\CS_{ad}:=&\!\!\!\!
\{\textbf{z}\equiv[\textbf{u},\theta,\mathbf{g},\phi_{1},\phi_{2}]\in
\widetilde{\textbf{X}} \times H^1(\Omega)
\times\textbf{U}_{1} \times \mathcal{U}_{2} \times\mathcal{U}_{3} \text{ such that }\nonumber\\
&& \CJ(\textbf{z})<\infty \text{ and  the equations (\ref{eq:020})-(\ref{eq:022}) hold}\}.
\end{eqnarray*}
Then, we have the following result:
\begin{theorem} \label{tmaexsolop}
Under the conditions of Theorem \ref{tmaexistenciasolucion2}, if one
of the conditions $(i)$ or $(ii)$ given in (\ref{eq:funcional}) is
satisfied, then the problem (\ref{eq:funcional}) has at least one
solution, that is, there exists at least a
$\hat{\mathbf{z}}\equiv[\hat{\mathbf{u}},\hat{\theta},\hat{\mathbf{g}},\hat{\phi}_{1},\hat{\phi}_{2}]\in
\CS_{ad}$ such that
\begin{equation*}
\CJ (\hat{\mathbf{z}})=\underset{\mathbf{z}\in \CS_{ad}}{\min}\CJ
(\mathbf{z}). \label{eq:solop1}
 \end{equation*}
\end{theorem}
\begin{proof}
From Theorem \ref{tmaexistenciasolucion2} we have that $\CS_{ad}$ is
nonempty. Denote by
$(\textbf{z}_{m})=([\textbf{u}_{m},\theta_{m},\textbf{g}_{m},\phi_{1_{m}},\phi_{2_{m}}])
\subset \CS_{ad}$, $m\in \mathbb{N}$, a minimizing sequence for
which
$\underset{m\rightarrow\infty}{\lim}\CJ(\textbf{z}_{m})=\underset{\textbf{z}\in
S_{ad}}{\min}\CJ (\textbf{z})$. If one of the conditions $(i)$ or
$(ii)$ is satisfied, then there exist constants $C_{1}$, $C_{2}$ and
$C_{3}$, independent of $m$, such that
$\|\mathbf{g}_{m}\|_{H^{\frac{1}{2}}(\Gamma^1_0)}\leq C_{1}$,
$\|\phi_{1_{m}}\|_{H^{\frac{1}{2}}({\Gamma_{0}\setminus\left\{x_{3}=0\right\}
})}\leq C_{2}$ and $\|\phi_{2_{m}}\|_{H^{\frac{1}{2}}(\{x_3=0\})}
\leq C_{3}$. Thus, from Theorem \ref{tmaexistenciasolucion2} we
conclude that there exist constants $C_{4}$ and $C_{5}$, independent
of $m$, such that $\left\Vert \mathbf{u}_{m}
\right\Vert_{H^1(\Omega)}\leq C_{4}$ and $\left\Vert \theta_{m}
\right\Vert_{H^1(\Omega)}\leq C_{5}$. Therefore, since
$\mathcal{\textbf{U}}_{1}$, $\mathcal{U}_{2}$ and $\mathcal{U}_{3}$
are closed convex subsets of
$\widetilde{\textbf{H}}^{1/2}_{00}(\Gamma_{0}^{1})$,
$H^{\frac{1}{2}}(\Gamma_{0}\setminus\left\{ x_{3}=0\right\})$ and
$H^{1/2}_{00}(\{ x_{3}=0\})$ respectively, we obtain
$\hat{\textbf{z}}\equiv[\hat{\textbf{u}},\hat{\theta},\hat{\textbf{g}},\hat{\phi}_{1},\hat{\phi}_{2}]\in
\mathbf{H}^1(\Omega) \times H^1(\Omega)\times\textbf{U}_{1} \times
\mathcal{U}_{2} \times\mathcal{U}_{3}$ such that, for some
subsequence of $(\textbf{z}_{m})_{m\in\mathbb{N}}\subset \CS_{ad}$
still denoted by $(\textbf{z}_{m})_{m\in\mathbb{N}},$ we have
\begin{eqnarray}
&&\textbf{u}_{m}\rightharpoonup \hat{\textbf{u}}  \ \mbox{in}\ \mathbf{H}^1(\Omega) \ \mbox{and
strongly in}\ \textbf{L}^p(\Omega),\ 2\leq p< 6, \label{convfu}\\
&&\theta_{m}\rightharpoonup \hat{\theta} \ \mbox{in}\ H^1(\Omega)\ \mbox{and
strongly in}\ L^l(\Omega),\ 2\leq l< 6,\nonumber\\
&&\textbf{g}_{m}\rightharpoonup \hat{\textbf{g}} \ \mbox{in}\ \textbf{H}^{\frac{1}{2}}(\Gamma_{0}^{1})\ \mbox{and
strongly in}\ \textbf{L}^{2}(\Gamma_{0}^{1}),\nonumber\\
&&\phi_{1_{m}}\rightharpoonup \hat{\phi}_{1} \ \mbox{in}\  H^{\frac{1}{2}}(\Gamma_{0}\setminus\left\{ x_{3}=0\right\})\ \mbox{and
strongly in}\ L^{2}(\Gamma_{0}\setminus\left\{ x_{3}=0\right\}),\nonumber\\
&&\phi_{2_{m}}\rightharpoonup \hat{\phi}_{2}\ \mbox{in}\  H^{\frac{1}{2}}(\{ x_{3}=0 \})\ \mbox{and
strongly in}\ L^{2}(\{ x_{3}=0 \}).\nonumber
\end{eqnarray}
Since $\left.\mathbf{u}_{m} \right|_{\Gamma^1_{0}} =
\textbf{g}_{m}$, $\left.\mathbf{u}_{m} \right|_{\Gamma^2_{0}} =
\textbf{u}^0$ and $\left.\theta_{m} \right|_{\{x_{3}=0\}} =
\phi_{2_{m}}$, it follows from the properties of the trace operators
that $\left.\hat{\mathbf{u}}
\right|_{\Gamma^1_{0}}=\hat{\textbf{g}}$, $\left.\hat{\mathbf{u}}
\right|_{\Gamma^2_{0}}=\textbf{u}^{0}$ and $\hat{\theta}
\mid_{\{x_{3}=0\}}=\hat{\phi}_{2}$; so, $\hat{\textbf{z}}$ satisfies
the boundary conditions (\ref{eq:022}). Moreover, since the third
component of $\textbf{u}_{m}$ denoted by $u_{m_{3}}$ is equal to $0$
on $\Gamma_{1}$ for all $m\in\mathbb{N}$, then from the continuity
of the trace operator we obtain $\hat{u}_{3}=0$ on $\Gamma_{1}$.
Also, using (\ref{convfu}) we obtain that $\text{div
}\textbf{u}_{m}\rightharpoonup \text{div }\hat{\textbf{u}}$ in
$\textbf{L}^2(\Omega)$, and given that $\text{div }\textbf{u}_{m}=0$
for all $m\in \mathbb{N}$, we conclude that $\text{div
}\hat{\textbf{u}}=0$. Moreover, as
$\hat{\textbf{u}}=\hat{\textbf{g}}$ on $\Gamma_{0}^{1}$ and
$\hat{\textbf{u}}=\textbf{u}^{0}$ on $\Gamma_{0}^{2}$, we obtain
that $\hat{\textbf{u}}\cdot \textbf{n}=0$ on
$\Gamma_{0}\setminus\left\{x_{3}=0\right\}$. Therefore, we conclude
that $\hat{\textbf{u}}\in \widetilde{\textbf{X}}$. A standard
procedure permits to pass the limit, as $m$ goes to $\infty,$ in the
variational formulation (\ref{eq:020})-(\ref{eq:021}), and we obtain
that $\hat{\textbf{z}}$ satisfies the weak formulation
(\ref{eq:020})-(\ref{eq:022}). Consequently we have that
$\hat{\textbf{z}}\equiv[\hat{\textbf{u}},\hat{\theta},\hat{\textbf{g}},\hat{\phi}_{1},\hat{\phi}_{2}]\in
\CS_{ad},$ and then
\begin{equation*} \label{eq:controloptimo1}
\CJ (\hat{\textbf{z}})\geq \underset{\textbf{z}\in
\CS_{ad}}{\text{inf}}\CJ (\textbf{z}).
\end{equation*}
Finally, recalling that the functional $\CJ$ is weakly lower
semicontinuous on $\CS_{ad}$, we have that
\begin{equation*} \label{eq:controloptimo2}
\CJ (\hat{\textbf{z}})= \underset{\textbf{z}\in
\CS_{ad}}{\text{min}}\CJ (\textbf{z}).
\end{equation*}
\end{proof}

\begin{remark}
Let $[\mathbf{u}_{b},\theta_{b}]$ the basic solution to the problem
(\ref{modest1})-(\ref{condfro3}) given by (\ref{estadobasico}). From
Theorem \ref{tmaexsolop}, we can obtain the existence of controls
$[\mathbf{g},\phi_{1},\phi_{2}]\in
\mathcal{\mathbf{U}}_{1}\times\mathcal{U}_{2}\times\mathcal{U}_{3}$
and a weak solution $[\mathbf{u},\theta]\in
\widetilde{\mathbf{X}}\times H^1(\Omega)$ of the problem
(\ref{eq:020})-(\ref{eq:022}), such that the functional
(\ref{eq:funcional}) is minimized if we consider
$\mathbf{u}_{d}=\mathbf{u}_{b}$ and $\theta_{d}=\theta_{b}$, the
basic state.
\end{remark}
\section{Necessary Optimality Conditions and an Optimality System}

\hspace{0.4cm}In order to obtain first-order optimality conditions, we start by
considering the following Banach spaces:
$\mathbb{G}=\widetilde{\textbf{X}}\times
H^1(\Omega)\times\mathcal{\textbf{U}}_{1} \times\mathcal{U}_{2}
\times\mathcal{U}_{3}$ and $\mathbb{H}=\mathbf{X}_0\times Y$, with
the usual inner products and norms. Moreover, if $\Gamma$ is a
connected subset of the boundary $\partial\Omega$, we define the
trace spaces
\begin{equation*}
H^{1/2}_{e}(\Gamma)=\{v\in L^{2}(\Gamma): \mbox{ there exists } g\in H^{\frac{1}{2}}(\partial\Omega), \ \left. g\right|_{\Gamma}=v \},
\end{equation*}
\begin{equation*}
\widetilde{\mathbf{H}}^{1/2}_{e}(\Gamma)=\left\{\mathbf{v}\in
\mathbf{L}^{2}(\Gamma):\exists\ \mathbf{g}\in
\mathbf{H}^{\frac{1}{2}}(\partial\Omega),\ \left.
\mathbf{g}\right|_{\Gamma}=\mathbf{v}, \int_{\Gamma} \mathbf{g}\cdot
\textbf{n} = 0, \ \mathbf{g}\cdot \textbf{n}=0 \ \mbox{on} \
\Gamma\setminus \{x_{3}=0\}, \ g_{3}=0 \ \mbox{on} \ \Gamma_{1}
\right\},
\end{equation*}
which are closed subspaces of $H^{1/2}(\Gamma)$ and ${\bf H}^{1/2}(\Gamma),$ respectively. Also, let $\mathbf{u}_{\mathbf{g}}^{0}$ defined by
\begin{align*}
\mathbf{u}_{\mathbf{g}}^{0}
 =
 \begin{cases}
 \mathbf{g}                             & \text{on  } \Gamma_{0}^{1},\\
 \mathbf{u}^{0}                             & \text{on  } \Gamma_{0}^{2}.\label{eq:hopf1}
 \end{cases}
\end{align*}
Then, taking into account that  $\mathbf{g}\in
\widetilde{\mathbf{H}}^{1/2}_{00}(\Gamma_{0}^{1})$ and
$\mathbf{u}^{0}\in
\widetilde{\mathbf{H}}^{1/2}_{00}(\Gamma_{0}^{2})$, we can easily
prove that $\mathbf{u}_{\mathbf{g}}^{0}\in
\widetilde{\mathbf{H}}^{1/2}_{00}(\Gamma_{0})$. Also, we consider
the following operators $\CF_{1}:\mathbb{G}\rightarrow
\textbf{X}'_{0}$, $\CF_{2}:\mathbb{G} \rightarrow Y'$,
$\CF_{3}:\mathbb{G} \rightarrow
\widetilde{\mathbf{H}}^{1/2}_{e}(\Gamma_{0})$ and
$\CF_{4}:\mathbb{G} \rightarrow H^{1/2}_{e}(\{x_{3}=0\})$, defined
at each point $\textbf{z}:=[\textbf{u},\theta,\textbf{g},\phi_{1},
\phi_{2}]$ by:
\begin{equation*}
\left\{
\begin{array}[c]{rcl}
\langle \CF_{1}(\textbf{z}),\textbf{v} \rangle &=& Pr\, a(\textbf{u},\textbf{v})+Pr\, M\, b_{1}(\theta,\textbf{v})+c(\textbf{u},\textbf{u},\textbf{v})-\langle f(\theta),\textbf{v} \rangle ,\;\forall\textbf{v}\in \textbf{X}_{0},\\
\langle \CF_{2}(\textbf{z}),W \rangle &=& c_{1}(\textbf{u},\theta,W)+a_{1}(\theta,W)+\left\langle B\theta,W\right\rangle _{\Gamma_{1}}-\left\langle \phi_{1},W\right\rangle _{\Gamma_{0}\backslash\{x_{3}=0\}},\;\forall W\in Y,\\
\CF_{3}(\textbf{z})&=&\left.\mathbf{u}\right|_{\Gamma_{0}} - \textbf{u}_{\textbf{g}}^{0},\\
\CF_{4}(\textbf{z})&=&\left. \theta \right|_{\{x_{3}=0\}} -
\phi_{2}.
\end{array}
\right.
\end{equation*}

In order to simplify the notation, let us denote by $\mathbb{M}$ the
space
\begin{equation*}
\mathbb{M}\equiv \textbf{X}'_{0}\times Y' \times \widetilde{\mathbf{H}}^{1/2}_{e}(\Gamma_{0}) \times H^{1/2}_{e}(\{x_{3}=0\}),
\end{equation*}
and define the operator
\begin{equation*}
\CFF:\mathbb{G}\rightarrow \mathbb{M},\;\text{ such that }
\CFF(\textbf{z}):=[\CF_{1}(\textbf{z}), \CF_{2}(\textbf{z}),
\CF_{3}(\textbf{z}), \CF_{4}(\textbf{z})].
\end{equation*}
Then the optimal control problem (\ref{eq:funcional}) is equivalent to:
\begin{eqnarray}
\left\{
\begin{array}[c]{l}
\text{Find } \textbf{z}:= [\textbf{u},\theta,\textbf{g},\phi_{1}, \phi_{2}]\in \mathbb{G} \text{ such that the functional}\\
\vspace{0,2 cm}\\
\CJ
[\textbf{u},\theta,\textbf{g},\phi_{1},\phi_{2}]=\frac{\gamma_1}{2}\|\text{rot
}
\mathbf{u}\|_{L^2(\Omega)}^2+\frac{\gamma_2}{2}\|\mathbf{u}-\mathbf{u}_d\|_{L^2(\Omega)}^2
+\frac{\gamma_3}{2}\|\theta-\theta_d\|_{L^2(\Omega)}^2 +\frac{\gamma_{4}}{2}\|\mathbf{g}\|^2_{H^{\frac{1}{2}}(\Gamma_{0}^{1})}\\
\hspace{3 cm}
+\frac{\gamma_{5}}{2}\|\phi_1\|^2_{H^{\frac{1}{2}}({\Gamma_{0}\setminus\left\{
x_{3}=0\right\} })}
+\frac{\gamma_6}{2}\|\phi_2\|^2_{H^{\frac{1}{2}}(\{x_3=0\})}, \vspace{0,4 cm} \\
\text{is minimized subject to } \langle
\CFF(\textbf{z}),[\textbf{v},W] \rangle=[\langle
\CF_{1}(\textbf{z}),\textbf{v} \rangle, \langle
\CF_{2}(\textbf{z}),W \rangle,
\CF_{3}(\textbf{z}),\CF_{4}(\textbf{z})]=[\textbf{0},0,\textbf{0},0].
\end{array}
\right. \label{eq:funcionalmul2}
\end{eqnarray}

\subsection{Existence of Lagrange Multipliers}
\hspace{0.4cm}In this subsection, we will prove the existence of Lagrange
multipliers. For that, first we will establish a regularity
condition for an optimal solution
$\hat{\mathbf{z}}\equiv[\hat{\mathbf{u}},\hat{\theta},\hat{\mathbf{g}},\hat{\phi}_{1},\hat{\phi}_{2}]\in
S_{ad},$ as was established in \cite{zowe}, p. 50. We follows the
ideas of \cite{delosReyes}. We start by establishing the following
two lemmas related to the Fr\'echet differentiability of $\CFF$ and
$\CJ.$
\begin{lemma}\label{lemmaderfref}
The operator $\CFF$ is Fr\'echet differentiable with respect to
$\mathbf{z}=[{\mathbf{u}},{\theta},{\mathbf{g}},{\phi_{1}},{\phi_{2}}]\in
\mathbb{G}$. Moreover, at an arbitrary point
$\hat{\mathbf{z}}=[\hat{\mathbf{u}},\hat{\theta},\hat{\mathbf{g}},\hat{\phi_{1}},\hat{\phi_{2}}]\in
\mathbb{G}$, the Fr\'echet derivative operator of
$\textbf{\textit{F}}$ with respect to $\mathbf{z}$ is the linear and
bounded operator
$\textbf{\textit{F}}_{\mathbf{z}}(\hat{\mathbf{z}}): \mathbb{G}
\rightarrow \mathbb{M}$ such that at each point
$\mathbf{t}=[\mathbf{h}_{1},h_{2},\mathbf{r},\varrho,\tau] \in
\mathbb{G}$, is defined by:
\begin{eqnarray}\label{frechet1a}
\left\{
\begin{array}[c]{l}
\langle \CF _{1\mathbf{z}}(\hat{\mathbf{z}})\mathbf{t},\mathbf{v} \rangle=Pr\, a(\mathbf{h}_1,\mathbf{v})+Pr\, M\, b_{1}(h_{2},\mathbf{v})+c(\hat{\mathbf{u}},\mathbf{h}_{1},\mathbf{v})+c(\mathbf{h}_{1},\hat{\mathbf{u}},\mathbf{v})-Pr R (h_{2},{v}_{3})_{L^{2}(\Omega)},\\
\langle \CF _{2\mathbf{z}}(\hat{\mathbf{z}})\mathbf{t},W \rangle=c_{1}(\hat{\mathbf{u}},h_{2},W)+c_{1}(\mathbf{h}_{1},\hat{\theta},W) +a_{1}(h_{2},W)+\left\langle Bh_{2},W\right\rangle _{\Gamma_{1}}-\left\langle \varrho,W\right\rangle _{\Gamma_{0}\backslash\{x_{3}=0\}},\\
\CF _{3\mathbf{z}}(\hat{\mathbf{z}})\mathbf{t}=\left.\mathbf{h}_{1} \right|_{\Gamma_{0}} -\mathcal{B}_1\mathbf{r},\\
\CF _{4\mathbf{z}}(\hat{\mathbf{z}})\mathbf{t}=\left. h_{2}
\right|_{\{x_{3}=0\}} -\tau,
\end{array}
\right.
\end{eqnarray}
for all $[\mathbf{v},W]\in \mathbb{H}$, where
$\mathcal{B}_1\in\mathcal{L}(\widetilde{\mathbf{H}}^{1/2}_{00}(\Gamma_{0}^{1}),\widetilde{\mathbf{H}}^{1/2}_{00}(\Gamma_0))$
is defined by
\begin{equation}\label{operB}
\mathcal{B}_1\mathbf{r}:=\left\{\begin{array}{lll}
\mathbf{r}&\mbox{ on }&\Gamma_0^1,\\
{\bf 0}&\mbox{ on }&\Gamma_0^2.
\end{array}\right.
\end{equation}
\end{lemma}

\begin{lemma}\label{lemmaderfrecj}
The functional $\CJ$ is Fr\'echet differentiable with respect to
$\mathbf{z}=[{\mathbf{u}},{\theta},{\mathbf{g}},{\phi_{1}},{\phi_{2}}]\in
\mathbb{G}$. Moreover, at an arbitrary point
$\hat{\mathbf{z}}=[\hat{\mathbf{u}},\hat{\theta},\hat{\mathbf{g}},\hat{\phi_{1}},\hat{\phi_{2}}]\in
\mathbb{G},$ the Fr\'echet derivative operator of $\CJ$ with respect
to $\mathbf{z}$ is the linear and bounded operator
$\CJ_{\mathbf{z}}(\hat{\mathbf{z}}): \mathbb{G}\rightarrow
\mathbb{R}$ such that at each point
$\mathbf{t}=[\mathbf{h}_{1},h_{2},\mathbf{r},\varrho,\tau] \in
\mathbb{G}$, is defined by:
\begin{eqnarray}
&\CJ_{\mathbf{z}}(\hat{\mathbf{z}})\mathbf{t}&\!\!\!\!\!=\gamma_{1}
(\mbox{rot } \hat{\mathbf{u}},\mbox{rot }
\mathbf{h}_{1})_{L^2(\Omega)}
+\gamma_{2}(\hat{\mathbf{u}}-\mathbf{u}_{d},\mathbf{h}_{1})_{L^2(\Omega)}+\gamma_{3}(\hat{\theta}-
\theta_{d}, h_{2})_{L^2(\Omega)}+\gamma_4(
\hat{\mathbf{g}},\mathbf{r})_{H^{\frac{1}{2}}(\Gamma^1_0)}\nonumber\\
&&+\gamma_5(
\hat{\phi}_{1},\varrho)_{H^{\frac{1}{2}}(\Gamma_{0}\setminus\left\{
x_{3}=0\right\}) }
+\gamma_6(\hat{\phi}_{2},\tau)_{H^{\frac{1}{2}}(\left\{
x_{3}=0\right\}) }.\label{derJ}
\end{eqnarray}
\end{lemma}

In the next lemma, we will give a condition to assure that $\hat{\bf
z}\in \mathcal{S}_{ad}$ satisfies the regular point condition (see
\cite{zowe}, p. 50). Thereafter the existence of Lagrange
multipliers is shown.
\begin{lemma}\label{lemma8}
Let
$\hat{\mathbf{z}}\equiv[\hat{\mathbf{u}},\hat{\theta},\hat{\mathbf{g}},\hat{\phi}_{1},\hat{\phi}_{2}]\in
\CS_{ad}$ be a feasible solution for the problem
(\ref{eq:funcionalmul2}). If $Pr$ is large enough and $M,R$ are
small enough such that
\begin{equation}\label{rpc1}
\beta_0:=\min\left\{Pr-C\left( Pr (M+R)+\Vert
\hat{\mathbf{u}}\Vert_{H^1(\Omega)}+\Vert
\hat{\theta}\Vert_{H^1(\Omega)}^{2}\right),\displaystyle\frac{1}{2}-C
Pr(R+M)\right\}>0,
\end{equation}
where $C$ is some positive
constant, which only depends on the domain $\Omega$, then $\hat{\bf
z}$ satisfies the regular point condition.
\end{lemma}
\begin{proof}
Given  $[{\bf a},{b},{\bf c},{d}]\in\mathbb{M}$, it is sufficient to
show the existence of
$\mathbf{t}=[\mathbf{h}_{1},h_{2},\mathbf{r},\varrho,\tau] \in
\mathbb{G}$ such that
\begin{eqnarray}\label{suf1}
Pr\, a(\mathbf{h}_1,\mathbf{v})+Pr\, M\, b_{1}(h_{2},\mathbf{v})+c(\hat{\mathbf{u}},\mathbf{h}_{1},\mathbf{v})+c(\mathbf{h}_{1},\hat{\mathbf{u}},\mathbf{v})-Pr R (h_{2},{v}_{3})_{L^{2}(\Omega)}=\langle {\bf a}, \mathbf{v} \rangle,\;\forall \mathbf{v}\in \mathbf{X}_{0},\\
c_{1}(\hat{\mathbf{u}},h_{2},W)+c_{1}(\mathbf{h}_{1},\hat{\theta},W) +a_{1}(h_{2},W)+\left\langle Bh_{2},W\right\rangle _{\Gamma_{1}}-\left\langle \varrho,W\right\rangle _{\Gamma_{0}\backslash\{x_{3}=0\}}=\langle b, W \rangle,\;\forall W\in Y,\label{suf2}\\
\mathbf{h}_{1}|_{\Gamma_{0}}={\bf c}+\mathcal{B}_1(\mathbf{r}-\hat{\mathbf{g}}),\label{suf3}\\
 h_{2}|_{\{x_{3}=0\}} = d+(\tau-\hat{\phi}_{2}).\label{suf4}
\end{eqnarray}
Setting
$[\mathbf{r},\varrho,\tau]=[\hat{\mathbf{g}},\hat{\phi}_{1},\hat{\phi}_{2}],$
we have that $\mathbf{h}_{1}|_{\Gamma_{0}}={\bf c}$ and
$h_{2}|_{\{x_{3}=0\}} = d.$ Then, proceeding as in the beginning
 of Subsection \ref{solweak}, we can prove that there exist $[\mathbf{h}^\epsilon_{1},h_2^\delta]\in
\widetilde{\mathbf{X}} \times H^{1}(\Omega)$ such that
$\left.\mathbf{h}^\epsilon_{1}\right|_{\Gamma_0}={\bf c}$ and
$\left. h_2^\delta\right|_{\{x_3=0\}}=d$. Therefore, rewriting the
unknowns $\mathbf{h}_1,h_2$ in the form
$\mathbf{h}_1=\mathbf{h}_1^\epsilon+\widetilde{\mathbf{h}}_1,$
$h_2=h_2^\delta+\widetilde{h}_2$ with
$[\widetilde{\mathbf{h}}_1,\widetilde{h}_2]\in \mathbb{H}$ new
unknown functions, from (\ref{suf1})-(\ref{suf4}), we obtain the
following linear system:
\begin{eqnarray}\label{suf1b}
Pr\, a(\widetilde{\mathbf{h}}_1,\mathbf{v})+Pr\, M\, b_{1}(\widetilde{h}_{2},\mathbf{v})+c(\hat{\mathbf{u}},\widetilde{\mathbf{h}}_{1},\mathbf{v})+c(\widetilde{\mathbf{h}}_{1},\hat{\mathbf{u}},\mathbf{v})-Pr R (\widetilde{h}_{2},{v}_{3})_{L^{2}(\Omega)}=\langle \widetilde{{\bf a}}, \mathbf{v} \rangle,\;\forall \mathbf{v}\in \mathbf{X}_{0},\\
c_{1}(\hat{\mathbf{u}},\widetilde{h}_{2},W)+c_{1}(\widetilde{\mathbf{h}}_{1},\hat{\theta},W)
+a_{1}(\widetilde{h}_{2},W)+\left\langle
B\widetilde{h}_{2},W\right\rangle _{\Gamma_{1}}=\langle
\widetilde{b}, W \rangle,\;\forall W\in Y,\label{suf2b}
\end{eqnarray}
where $$\langle \widetilde{{\bf a}}, \mathbf{v} \rangle=\langle {\bf
a}, \mathbf{v} \rangle-Pr\,
a(\mathbf{h}^{\epsilon}_1,\mathbf{v})-Pr\, M\,
b_{1}(h_{2}^\delta,\mathbf{v})-c(\hat{\mathbf{u}},\mathbf{h}^\epsilon_{1},\mathbf{v})-c(\mathbf{h}^\epsilon_{1},\hat{\mathbf{u}},\mathbf{v})+Pr
R (h_{2}^\delta,{v}_{3})_{L^{2}(\Omega)},$$
$$\langle
\widetilde{b},W \rangle=\langle b, W
\rangle-c_{1}(\hat{\mathbf{u}},h_{2}^\delta,W)-c_{1}(\mathbf{h}^\epsilon_{1},\hat{\theta},W)-a_{1}(h^\delta_{2},W)-\left\langle
B h^\delta_{2},W\right\rangle _{\Gamma_{1}}+\left\langle
\hat{\phi}_{1},W\right\rangle _{\Gamma_{0}\backslash\{x_{3}=0\}}.$$
In order to prove the existence of a solution for
(\ref{suf1b})-(\ref{suf2b}), we will apply the Lax-Milgram theorem.
For that, we consider the bilinear form
$A:\mathbb{H}\times\mathbb{H}\rightarrow \mathbb{R}$ defined by
\begin{eqnarray}\label{suf1c}
&A([\widetilde{\mathbf{h}}_1,\widetilde{h}_{2}],[\mathbf{v},W])&\!\!\!\!\!=
Pr\, a(\widetilde{\mathbf{h}}_1,\mathbf{v})+Pr\, M\, b_{1}(\widetilde{h}_{2},\mathbf{v})+c(\hat{\mathbf{u}},\widetilde{\mathbf{h}}_{1},\mathbf{v})+c(\widetilde{\mathbf{h}}_{1},\hat{\mathbf{u}},\mathbf{v})-Pr R (\widetilde{h}_{2},{v}_{3})_{L^{2}(\Omega)}\nonumber\\
&&+c_{1}(\hat{\mathbf{u}},\widetilde{h}_{2},W)+c_{1}(\widetilde{\mathbf{h}}_{1},\hat{\theta},W)
+a_{1}(\widetilde{h}_{2},W)+\left\langle
B\widetilde{h}_{2},W\right\rangle _{\Gamma_{1}},
\end{eqnarray}
and $I:\mathbb{H}\rightarrow \mathbb{R}$ defined by
$I[\mathbf{v},W]:= \langle \widetilde{{\bf a}}, \mathbf{v}
\rangle+\langle \widetilde{b},W \rangle$. Thus, we rewrite
(\ref{suf1b})-(\ref{suf2b}) as
\begin{eqnarray}\label{suf1d}
A([\widetilde{\mathbf{h}}_1,\widetilde{h}_{2}],[\mathbf{v},W])=I[\mathbf{v},W].
\end{eqnarray}
It is not difficult to prove that $A(\cdot,\cdot)$ is continuous and
$I\in \mathbb{H}'$. Now we prove the
$\mathbb{H}\times\mathbb{H}$-coercivity of $A$. For that, taking
$[\mathbf{v},W]=[\widetilde{\mathbf{h}}_1,\widetilde{h}_{2}]$ in
(\ref{suf1c}), and using the H\"older, Poincar\'e and Young
inequalities and Sobolev embeddings we get
\begin{eqnarray}
A([\widetilde{\mathbf{h}}_1,\widetilde{h}_{2}],[\widetilde{\mathbf{h}}_1,\widetilde{h}_{2}])\!\!\!\!\!&=&\!\!\!\!\!
Pr\Vert \nabla \widetilde{\mathbf{h}}_1\Vert^2_{L^{2}(\Omega)}+Pr\, M\, b_{1}(\widetilde{h}_{2},\widetilde{\mathbf{h}}_1)+c(\widetilde{\mathbf{h}}_{1},\hat{\mathbf{u}},\widetilde{\mathbf{h}}_1)-Pr R (\widetilde{h}_{2},\widetilde{h}_{1_{3}})_{L^{2}(\Omega)}\nonumber\\
&&\!\!\!\!\!+c_{1}(\widetilde{\mathbf{h}}_{1},\hat{\theta},\widetilde{h}_{2})
+\Vert \nabla\widetilde{h}_{2}\Vert^2_{L^{2}(\Omega)}+
B\Vert\widetilde{h}_{2}\Vert_{L^{2}(\Gamma_{1})}^{2}\nonumber\\
&\geq&\!\!\!\!\! \left(Pr-C Pr (M+R)-C\Vert \hat{\mathbf{u}}\Vert_{H^1(\Omega)}-C\Vert \hat{\theta}\Vert_{H^1(\Omega)}^{2}\right)\Vert \nabla\widetilde{\mathbf{h}}_1\Vert^2_{L^2(\Omega)}\nonumber\\
&&\!\!\!\!\!+\left(1-C Pr(R+M)-\displaystyle\frac{1}{2}\right)\Vert \nabla\widetilde{h}_{2}\Vert^2_{L^2(\Omega)}\nonumber\\
&\geq&\!\!\!\!\!\beta_0 \Vert
[\widetilde{\mathbf{h}}_1,\widetilde{h}_{2}]\Vert_{\mathbb{H}}^{2},\label{suf2v}
\end{eqnarray}
where $\beta_0=C\min\left\{Pr-C \left(Pr (M+R)+\Vert
\hat{\mathbf{u}}\Vert_{H^1(\Omega)}+\Vert
\hat{\theta}\Vert_{H^1(\Omega)}^{2}\right),\displaystyle\frac{1}{2}-C
Pr(R+M)\right\}>0$. Therefore, from (\ref{suf1d}) and (\ref{suf2v})
and the Lax-Milgram theorem we conclude the existence of
$[\widetilde{\mathbf{h}}_1,\widetilde{h}_{2}]\in\mathbb{H}$ solution
of (\ref{suf1b})-(\ref{suf2b}), and consequently, we obtain that
$[{\mathbf{h}}_1,{h}_{2}]\in \widetilde{\textbf{X}}\times
H^1(\Omega)$ is solution of (\ref{suf1})-(\ref{suf4}).
\end{proof}

In the next theorem, we will prove the existence of Lagrange
multipliers provided a local optimal solution
$\hat{\mathbf{z}}\equiv[\hat{\mathbf{u}},\hat{\theta},\hat{\mathbf{g}},\hat{\phi}_{1},\hat{\phi}_{2}]\in
S_{ad}$ verifies the regular point condition (see Lemma
\ref{lemma8}).
\begin{theorem}\label{ml1}
Let
$\hat{\mathbf{z}}\equiv[\hat{\mathbf{u}},\hat{\theta},\hat{\mathbf{g}},\hat{\phi}_{1},\hat{\phi}_{2}]\in
\CS_{ad}$ be a local optimal solution for the control problem
(\ref{eq:funcionalmul2}) and assume (\ref{rpc1}). Then, there exist
Lagrange multipliers $[{\boldsymbol
\lambda_1},\lambda_2,{\boldsymbol \lambda_3},\lambda_4]\in
{\mathbb{H}}\times(\widetilde{\mathbf{H}}^{1/2}_{e}(\Gamma_{0}))'\times(H^{1/2}_{e}(\{x_{3}=0\}))'$
such that for all $[\mathbf{h}_{1},h_{2},\mathbf{r},\varrho,\tau]
\in \widetilde{\mathbf{X}}\times H^1(\Omega)\times
\mathcal{C}(\hat{\mathbf{g}})\times
\mathcal{C}(\hat{\phi}_{1})\times \mathcal{C}(\hat{\phi}_{2})$ it
holds:
\begin{eqnarray}\label{lagmul}
\gamma_{1} (\text{rot } \hat{\mathbf{u}},\text{rot }
\mathbf{h}_{1})_{L^2(\Omega)}
\!+\!\gamma_{2}(\hat{\mathbf{u}}\!-\!\mathbf{u}_{d},\mathbf{h}_{1})_{L^2(\Omega)}\!+\!\gamma_{3}(\hat{\theta}-
\theta_{d}, h_{2})_{L^2(\Omega)}\!+\gamma_5(
\hat{\phi}_{1},\varrho)_{H^{\frac{1}{2}}(\Gamma_{0}\setminus\left\{
x_{3}=0\right\}) }\nonumber\\
+\gamma_4(
\hat{\mathbf{g}},\mathbf{r})_{H^{\frac{1}{2}}(\Gamma^1_0)}\!+\gamma_6(\hat{\phi}_{2},\tau)_{H^{\frac{1}{2}}(\left\{
x_{3}=0\right\})}\!-Pr\, a(\mathbf{h}_1,{\boldsymbol \lambda_1})-Pr\, M\, b_{1}(h_{2},{\boldsymbol \lambda_1})-c(\hat{\mathbf{u}},\mathbf{h}_{1},{\boldsymbol \lambda_1})-c(\mathbf{h}_{1},\hat{\mathbf{u}},{\boldsymbol \lambda_1})\nonumber\\
+Pr R (h_{2},{\lambda_{1_{3}}})_{L^{2}(\Omega)}
-c_{1}(\hat{\mathbf{u}},h_{2},\lambda_2)
-c_{1}(\mathbf{h}_{1},\hat{\theta},\lambda_2)-a_{1}(h_{2},\lambda_2)-\left\langle
Bh_{2},\lambda_2\right\rangle _{\Gamma_{1}}+\left\langle
\varrho,\lambda_2\right\rangle
_{\Gamma_{0}\backslash\{x_{3}=0\}}\!\nonumber\\
-\langle
{\boldsymbol\lambda_3},\mathbf{h}_{1}\!\mid_{\Gamma_0}\!-\mathcal{B}_1\mathbf{r}\rangle_{(\widetilde{\mathbf{H}}^{1/2}_{e}(\Gamma_{0}))',\widetilde{\mathbf{H}}^{1/2}_{e}(\Gamma_{0})}\!-\langle
\lambda_4,h_2\!\mid_{\{x_{3}=0\}}\!-\tau\rangle_{(H^{1/2}_{e}(\{x_{3}=0\}))',H^{1/2}_{e}(\{x_{3}=0\})}\geq
0,\ \ \ \ \
\end{eqnarray}
where $\mathcal{C}(\hat{\mathbf{g}})\times
\mathcal{C}(\hat{\phi}_{1})\times
\mathcal{C}(\hat{\phi}_{2})\!=\!\left\{\![\omega_1({\mathbf{g}}\!-\!\hat{\mathbf{g}}),\omega_2(\phi_1\!-\!\hat{\phi}_{1}),\omega_3(\phi_2\!-\!\hat{\phi}_{2})],\;
\omega_1,\omega_2,\omega_3\geq 0,\; [\mathbf{g},\phi_1,\phi_2]\in
\mathcal{\mathbf{U}}_{1}\! \times\mathcal{U}_{2}\!
\times\mathcal{U}_{3}\!\right\}\!.$
\end{theorem}

\begin{proof}
From Lemma \ref{lemma8}, $\hat{\bf z}\in \mathcal{S}_{ad}$ satisfies
the regular point condition. Then, there exist Lagrange multipliers
$[{\boldsymbol \lambda_1},\lambda_2,{\boldsymbol
\lambda_3},\lambda_4]\in
{\mathbb{H}}\times(\widetilde{\mathbf{H}}^{1/2}_{e}(\Gamma_{0}))'\times(H^{1/2}_{e}(\{x_{3}=0\}))'$
such that
\begin{eqnarray*}
&\CJ_{\mathbf{z}}(\hat{\mathbf{z}})\mathbf{h}&\!\!\!\!\! -\langle
\CF_{1\mathbf{z}}(\hat{\mathbf{z}})\mathbf{h},{\boldsymbol\lambda_{1}}
\rangle_{\mathbf{X}'_{0},\mathbf{X}_{0}}-\langle
\CF_{2\mathbf{z}}(\hat{\mathbf{z}})\mathbf{h},\lambda_{2}
\rangle_{Y',Y}\nonumber\\
&& -
\langle{\boldsymbol\lambda_{3}},\CF_{3\mathbf{z}}(\hat{\mathbf{z}})\mathbf{h}\rangle_{(\widetilde{\mathbf{H}}^{1/2}_{e}(\Gamma_{0}))',\widetilde{\mathbf{H}}^{1/2}_{e}(\Gamma_{0})}
-
\langle\lambda_{4},\CF_{4\mathbf{x}}(\hat{\mathbf{z}})\mathbf{h}\rangle_{(H^{1/2}_{e}(\{x_{3}=0\}))',H^{1/2}_{e}(\{x_{3}=0\})}\geq0,
\end{eqnarray*}
for all $[\mathbf{h}_{1},h_{2},\mathbf{r},\varrho,\tau] \in
\widetilde{\mathbf{X}}\times H^1(\Omega)\times
\mathcal{C}(\hat{\mathbf{g}})\times \mathcal{C}(\hat{\phi}_{1})\times
\mathcal{C}(\hat{\phi}_{2})$. Thus, the proof of theorem follows from
(\ref{frechet1a})-(\ref{derJ}).
\end{proof}

\subsection{Optimality System}

\hspace{0.4cm}In this subsection, we derive the equations that are satisfied by
the Lagrange multipliers ${\boldsymbol \eta}=[{\boldsymbol
\lambda_1},\lambda_2,{\boldsymbol \lambda_3},\lambda_4]$ provided by
Theorem \ref{ml1}.

\begin{theorem} (Adjoint
Equations) Let
$\hat{\mathbf{z}}\equiv[\hat{\mathbf{u}},\hat{\theta},\hat{\mathbf{g}},\hat{\phi}_{1},\hat{\phi}_{2}]\in
\CS_{ad}$ be an optimal solution for the control problem
(\ref{eq:funcionalmul2}) and assume (\ref{rpc1}). Then, there exist
functions (Lagrange multipliers) ${\boldsymbol \eta}=[{\boldsymbol
\lambda_1},\lambda_2,{\boldsymbol \lambda_3},\lambda_4]\in
{\mathbb{H}}\times(\widetilde{\mathbf{H}}^{1/2}_{e}(\Gamma_{0}))'\times(H^{1/2}_{e}(\{x_{3}=0\}))'$
which satisfy, in a variational sense, the following adjoint
equations to the control problem (\ref{eq:funcionalmul2}):
\begin{equation}\label{eqad}
\left\{
\begin{array}
[c]{lll}%
Pr \Delta {\boldsymbol
\lambda_{1}}+(\hat{\mathbf{u}}\cdot\nabla){\boldsymbol \lambda_{1}}
-\nabla^{T}\hat{\mathbf{u}}\cdot{\boldsymbol \lambda_{1}}
-\lambda_{2}\nabla\hat{\theta} =\gamma_{1} \text{rot}(\text{rot }
\hat{\mathbf{u}})
-\gamma_{2}(\hat{\mathbf{u}}\!-\!\mathbf{u}_{d}) \text{ in } \Omega,\\
\Delta \lambda_{2}+(\hat{\mathbf{u}}\cdot \nabla)\lambda_{2}+ Pr M
\text{div}\displaystyle\left(\frac{\partial {\boldsymbol
\lambda_{1}}}{\partial x_{3}}\right)+Pr R\;\lambda_{1_{3}}=
-\gamma_{3}(\hat{\theta}-\theta_{d}) \text{ in } \Omega,\\
\text{div}\:{\boldsymbol \lambda_{1}}=0 \text{ in } \Omega,\\
{\boldsymbol \lambda_{1}}=0 \text{ on } \Gamma_{0}, \qquad
\lambda_{1_{3}}=0 \text{ on } \Gamma_1, \qquad {\boldsymbol
\lambda_{3}}=\gamma_{1}(\text{rot }\hat{\mathbf{u}}\!\times\!
\mathbf{n})- Pr(\nabla{\boldsymbol \lambda_{1}}\cdot \mathbf{n})
\;\text{ on }
\Gamma_{0},\\
\gamma_{1}(\text{rot }\hat{\mathbf{u}}\!\times\! \mathbf{n})-
Pr(\nabla{\boldsymbol \lambda_{1}}\cdot \mathbf{n})=0 \;\text{ on }
\Gamma_{1}, \qquad B\lambda_{2}+\left(\nabla \lambda_{2} + Pr M
\displaystyle\frac{\partial{\boldsymbol \lambda_{1}}}{\partial
x_{3}}\right)\cdot \mathbf{n}=0 \;\text{ on } \Gamma_{1},
\\
\lambda_{4}=-\left(\nabla \lambda_{2} + Pr M
\displaystyle\frac{\partial{\boldsymbol \lambda_{1}}}{\partial
x_{3}}\right)\cdot \mathbf{n} \;\text{ on } \{x_{3}=0\},\qquad
\lambda_{2}=0 \text { on } \{x_3=0\},\\
\left(\nabla \lambda_{2} + Pr M
\displaystyle\frac{\partial{\boldsymbol \lambda_{1}}}{\partial
x_{3}}\right)\cdot \mathbf{n}=0 \;\text{ on }
\Gamma_{0}\setminus\left\{ x_{3}=0\right\}.
\end{array}
\right.
\end{equation}
\end{theorem}

\begin{proof} From (\ref{lagmul}) we obtain, taking
$[\mathbf{r},\varrho,\tau]=[{\boldsymbol 0},0,0]$, that for all
$[\mathbf{h}_{1},h_{2}]\in \widetilde{\mathbf{X}}\times
H^1(\Omega)$,
\begin{eqnarray}\label{sistopt1.1}
\gamma_{1} (\text{rot } \hat{\mathbf{u}},\text{rot }
\mathbf{h}_{1})_{L^2(\Omega)}
\!+\!\gamma_{2}(\hat{\mathbf{u}}\!-\!\mathbf{u}_{d},\mathbf{h}_{1})_{L^2(\Omega)}\!+\!\gamma_{3}(\hat{\theta}-
\theta_{d}, h_{2})_{L^2(\Omega)}-Pr\, a(\mathbf{h}_1,{\boldsymbol \lambda_1})-Pr  M\, b_{1}(h_{2},{\boldsymbol \lambda_1})\nonumber\\
-c(\hat{\mathbf{u}},\mathbf{h}_{1},{\boldsymbol
\lambda_1})-c(\mathbf{h}_{1},\hat{\mathbf{u}},{\boldsymbol
\lambda_1})+Pr R (h_{2},{\lambda_{1_{3}}})_{L^{2}(\Omega)}
-c_{1}(\hat{\mathbf{u}},h_{2},\lambda_2)
-c_{1}(\mathbf{h}_{1},\hat{\theta},\lambda_2)-a_{1}(h_{2},\lambda_2)\nonumber\\
-\left\langle Bh_{2},\lambda_2\right\rangle _{\Gamma_{1}} -\langle
{\boldsymbol\lambda_3},\mathbf{h}_{1}\!\mid_{\Gamma_0}\rangle_{(\widetilde{\mathbf{H}}^{1/2}_{e}(\Gamma_{0}))',\widetilde{\mathbf{H}}^{1/2}_{e}(\Gamma_{0})}\!-\langle
\lambda_4,h_2\!\mid_{\{x_{3}=0\}}\rangle_{(H^{1/2}_{e}(\{x_{3}=0\}))',H^{1/2}_{e}(\{x_{3}=0\})}=
0.\ \ \ \
\end{eqnarray}
Taking $h_{2}=0$ in (\ref{sistopt1.1}), we get \vspace{-0,2 cm}
\begin{eqnarray}\label{sistopt2.1}
\gamma_{1} (\text{rot } \hat{\mathbf{u}},\text{rot }
\mathbf{h}_{1})_{L^2(\Omega)}
\!+\!\gamma_{2}(\hat{\mathbf{u}}\!-\!\mathbf{u}_{d},\mathbf{h}_{1})_{L^2(\Omega)}
-Pr\, a(\mathbf{h}_1,{\boldsymbol
\lambda_1})-c(\hat{\mathbf{u}},\mathbf{h}_{1},{\boldsymbol
\lambda_1})\nonumber\\
-c(\mathbf{h}_{1},\hat{\mathbf{u}},{\boldsymbol
\lambda_1})-c_{1}(\textbf{h}_{1},\hat{\theta},\lambda_{2}) -\langle
{\boldsymbol
\lambda_{3}},\mathbf{h}_{1}\!\mid_{\Gamma_0}\rangle_{(\widetilde{\mathbf{H}}^{1/2}_{e}(\Gamma_{0}))',\widetilde{\mathbf{H}}^{1/2}_{e}(\Gamma_{0})}=0,\;\;
\forall \textbf{h}_{1}\in \widetilde{\textbf{X}},
\end{eqnarray}
and thus, using the Green formula, we obtain
\begin{eqnarray}\label{sistopt2.1a}
-\gamma_{1} \displaystyle\int_\Omega\text{rot(rot }
\hat{\mathbf{u}})\cdot \mathbf{h}_{1}\; d\Omega
+\gamma_{1}\int_{\partial\Omega}(\text{rot
}\hat{\textbf{u}}\!\times\! \textbf{n})\!\cdot\!\textbf{h}_{1}\;dS
+\gamma_{2}\int_\Omega(\hat{\mathbf{u}}\!-\!\mathbf{u}_{d})\cdot\mathbf{h}_{1}\;
d\Omega +Pr\, \int_\Omega \Delta{\boldsymbol
\lambda_1}\cdot\mathbf{h}_1\; d\Omega
\nonumber\\
-Pr\int_{\partial\Omega}\frac{\partial\lambda_{1}}{\partial\textbf{n}
}\cdot\textbf{h}_{1}\:
dS+\int_\Omega[(\hat{\mathbf{u}}\cdot\nabla){\boldsymbol
\lambda_1}]\cdot \mathbf{h}_{1}\;d\Omega -\int_\Omega
\nabla^{T}\hat{\mathbf{u}}\cdot{\boldsymbol \lambda_1}\cdot
\mathbf{h}_{1}\;
d\Omega-\int_\Omega \lambda_{2}\nabla\hat{\theta}\cdot\mathbf{h}_{1}\; d\Omega\nonumber\\
-\langle {\boldsymbol
\lambda_{3}},\mathbf{h}_{1}\!\mid_{\Gamma_0}\rangle_{(\widetilde{\mathbf{H}}^{1/2}_{e}(\Gamma_{0}))',\widetilde{\mathbf{H}}^{1/2}_{e}(\Gamma_{0})}=0,\;\;
\forall \textbf{h}_{1}\in \widetilde{\textbf{X}}.
\end{eqnarray}
Similarly, taking $\mathbf{h}_{1}=0$ in (\ref{sistopt1.1}), we get
\begin{eqnarray}\label{sistopt3.1}
\gamma_{3}(\hat{\theta}- \theta_{d}, h_{2})_{L^2(\Omega)}
-Pr M\;b_{1}(h_{2},{\boldsymbol\lambda_{1}})+Pr R\;(h_{2},\lambda_{1_{3}})_{L^2(\Omega)}-c_{1}(\hat{\textbf{u}},h_{2},\lambda_{2})\nonumber\\
-a_{1}(h_{2},\lambda_{2})-\langle B
h_{2},\lambda_{2}\rangle_{\Gamma_{1}}- \langle
\lambda_{4},h_2\!\mid_{\{x_{3}=0\}}\rangle_{(H^{1/2}_{e}(\{x_{3}=0\}))',H^{1/2}_{e}(\{x_{3}=0\})}=0,\;\;\forall
h_{2}\in H^1(\Omega),
\end{eqnarray}
and thus, using the Green formula, we obtain
\begin{eqnarray}\label{sistopt3.1a}
\gamma_{3}\!\displaystyle\int_\Omega\!(\hat{\theta}- \theta_{d})
h_{2}\; d\Omega +Pr M\!\int_\Omega\! \text{div}\left(\frac{\partial
{\boldsymbol \lambda_1}}{\partial x_{3}} \right) h_{2}\; d\Omega -
Pr M\!\int_{\partial\Omega}\!\left(\frac{\partial {\boldsymbol
\lambda_1}}{\partial x_{3}}\cdot \mathbf{n} \right) h_{2}\; dS
+Pr R\int_{\Omega}\lambda_{1_{3}} h_{2}\;d\Omega\nonumber\\
+\int_{\Omega}[(\hat{\textbf{u}}\cdot
\nabla)\lambda_{2}]h_{2}\;d\Omega +\int_{\Omega}\Delta \lambda_{2}\;
h_{2}\;d\Omega - \int_{\partial\Omega} \frac{\partial
\lambda_{2}}{\partial \mathbf{n}}\; h_{2}\;dS -
B\int_{\Gamma_{1}}\lambda_{2} h_{2}\; dS \nonumber\\- \langle
\lambda_{4},h_2\!\mid_{\{x_{3}=0\}}\rangle_{(H^{1/2}_{e}(\{x_{3}=0\}))',H^{1/2}_{e}(\{x_{3}=0\})}=0,\;\;\forall
h_{2}\in H^1(\Omega).
\end{eqnarray}
Observe that, if additionally we take the test functions
$\mathbf{h}_{1}\in \mathbf{V}$ in (\ref{sistopt2.1}) and $h_{2}\in
H^1_0(\Omega)$ in (\ref{sistopt3.1}), we get
\begin{eqnarray}\label{sistopt2.1b}
-\gamma_{1} \displaystyle\int_\Omega\text{rot(rot }
\hat{\mathbf{u}})\cdot \mathbf{h}_{1}\; d\Omega
+\gamma_{2}\int_\Omega(\hat{\mathbf{u}}\!-\!\mathbf{u}_{d})\cdot\mathbf{h}_{1}\;
d\Omega +Pr\, \int_\Omega \Delta{\boldsymbol
\lambda_1}\cdot\mathbf{h}_1\; d\Omega
\nonumber\\
+\int_\Omega[(\hat{\mathbf{u}}\cdot\nabla){\boldsymbol
\lambda_1}]\cdot \mathbf{h}_{1}\;d\Omega -\int_\Omega
\nabla^{T}\hat{\mathbf{u}}\cdot{\boldsymbol \lambda_1}\cdot
\mathbf{h}_{1}\; d\Omega-\int_\Omega
\lambda_{2}\nabla\hat{\theta}\cdot\mathbf{h}_{1}\; d\Omega=0,\;\;
\forall \textbf{h}_{1}\in \textbf{V},
\end{eqnarray}
\begin{eqnarray}\label{sistopt3.1b}
\gamma_{3}\!\displaystyle\int_\Omega\!(\hat{\theta}- \theta_{d})
h_{2}\; d\Omega +Pr M\!\int_\Omega\! \text{div}\left(\frac{\partial
{\boldsymbol \lambda_1}}{\partial x_{3}} \right) h_{2}\; d\Omega
+Pr R\int_{\Omega}\lambda_{1_{3}} h_{2}\;d\Omega\nonumber\\
+\int_{\Omega}[(\hat{\textbf{u}}\cdot
\nabla)\lambda_{2}]h_{2}\;d\Omega +\int_{\Omega}\Delta \lambda_{2}\;
h_{2}\;d\Omega =0,\;\;\forall h_{2}\in H^1_0(\Omega),
\end{eqnarray}
and therefore, since $[\mathbf{h}_{1},h_{2}] \in
\widetilde{\mathbf{X}}\times H^1(\Omega)$ is arbitrary, from
(\ref{sistopt2.1a}), (\ref{sistopt3.1a}), (\ref{sistopt2.1b}) and
(\ref{sistopt3.1b}), we deduce that $[{\boldsymbol
\lambda_1},\lambda_2,{\boldsymbol \lambda_3},\lambda_4]$ satisfy, in
a variational sense, the adjoint equations (\ref{eqad}).
\end{proof}

Finally, taking $[\mathbf{h}_{1},h_{2}]=[{\boldsymbol 0},0]$ in
(\ref{lagmul}), we obtain
\begin{eqnarray}\label{lagmul1a}
\gamma_4(
\hat{\mathbf{g}},\mathbf{r})_{H^{\frac{1}{2}}(\Gamma^1_0)}+\gamma_5(
\hat{\phi}_{1},\varrho)_{H^{\frac{1}{2}}(\Gamma_{0}\setminus\left\{
x_{3}=0\right\})
}+\gamma_6(\hat{\phi}_{2},\tau)_{H^{\frac{1}{2}}(\left\{
x_{3}=0\right\})} +\left\langle \varrho,\lambda_2\right\rangle
_{\Gamma_{0}\backslash\{x_{3}=0\}}\!\nonumber\\
+\langle
{\boldsymbol\lambda_3},\mathcal{B}_1\mathbf{r}\rangle_{(\widetilde{\mathbf{H}}^{1/2}_{e}(\Gamma_{0}))',\widetilde{\mathbf{H}}^{1/2}_{e}(\Gamma_{0})}\!+\langle
\lambda_4,\tau\rangle_{(H^{1/2}_{e}(\{x_{3}=0\}))',H^{1/2}_{e}(\{x_{3}=0\})}\geq
0,
\end{eqnarray}
for all $[\mathbf{r},\varrho,\tau] \in
\mathcal{C}(\hat{\mathbf{g}})\times
\mathcal{C}(\hat{\phi}_{1})\times \mathcal{C}(\hat{\phi}_{2})$.
Taking $\mathbf{r}=\mathbf{g}-\hat{\mathbf{g}}$, $\varrho=\phi_1 -
\hat{\phi}_1$ and $\tau=\phi_2 - \hat{\phi}_2$ in (\ref{lagmul1a}),
and recalling the definition of $\mathcal{B}_1$ given in
(\ref{operB}), we obtain
\begin{eqnarray*}
\langle\gamma_4
\hat{\mathbf{g}}+{\boldsymbol\lambda_3},\mathbf{g}-\hat{\mathbf{g}}\rangle
_{(H^{\frac{1}{2}}(\Gamma^1_0))',H^{\frac{1}{2}}(\Gamma^1_0)}+\langle\gamma_5
\hat{\phi}_{1}+\lambda_{2},\phi_1 -
\hat{\phi}_1\rangle_{(H^{\frac{1}{2}}(\Gamma_{0}\setminus\left\{
x_{3}=0\right\}))',H^{\frac{1}{2}}(\Gamma_{0}\setminus\left\{
x_{3}=0\right\})}\nonumber\\
+\langle\gamma_6\hat{\phi}_{2}+\lambda_4,\phi_2 -
\hat{\phi}_2\rangle_{(H^{\frac{1}{2}}(\left\{
x_{3}=0\right\}))',H^{\frac{1}{2}}(\left\{ x_{3}=0\right\})} \geq 0.
\end{eqnarray*}
Thus, the \textit{optimality conditions} are
\begin{eqnarray}
\langle\gamma_4
\hat{\mathbf{g}}+{\boldsymbol\lambda_3},\mathbf{g}-\hat{\mathbf{g}}\rangle
_{(H^{\frac{1}{2}}(\Gamma^1_0))',H^{\frac{1}{2}}(\Gamma^1_0)}&\geq&
0, \:\; \forall \mathbf{g}\in \mathcal{\mathbf{U}}_{1}, \label{lagmul1b1}\\
\langle\gamma_5 \hat{\phi}_{1}+\lambda_{2},\phi_1 -
\hat{\phi}_1\rangle_{(H^{\frac{1}{2}}(\Gamma_{0}\setminus\left\{
x_{3}=0\right\}))',H^{\frac{1}{2}}(\Gamma_{0}\setminus\left\{
x_{3}=0\right\})} &\geq& 0,\:\; \forall \phi_1\in \mathcal{U}_{2},\label{lagmul1b2}\\
\langle\gamma_6\hat{\phi}_{2}+\lambda_4,\phi_2 -
\hat{\phi}_2\rangle_{(H^{\frac{1}{2}}(\left\{
x_{3}=0\right\}))',H^{\frac{1}{2}}(\left\{ x_{3}=0\right\})}&\geq&
0, \:\; \forall \phi_2\in \mathcal{U}_{3}.\label{lagmul1b3}
\end{eqnarray}
Therefore, the state equations described in (\ref{model001}), the
adjoint equations given in (\ref{eqad}) and the optimality
conditions obtained in (\ref{lagmul1b1})-(\ref{lagmul1b3}), form the
optimality system of the optimal control problem
(\ref{eq:funcionalmul2}).
\begin{remark}
Following the ideas of \cite{delosReyes} we could suggest a
semi-smooth Newton method  applied to constrained boundary optimal
control of the RBM system. However, due to the lack of sufficient
regularity of the Lagrange multipliers for the pointwise control
constraint in the optimality system, a direct application of the
method to the infinite dimensional problem is not possible.
Therefore, following the ideas of \cite{delosReyes}, seems
reasonable to apply the semi-smooth Newton method to a
regularization of the original control problem, and finally to
analyze the convergence of the regularized solutions to the optimal
solution.
\end{remark}
%
%
%

\section{Second Order Sufficient Condition}
\hspace{0.4cm}In this section, we will analyze sufficient conditions for
$\hat{\mathbf{z}}=[\hat{\mathbf{u}},\hat{\theta},\hat{\mathbf{g}},\hat{\phi}_{1},\hat{\phi}_{2}]\in
S_{ad}$ be a local optimal solution. We will establish a
coercitivity condition on the second derivative of the Lagrangian
$\mathcal{L}$ in order to assure that an admissible point $\hat{\bf
z}$ is a local optimal solution. Here, we recall that
\begin{eqnarray*}
&\mathcal{L}(\textbf{z},{\boldsymbol \eta})& =
\CJ(\textbf{z})-\langle \CF_{1}(\textbf{z}),{\boldsymbol\lambda_1}
\rangle_{\textbf{X}'_{0},\textbf{X}_{0}}-\langle
\CF_{2}(\textbf{z}),\lambda_{2}
\rangle_{Y',Y} - \langle\mathbf{\boldsymbol\lambda_3},\CF_{3}(\textbf{z})\rangle_{(\widetilde{\mathbf{H}}^{1/2}_{e}(\Gamma_{0}))',\widetilde{\mathbf{H}}^{1/2}_{e}(\Gamma_{0})}\nonumber\\
&&\text{ }\;\; - \langle\lambda_{4},\CF_{4}(\textbf{z})\rangle_{(H^{1/2}_{e}(\{x_{3}=0\}))',H^{1/2}_{e}(\{x_{3}=0\})}.
\end{eqnarray*}
We have that the Lagrange multiplier ${\boldsymbol \eta}$ satisfies
$\mathcal{L}_{[\bf u,\theta]}([\hat{\bf z}, {\boldsymbol
\eta}])[\mathbf{h}_{1},h_{2}]=0$ for all $[\mathbf{h}_{1},h_{2}]\in
\widetilde{\mathbf{X}}\times H^1(\Omega)$, that is,
\begin{eqnarray}\label{eqq51b}
\gamma_{1} (\text{rot } \hat{\mathbf{u}},\text{rot }
\mathbf{h}_{1})_{L^2(\Omega)}
\!+\!\gamma_{2}(\hat{\mathbf{u}}\!-\!\mathbf{u}_{d},\mathbf{h}_{1})_{L^2(\Omega)}\!+\!\gamma_{3}(\hat{\theta}-
\theta_{d}, h_{2})_{L^2(\Omega)}-Pr\, a(\mathbf{h}_1,{\boldsymbol \lambda_1})-Pr\, M\, b_{1}(h_{2},{\boldsymbol \lambda_1})\nonumber\\
-c(\hat{\mathbf{u}},\mathbf{h}_{1},{\boldsymbol
\lambda_1})-c(\mathbf{h}_{1},\hat{\mathbf{u}},{\boldsymbol
\lambda_1})+Pr R (h_{2},{\lambda_{1_{3}}})_{L^{2}(\Omega)}
-c_{1}(\hat{\mathbf{u}},h_{2},\lambda_2)
-c_{1}(\mathbf{h}_{1},\hat{\theta},\lambda_2)-a_{1}(h_{2},\lambda_2)\nonumber\\
-\left\langle Bh_{2},\lambda_2\right\rangle _{\Gamma_{1}} -\langle
{\boldsymbol\lambda_3},\mathbf{h}_{1}\!\mid_{\Gamma_0}\rangle_{(\widetilde{\mathbf{H}}^{1/2}_{e}(\Gamma_{0}))',\widetilde{\mathbf{H}}^{1/2}_{e}(\Gamma_{0})}\!-\langle
\lambda_4,h_2\!\mid_{\{x_{3}=0\}}\rangle_{(H^{1/2}_{e}(\{x_{3}=0\}))',H^{1/2}_{e}(\{x_{3}=0\})}=
0,\ \ \ \
\end{eqnarray}
for all $[\mathbf{h}_{1},h_{2}]\in \widetilde{\mathbf{X}}\times
H^1(\Omega)$. In the next lemma we will establish a key estimate
which is verified by the Lagrange multipliers $[{\boldsymbol
\lambda_1},\lambda_2]\in \mathbb{H}.$
\begin{lemma}\label{est12}Let $\hat{\mathbf{z}}=[\hat{\mathbf{u}},\hat{\theta},\hat{\mathbf{g}},\hat{\phi}_{1},\hat{\phi}_{2}]$ an admissible point for the constrained optimal control problem (\ref{eq:funcional}) and assume (\ref{rpc1}). Then, the Lagrange multipliers $[{\boldsymbol\lambda_1},\lambda_2]\in \mathbb{H}$
satisfy
\begin{eqnarray}\label{estimulti}
\Vert
{\boldsymbol\lambda_1}\Vert^2_{H^1(\Omega)}+\Vert\lambda_2\Vert^2_{H^1(\Omega)}\leq
\frac{1}{\beta_0}C_1\mathcal{M}[\hat{\bf u},\hat{\theta}],
\end{eqnarray}
where $\mathcal{M}[\hat{\bf
u},\hat{\theta}]:=\displaystyle\frac{\gamma_1^2}{Pr}\Vert \hat{\bf
u}\Vert^2_{H^1(\Omega)}+\frac{\gamma_2^2}{Pr}\Vert \hat{\bf u}-{\bf
u}_d\Vert^2_{L^2(\Omega)}+{\gamma^2_3}\Vert
\hat{\theta}-\theta_d\Vert^2_{L^2(\Omega)}$ and $C_1$ is a positive
constant depending only on $\Omega.$
\end{lemma}
\begin{proof}
From (\ref{eqq51b}), setting $[{\bf
h_1},h_2]=[{\boldsymbol\lambda_1},\lambda_2]\in \mathbb{H},$ we have
\begin{eqnarray}
&Pr\Vert
\nabla{\boldsymbol\lambda_1}&\!\!\!\!\!\!\Vert^2_{L^{2}(\Omega)}+\Vert
\nabla
\lambda_2\Vert^2_{L^{2}(\Omega)}+B\Vert\lambda_2\Vert^{2}_{L^{2}(\Gamma_1)}=\gamma_{1}
(\text{rot } \hat{\mathbf{u}},\text{rot }
{\boldsymbol\lambda_1})_{L^2(\Omega)}
\!+\!\gamma_{2}(\hat{\mathbf{u}}\!-\!\mathbf{u}_{d},{\boldsymbol\lambda_1})_{L^2(\Omega)}\nonumber\\
&&+\!\gamma_{3}(\hat{\theta}- \theta_{d},
\lambda_2)_{L^2(\Omega)}-Pr\, M\, b_{1}(\lambda_2,{\boldsymbol
\lambda_1}) -c(\hat{\mathbf{u}},{\boldsymbol\lambda_1},{\boldsymbol
\lambda_1})-c({\boldsymbol\lambda_1},\hat{\mathbf{u}},{\boldsymbol
\lambda_1})\nonumber\\
&&+Pr R (\lambda_2,{\lambda_{1_{3}}})_{L^{2}(\Omega)}
-c_{1}(\hat{\mathbf{u}},\lambda_2,\lambda_2)
-c_{1}({\boldsymbol\lambda_1},\hat{\theta},\lambda_2).\label{sf1}
\end{eqnarray}
Then, by using the H\"older, Poincar\'e and Young inequalities and
Sobolev embeddings, from (\ref{sf1}) we get
\begin{eqnarray*}
&&Pr\Vert\nabla{\boldsymbol\lambda_1}\Vert_{L^2(\Omega)}^2+\Vert
\nabla\lambda_2\Vert_{L^2(\Omega)}^2\leq \frac{C\gamma_1^2}{Pr}\Vert
\hat{\bf u}\Vert^2_{H^1(\Omega)}+\frac{C\gamma_2^2}{Pr}\Vert
\hat{\bf u}-{\bf
u}_d\Vert^2_{L^2(\Omega)}+\frac{Pr}{2}\Vert\nabla{\boldsymbol\lambda_1}\Vert_{L^2(\Omega)}^2\\
&&\ \ \ +C Pr (R+M)(\Vert
\nabla{\boldsymbol\lambda_1}\Vert^2_{L^2(\Omega)}+\Vert
\nabla\lambda_2\Vert^2_{L^2(\Omega)})+C\Vert
\nabla{\boldsymbol\lambda_1}\Vert^2_{L^2(\Omega)}\Vert  \hat{\bf
u}\Vert_{H^1(\Omega)}+{C\gamma^2_3}\Vert
\hat{\theta}-\theta_d\Vert^2_{L^2(\Omega)}\\
&& \ \ \ +C\Vert \nabla
{\boldsymbol\lambda_1}\Vert^2_{L^2(\Omega)}\Vert
\hat{\theta}\Vert^2_{H^1(\Omega)}+\frac{1}{2}\Vert
\nabla\lambda_2\Vert_{L^2(\Omega)}^2,
\end{eqnarray*}
where $C$ only depends on $\Omega.$ Thus, we can get
\begin{eqnarray*}
&&\left(Pr-C\left(Pr(M+R)+\Vert \hat{\bf u}\Vert_{H^1(\Omega)}+\Vert
\hat{\theta}\Vert^2_{H^1(\Omega)}\right)\right)\Vert \nabla
{\boldsymbol\lambda_1}\Vert^2_{L^2(\Omega)}\\
&& \ \ \ \ \ +\left(\frac{1}{2}-C
 Pr(R+M)\right)\Vert \nabla\lambda_2\Vert^2_{L^2(\Omega)}\leq C\left(\frac{\gamma_1^2}{Pr}\Vert
\hat{\bf u}\Vert^2_{H^1(\Omega)}+\frac{\gamma_2^2}{Pr}\Vert \hat{\bf
u}-{\bf u}_d\Vert^2_{L^2(\Omega)}+{\gamma^2_3}\Vert
\hat{\theta}-\theta_d\Vert^2_{L^2(\Omega)}\right).
\end{eqnarray*}
Then, since by hypothesis $\beta_0\!=\!\min\!\left\{\!\displaystyle
Pr-C\!\left(Pr(M+R)\!+\!\Vert \hat{\bf
u}\Vert_{H^1(\Omega)}\!+\!\Vert
\hat{\theta}\Vert^2_{H^1(\Omega)}\right),\frac{1}{2}-C
Pr(R+M)\right\}\!>\!0$, using the Poincar\'e inequality we conclude
(\ref{estimulti}).
\end{proof}
\begin{theorem}
Let
$\hat{\mathbf{z}}=[\hat{\mathbf{u}},\hat{\theta},\hat{\mathbf{g}},\hat{\phi}_{1},\hat{\phi}_{2}]$
an admissible point for the constrained optimal control problem
(\ref{eq:funcional}) and assume (\ref{rpc1}). If $\frac{C^2
C_1}{\beta_0 \Lambda^2}\mathcal{M}[\hat{\bf u},\hat{\theta}]<1$,
where $\Lambda:= \frac{C
\min\{\gamma_4,\gamma_5,\gamma_6\}\beta_0^2}{(1+\beta_0)^2}$, and
$\mathcal{M}[\hat{\bf u},\hat{\theta}]$, $C_1$ are given in Lemma
\ref{est12}, then there exists $K_0>0$ such that
\begin{eqnarray}\label{coer}
\mathcal{L}_{{\bf z} {\bf z}}[\hat{\bf z}, {\boldsymbol\eta}][{\bf
t}, {\bf t}] \geq K_0 \|{\bf t}\|^2_{\mathbb{G}},
\end{eqnarray}
for all ${\bf t}\in
\ker(\textbf{\textit{F}}_{\mathbf{z}}(\hat{\mathbf{z}})).$
Consequently, the point $\hat{\bf z}$ is a local optimal solution.
\end{theorem}
\begin{proof}
Let ${\bf t}=[\mathbf{h}_{1},h_{2},\mathbf{r},\varrho,\tau]\in
\mathbb{G}.$ Then, the second derivative of the Lagrangian
$\mathcal{L}$, with respect to ${\bf z}$ at the point $[\hat{\bf
z},{\boldsymbol\eta}]$ in all directions $[{\bf t},{\bf t}]$, is
given by
\begin{eqnarray}
\mathcal{L}_{{\bf z}{\bf z}}[\hat{\bf z},{\boldsymbol\eta}][{\bf t},{\bf t}]&=&\gamma_1\Vert {\rm rot}\ {\bf h}_1\Vert^2_{L^2(\Omega)}+\gamma_2\Vert {\bf h}_1\Vert^2_{L^2(\Omega)}+\gamma_3\Vert {h}_2\Vert^2_{L^2(\Omega)}+\gamma_4\Vert {\bf{r}}\Vert^2_{H^{\frac{1}{2}}(\Gamma^1_0)}+\gamma_5\Vert \varrho\Vert^2_{H^{\frac{1}{2}}(\Gamma_0\setminus\{x_3=0\})}\nonumber\\
&&+\gamma_6\Vert \tau\Vert^2_{H^{\frac{1}{2}}(\{x_3=0\})}-2 c({\bf
h}_1,{\bf h}_1,\lambda_1)-2 c_1({\bf
h}_1,{h}_2,\lambda_2).\label{suf54}
\end{eqnarray}
Thus, by using the H\"older and Young inequalities, we bound
(\ref{suf54}) as follows
\begin{eqnarray}
\mathcal{L}_{{\bf z}{\bf z}}[\hat{\bf z},\hat{\eta}][{\bf t},{\bf t}]&\geq& \gamma_1\Vert {\rm rot}\ {\bf h}_1\Vert^2_{L^2(\Omega)}+\gamma_2\Vert {\bf h}_1\Vert^2_{L^2(\Omega)}+\gamma_3\Vert {h}_2\Vert^2_{L^2(\Omega)}+\gamma_4\Vert {\bf{r}}\Vert^2_{H^{\frac{1}{2}}(\Gamma^1_0)}+\gamma_5\Vert \varrho\Vert^2_{H^{\frac{1}{2}}(\Gamma_0\setminus\{x_3=0\})}\nonumber\\
&&+\gamma_6\Vert \tau\Vert^2_{H^{\frac{1}{2}}(\{x_3=0\})}-C(\Vert
\lambda_1\Vert_{H^1(\Omega)}+\Vert
\lambda_2\Vert_{H^1(\Omega)})\Vert [{\bf
h}_1,h_{2}]\Vert^2_{\widetilde{\mathbf{X}}\times
H^{1}(\Omega)}.\label{yc1}
\end{eqnarray}
If ${\bf t}\in
\ker(\textbf{\textit{F}}_{\mathbf{z}}(\hat{\mathbf{z}})),$ from
(\ref{frechet1a})-(\ref{operB}) we obtain
\begin{eqnarray}\label{suf1g}
Pr\, a(\mathbf{h}_1,\mathbf{v})+Pr\, M\, b_{1}(h_{2},\mathbf{v})+c(\hat{\mathbf{u}},\mathbf{h}_{1},\mathbf{v})+c(\mathbf{h}_{1},\hat{\mathbf{u}},\mathbf{v})-Pr R (h_{2},{v}_{3})_{L^{2}(\Omega)}=0,\;\forall \mathbf{v}\in \mathbf{X}_{0},\\
c_{1}(\hat{\mathbf{u}},h_{2},W)+c_{1}(\mathbf{h}_{1},\hat{\theta},W) +a_{1}(h_{2},W)+\left\langle Bh_{2},W\right\rangle _{\Gamma_{1}}-\left\langle \varrho,W\right\rangle _{\Gamma_{0}\backslash\{x_{3}=0\}}=0,\;\forall W\in Y,\label{suf2g}\\
\mathbf{h}_{1}|_{\Gamma_{0}}=\mathcal{B}_1 \mathbf{r},\\
h_{2}|_{\{x_{3}=0\}} = \tau.\label{suf4g}
\end{eqnarray}
Proceeding as in the proof of Lemma \ref{lemma8}, we can prove that
there exist $[\mathbf{h}^\epsilon_{1},h_2^\delta]\in
\widetilde{\mathbf{X}} \times H^{1}(\Omega)$ such that
$\left.\mathbf{h}^\epsilon_{1}\right|_{\Gamma_0}=\mathcal{B}_1\mathbf{r}$,
$\left. h_2^\delta\right|_{\{x_3=0\}}=\tau$, and the estimates in
(\ref{thetadeltaest}) remain true for $\theta_\delta=h_2^\delta$,
$\phi_{1}=\varrho$ and $\phi_{2}=\tau$. Therefore, rewriting the
unknowns $\mathbf{h}_1,h_2$ in the form
$\mathbf{h}_1=\mathbf{h}_1^\epsilon+\widetilde{\mathbf{h}}_1,$
$h_2=h_2^\delta+\widetilde{h}_2$ with
$[\widetilde{\mathbf{h}}_1,\widetilde{h}_2]\in \mathbb{H}$ new
unknown functions, from (\ref{suf1g})-(\ref{suf4g}), we obtain
\begin{eqnarray}\label{suf1w}
A([\widetilde{\mathbf{h}}_1,\widetilde{h}_{2}],[\mathbf{v},W])=\bar{I}[\mathbf{v},W],
\end{eqnarray}
where $A$ is the bilinear form defined in (\ref{suf1c}) and
$\bar{I}:\mathbb{H}\rightarrow \mathbb{R}$ is defined by
$\bar{I}[\mathbf{v},W]:= \langle \bar{{\bf a}}, \mathbf{v}
\rangle+\langle \bar{b},W \rangle$ with
$$\langle \bar{{\bf a}}, \mathbf{v} \rangle=-Pr\,
a(\mathbf{h}^{\epsilon}_1,\mathbf{v})-Pr  M\,
b_{1}(h_{2}^\delta,\mathbf{v})-c(\hat{\mathbf{u}},\mathbf{h}^\epsilon_{1},\mathbf{v})-c(\mathbf{h}^\epsilon_{1},\hat{\mathbf{u}},\mathbf{v})+Pr
R (h_{2}^\delta,{v}_{3})_{L^{2}(\Omega)},$$
$$\langle
\bar{b},W
\rangle=-c_{1}(\hat{\mathbf{u}},h_{2}^\delta,W)-c_{1}(\mathbf{h}^\epsilon_{1},\hat{\theta},W)-a_{1}(h^\delta_{2},W)-\left\langle
B h^\delta_{2},W\right\rangle _{\Gamma_{1}}+\left\langle
\varrho,W\right\rangle _{\Gamma_{0}\backslash\{x_{3}=0\}}.$$
Moreover, arguing as in the proof of Lemma \ref{lemma8}, from
(\ref{suf1w}) we can get there exists $C>0$, depending only on $Pr$,
$R$, $M$, $B$, $\Vert [\hat{\bf
u},\hat{\theta}]\Vert_{\widetilde{\mathbf{X}}\times H^{1}(\Omega)}$
and $\Omega$, such that
$$\Vert
[\widetilde{\mathbf{h}}_1,\widetilde{h}_{2}]\Vert_{\mathbb{H}}\leq
\frac{C}{\beta_0}\left(\Vert
[\mathbf{h}^\epsilon_{1},h_2^\delta]\Vert_{\widetilde{\mathbf{X}}
\times H^{1}(\Omega)}+\Vert
\varrho\Vert_{H^{\frac{1}{2}}(\Gamma_0\setminus\{x_3=0\})}\right),$$
with $\beta_0$ defined in (\ref{rpc1}), and consequently,
\begin{eqnarray}
\Vert [{\mathbf{h}}_1,{h}_{2}]\Vert_{\widetilde{\mathbf{X}} \times H^{1}(\Omega)} &\leq & \Vert [\widetilde{\mathbf{h}}_1,\widetilde{h}_{2}]\Vert_{\mathbb{H}}+\Vert [\mathbf{h}^\epsilon_{1},h_2^\delta]\Vert_{\widetilde{\mathbf{X}} \times H^{1}(\Omega)}\nonumber\\
&\leq &C\left(\frac{1}{\beta_0}+1\right)\left(\Vert
\mathbf{r}\Vert_{H^{\frac{1}{2}}(\Gamma^1_0)}+\Vert
\varrho\Vert_{H^{\frac{1}{2}}(\Gamma_0\setminus\{x_3=0\})}+\Vert
\tau\Vert_{H^{\frac{1}{2}}(\{x_3=0\})}\right).\label{sufic9}
\end{eqnarray}
Thus, from (\ref{yc1}) and (\ref{sufic9}) we get
\begin{eqnarray}
\mathcal{L}_{{\bf z}{\bf z}}[\hat{\bf z},{\boldsymbol \eta}][{\bf t},{\bf t}]&\geq& \gamma_1\Vert {\rm rot}\ {\bf h}_1\Vert^2_{L^2(\Omega)}+\gamma_2\Vert {\bf h}_1\Vert^2_{L^2(\Omega)}+\gamma_3\Vert {h}_2\Vert^2_{L^2(\Omega)}+\frac{\gamma_4}{2}\Vert {\bf{r}}\Vert^2_{H^{\frac{1}{2}}(\Gamma^1_0)}+\frac{\gamma_5}{2}\Vert \varrho\Vert^2_{H^{\frac{1}{2}}(\Gamma_0\setminus\{x_3=0\})}\nonumber\\
&&+\frac{\gamma_6}{2}\Vert \tau\Vert^2_{H^{\frac{1}{2}}(\{x_3=0\})}+C\frac{\min\{\gamma_{4},\gamma_5,\gamma_6\}\beta_0^2}{(1+\beta_0)^2}\Vert [{\mathbf{h}}_1,{h}_{2}]\Vert^{2}_{\widetilde{\mathbf{X}} \times H^{1}(\Omega)}\nonumber\\
&&-C(\Vert \lambda_1\Vert^2_{H^1(\Omega)}+\Vert \lambda_2\Vert^2_{H^1(\Omega)})\Vert [{\mathbf{h}}_1,{h}_{2}]\Vert^{2}_{\widetilde{\mathbf{X}} \times H^{1}(\Omega)}\nonumber\\
&\geq&\frac{\gamma_4}{2}\Vert {\bf{r}}\Vert^2_{H^{\frac{1}{2}}(\Gamma^1_0)}+\frac{\gamma_5}{2}\Vert \varrho\Vert^2_{H^{\frac{1}{2}}(\Gamma_0\setminus\{x_3=0\})}+ \frac{\gamma_6}{2}\Vert \tau\Vert^2_{H^{\frac{1}{2}}(\{x_3=0\})}\nonumber\\
&&+\left(\Lambda-C(\Vert \lambda_1\Vert^2_{H^1(\Omega)}+\Vert
\lambda_2\Vert^2_{H^1(\Omega)})\right)\Vert
[{\mathbf{h}}_1,{h}_{2}]\Vert^{2}_{\widetilde{\mathbf{X}} \times
H^{1}(\Omega)}.\label{sufic10}
\end{eqnarray}
Therefore, by using estimate (\ref{estimulti}) in Lemma \ref{est12},
from (\ref{sufic10}) we have
\begin{eqnarray*}
\mathcal{L}_{{\bf z}{\bf z}}[\hat{\bf z},{\boldsymbol\eta}][{\bf t},{\bf t}]&\geq& \frac{\gamma_4}{2}\Vert {\bf{r}}\Vert^2_{H^{\frac{1}{2}}(\Gamma^1_0)}+\frac{\gamma_5}{2}\Vert \varrho\Vert^2_{H^{\frac{1}{2}}(\Gamma_0\setminus\{x_3=0\})}+ \frac{\gamma_6}{2}\Vert \tau\Vert^2_{H^{\frac{1}{2}}(\{x_3=0\})}\nonumber\\
&&+\left(\Lambda-C\left(\frac{C_1}{\beta_0}\mathcal{M}[\hat{\bf
u},\hat{\theta}]\right)^{1/2}\right)\Vert
[{\mathbf{h}}_1,{h}_{2}]\Vert^{2}_{\widetilde{\mathbf{X}} \times
H^{1}(\Omega)}.
\end{eqnarray*}
Then, since by hypothesis
$C\left(\frac{C_1}{\beta_0}\mathcal{M}[\hat{\bf
u},\hat{\theta}]\right)^{1/2}<\Lambda$, we deduce that
$\Upsilon:=\Lambda - C\left(\frac{C_1}{\beta_0}\mathcal{M}[\hat{\bf
u},\hat{\theta}]\right)^{1/2}>0$, and consequently,
\begin{eqnarray*}
\mathcal{L}_{{\bf z} {\bf z}}[\hat{\bf z}, {\boldsymbol\eta}][{\bf
t}, {\bf t}] \geq  \min
\left\{\Upsilon,\frac{\gamma_4}{2},\frac{\gamma_5}{2},\frac{\gamma_6}{2}
\right\}\|{\bf t}\|^2_{\mathbb{G}}.
\end{eqnarray*}
Thus, we conclude the coercitivity condition (\ref{coer}). Taking in
particular ${\bf
t}=[\mathbf{h}_{1},h_{2},\mathbf{r},\varrho,\tau]\in
\ker(\textbf{\textit{F}}_{\mathbf{z}}(\hat{\mathbf{z}}))$ with
$[\mathbf{r},\varrho,\tau] \in \mathcal{C}(\hat{\mathbf{g}})\times
\mathcal{C}(\hat{\phi}_{1})\times \mathcal{C}(\hat{\phi}_{2})$, we
obtain that the point $\hat{\bf z}$ is a local optimal solution (cf.
\cite{Maur}).
\end{proof}

\section{Uniqueness of Optimal Solution}
\hspace{0.4cm}In this section we will establish a result related to the uniqueness
of the optimal solution of problem (\ref{eq:funcional}). For that,
suppose that there exist $\hat{\textbf{z}}_i
=[\hat{\textbf{u}}_i,\hat{\theta}_i,\hat{\mathbf{g}}_i,\hat{\phi}_{1}^i,\hat{\phi}_{2}^{i}]\in
\CS_{ad}$ $(i=1,2)$ optimal solutions of problem
(\ref{eq:funcional}), and let ${\boldsymbol \eta}_i=[{\boldsymbol
\lambda}_1^i,\lambda_2^i, {\boldsymbol \lambda}_3^i,\lambda_4^i]$
two Lagrange multipliers corresponding to the solutions
$\hat{\textbf{z}}_i$ satisfying the relations (\ref{sistopt2.1}),
(\ref{sistopt3.1}) and (\ref{lagmul1b1})-(\ref{lagmul1b3}). Let us
denote $\hat{\textbf{u}}=\hat{\textbf{u}}_1-\hat{\textbf{u}}_2$,
$\hat{\theta}=\hat{\theta}_1-\hat{\theta}_2$,
$\hat{\mathbf{g}}=\hat{\mathbf{g}}_1-\hat{\mathbf{g}}_2$,
$\hat{\phi}_{1}=\hat{\phi}_{1}^1-\hat{\phi}_{1}^2$,
$\hat{\phi}_{2}=\hat{\phi}_{2}^1-\hat{\phi}_{2}^2$, ${\boldsymbol
\lambda}_1={\boldsymbol \lambda}_1^1-{\boldsymbol \lambda}_1^2$,
$\lambda_2=\lambda_2^1-\lambda_2^2$, ${\boldsymbol
\lambda}_3={\boldsymbol \lambda}_3^1-{\boldsymbol \lambda}_3^2$ and
$\lambda_4=\lambda_4^1-\lambda_4^2$. Then, taking into account that
$\hat{\textbf{z}}_1$ and $\hat{\textbf{z}}_2$ satisfy the equations
(\ref{eq:020})-(\ref{eq:022}), we deduce that $\hat{\textbf{z}}
=[\hat{\textbf{u}},\hat{\theta},\hat{\mathbf{g}},\hat{\phi}_{1},\hat{\phi}_{2}]$
satisfies
\begin{equation}
Pr\, a(\hat{\mathbf{u}},\mathbf{v})+Pr\, M\,
b_{1}(\hat{\theta},\mathbf{v})+c(\hat{\mathbf{u}}_1,\hat{\mathbf{u}},\mathbf{v})+c(\hat{\mathbf{u}},\hat{\mathbf{u}}_2,\mathbf{v})=\int_\Omega
Pr R \hat{\theta} v_3 ,\;\:\forall\mathbf{v}\in
\mathbf{X}_{0},\label{uniq1a}
\end{equation}
\begin{equation}
c_{1}(\hat{\mathbf{u}}_1,\hat{\theta},W)+c_{1}(\hat{\mathbf{u}},\hat{\theta}_2,W)+a_{1}(\hat{\theta},W)+\left\langle
B\hat{\theta},W\right\rangle _{\Gamma_{1}}=\left\langle
\hat{\phi}_{1},W\right\rangle
_{\Gamma_{0}\backslash\{x_{3}=0\}},\;\:\forall W\in Y,
\label{uniq1b}
\end{equation}
\begin{equation}
\hat{\mathbf{u}}=\hat{\mathbf{g}}\: \text{ on } \:\Gamma_{0}^{1}, \
\ \hat{\mathbf{u}}={\boldsymbol 0}\: \text{ on } \:\Gamma_{0}^{2} \
\ \text{ and }\quad\hat{\theta}=\hat{\phi}_{2}\: \text{ on }
\:\{x_{3}=0\}. \label{uniq1c}
\end{equation}
Proceeding as in the beginning of Subsection \ref{solweak}, we can
prove that there exist $[\mathbf{u}_\epsilon,\theta_\delta]\in
\widetilde{\mathbf{X}} \times H^{1}(\Omega)$ such that
$\left.\mathbf{u}_\epsilon\right|_{\Gamma_0^1}=\hat{\mathbf{g}}$,
$\left.\mathbf{u}_\epsilon\right|_{\partial\Omega\setminus\Gamma_0^1}={\boldsymbol
0}$, $\left. \theta_\delta\right|_{\{x_3=0\}}=\hat{\phi}_2$, the
estimates in (\ref{thetadeltaest}) remain true for
$\phi_{1}=\hat{\phi}_1$ and $\phi_{2}=\hat{\phi}_2$, and $\Vert
\mathbf{u}_\epsilon\Vert_{H^{1}(\Omega)}\leq C\Vert
\hat{\mathbf{g}}\Vert_{H^{\frac{1}{2}}(\Gamma_0^1)}$. Therefore,
rewriting $[\hat{\mathbf{u}},\hat{\theta}]\in \widetilde{\mathbf{X}}
\times H^{1}(\Omega)$ in the form
$\hat{\mathbf{u}}=\mathbf{u}+\mathbf{u}_\epsilon$ and
$\hat{\theta}=\theta+\theta_\delta$, from
(\ref{uniq1a})-(\ref{uniq1c}) we obtain that $[\mathbf{u},\theta]\in
\mathbf{X}_0\times Y$ satisfies
\begin{eqnarray}
&& Pr\, a(\textbf{u},\mathbf{v})+Pr\, M\, b_{1}(\theta,\mathbf{v})+c(\hat{\mathbf{u}}_1,\mathbf{u},\mathbf{v})+c(\mathbf{u},\hat{\mathbf{u}}_2,\mathbf{v})=\int_\Omega Pr R \theta v_3 +\int_\Omega Pr R \theta_\delta v_3 \nonumber\\
&&\hspace{0.5cm}-Pr\, a(\textbf{u}_{\epsilon},\mathbf{v})-Pr\, M\, b_{1}(\theta_{\delta},\mathbf{v})-c(\mathbf{u}_{\epsilon},\hat{\mathbf{u}}_2,\mathbf{v})-c(\hat{\mathbf{u}}_1,\mathbf{u}_{\epsilon},\mathbf{v}) ,\;\:\forall\mathbf{v}\in \mathbf{X}_{0},\label{uniq2a}\\
&&c_{1}(\hat{\mathbf{u}}_1,\theta,W)+c_{1}(\mathbf{u},\hat{\theta}_2,W)+a_{1}(\theta,W)+\left\langle B\theta,W\right\rangle _{\Gamma_{1}}=\left\langle \hat{\phi}_{1},W\right\rangle _{\Gamma_{0}\backslash\{x_{3}=0\}}\nonumber\\
&&\hspace{0.5cm}-c_{1}(\mathbf{u}_{\varepsilon},\hat{\theta}_2,W)-
c_{1}(\hat{\mathbf{u}}_1,\theta_\delta,W)-a_{1}({\theta_\delta},W)-\left\langle
B{\theta_\delta},W\right\rangle _{\Gamma_{1}}\:\forall W\in Y.
\label{uniq2b}
\end{eqnarray}
Setting $\mathbf{v}=\mathbf{u}$ in (\ref{uniq2a}), $W=\theta$ in
(\ref{uniq2b}), and using the H\"older inequality, Sobolev
embeddings and the Poincar\'e inequality, we get
\begin{eqnarray}\label{uniq3a}
&Pr \Vert \nabla \mathbf{u}\Vert_{L^{2}(\Omega)}\leq&\!\!\!\!
C\left(\Vert \nabla \mathbf{u}\Vert_{L^{2}(\Omega)}\Vert
\hat{\mathbf{u}}_2\Vert_{H^{1}(\Omega)}+\left(Pr+\Vert
\hat{\mathbf{u}}_1\Vert_{H^{1}(\Omega)}+\Vert
\hat{\mathbf{u}}_2\Vert_{H^{1}(\Omega)}\right)\Vert
\mathbf{u}_\epsilon\Vert_{H^{1}(\Omega)}\right)\nonumber\\
&&\!\!\!\! + C (Pr R +Pr M)\left(\Vert \nabla
\theta\Vert_{L^{2}(\Omega)}+\Vert \theta_\delta
\Vert_{H^{1}(\Omega)} \right),
\end{eqnarray}

\vspace{-0.8 cm}

\begin{eqnarray}\label{uniq3b}
&\Vert \nabla \theta\Vert_{L^{2}(\Omega)}\leq&\!\!\!\! C \left(
\Vert \hat{\phi}_{1}
\Vert_{H^{\frac{1}{2}}(\Gamma_{0}\backslash\{x_{3}=0\})}+\left(\Vert
\nabla \mathbf{u}\Vert_{L^{2}(\Omega)}+\Vert
\mathbf{u}_\epsilon\Vert_{H^{1}(\Omega)}\right)\Vert
\hat{\theta}_2\Vert_{H^{1}(\Omega)}\right)\nonumber\\
&&\!\!\!\! + C \left(1+\Vert
\hat{\mathbf{u}}_1\Vert_{H^{1}(\Omega)}+B \right)\Vert \theta_\delta
\Vert_{H^{1}(\Omega)}.
\end{eqnarray}
From inequality (\ref{eq:estutheta}) we have that $\Vert \hat{\mathbf{u}}_i\Vert_{H^1(\Omega)}+ \Vert
\hat{\theta}_i\Vert_{H^1(\Omega)} \leq C S_i,$ $i=1,2,$ where
\begin{eqnarray}\label{hel1}
S_i=C\left(\Vert \mathbf{u}^{0}
\Vert_{H^{\frac{1}{2}}(\Gamma_{0}^{2})}+\Vert \hat{\mathbf{g}}_i
\Vert_{H^{\frac{1}{2}}(\Gamma_{0}^{1})}+\Vert \hat{\phi}^i_{1}
\Vert_{H^{\frac{1}{2}}(\Gamma_{0}\backslash\{x_{3}=0\})}+\Vert
\hat{\phi}^i_{2}\Vert_{H^{\frac{1}{2}}(\{x_{3}=0\})}\right).
\end{eqnarray}
Thus, from (\ref{uniq3a}), (\ref{uniq3b})  and (\ref{hel1}) we obtain
\begin{eqnarray}
Pr \Vert \nabla \mathbf{u}\Vert_{L^{2}(\Omega)}\!\!\!&\leq&\!\!\!
C((S_2+Pr(R+M)S_2)\Vert \nabla
\mathbf{u}\Vert_{L^{2}(\Omega)}+(Pr+S_1+S_2+Pr(R+M)S_2)\Vert
\mathbf{u}_\epsilon\Vert_{H^{1}(\Omega)})\nonumber\\
&&\!\!\!+C Pr(R+M)(1+S_1+B)\Vert \theta_\delta
\Vert_{H^{1}(\Omega)}+C Pr(R+M)\Vert \hat{\phi}_{1}
\Vert_{H^{\frac{1}{2}}(\Gamma_{0}\backslash\{x_{3}=0\})}\nonumber\\
\!\!\!&\leq&\!\!\! C((S_2+Pr(R+M)S_2)\Vert \nabla
\mathbf{u}\Vert_{L^{2}(\Omega)}+(Pr+S_1+S_2+Pr(R+M)S_2)\Vert
\hat{\mathbf{g}}\Vert_{H^{\frac{1}{2}}(\Gamma_0^1)})\nonumber\\
&&\!\!\!+C Pr(R+M)(1+S_1+B)(\Vert \hat{\phi}_{2}
\Vert_{H^{\frac{1}{2}}(\{x_{3}=0\})}+\Vert \hat{\phi}_{1}
\Vert_{H^{\frac{1}{2}}(\Gamma_{0}\backslash\{x_{3}=0\})}).\label{hel2}
\end{eqnarray}
Taking $Pr$ large enough and $M,R$ small enough, from (\ref{hel2})
we get
\begin{eqnarray}\label{hel3a}
\Vert \nabla \mathbf{u}\Vert_{L^{2}(\Omega)}\leq \mathcal{H}_0\left(\Vert
\hat{\mathbf{g}}\Vert_{H^{\frac{1}{2}}(\Gamma_0^1)}+\Vert
\hat{\phi}_{2} \Vert_{H^{\frac{1}{2}}(\{x_{3}=0\})}+\Vert
\hat{\phi}_{1}
\Vert_{H^{\frac{1}{2}}(\Gamma_{0}\backslash\{x_{3}=0\})}\right),
\end{eqnarray}
where $\mathcal{H}_0=C(Pr+S_1+S_2+B+1)(1+Pr(R+M))/Pr$,  and therefore
\begin{eqnarray}\label{hel3}
&\Vert  \hat{\mathbf{u}}\Vert_{H^{1}(\Omega)}&\!\!\!\!\leq \Vert
\mathbf{u}\Vert_{H^{1}(\Omega)}+\Vert
\mathbf{u}_\epsilon\Vert_{H^{1}(\Omega)}\leq C\left( \Vert \nabla
\mathbf{u}\Vert_{L^{2}(\Omega)}+\Vert
\hat{\mathbf{g}}\Vert_{H^{\frac{1}{2}}(\Gamma_0^1)}\right) \nonumber\\
&& \!\!\!\!\leq C(\mathcal{H}_0+1)\left(\Vert
\hat{\mathbf{g}}\Vert_{H^{\frac{1}{2}}(\Gamma_0^1)}+\Vert
\hat{\phi}_{2} \Vert_{H^{\frac{1}{2}}(\{x_{3}=0\})}+\Vert
\hat{\phi}_{1}
\Vert_{H^{\frac{1}{2}}(\Gamma_{0}\backslash\{x_{3}=0\})}\right).
\end{eqnarray}
In the same spirit, from (\ref{uniq3b}) and (\ref{hel3a}) we can
obtain
\begin{eqnarray*}
\Vert \nabla \theta\Vert_{L^{2}(\Omega)}\leq \mathcal{H}_1\left(\Vert
\hat{\mathbf{g}}\Vert_{H^{\frac{1}{2}}(\Gamma_0^1)}+\Vert
\hat{\phi}_{2} \Vert_{H^{\frac{1}{2}}(\{x_{3}=0\})}+\Vert
\hat{\phi}_{1}
\Vert_{H^{\frac{1}{2}}(\Gamma_{0}\backslash\{x_{3}=0\})}\right),
\end{eqnarray*}
where $\mathcal{H}_1=C(S_1+S_2+B+1+S_2\mathcal{H}_0),$ and thus
 \begin{eqnarray}\label{hel4}
\Vert  \hat{\theta}\Vert_{H^{1}(\Omega)}\leq C(\mathcal{H}_1+1)\left(\Vert
\hat{\mathbf{g}}\Vert_{H^{\frac{1}{2}}(\Gamma_0^1)}+\Vert
\hat{\phi}_{2} \Vert_{H^{\frac{1}{2}}(\{x_{3}=0\})}+\Vert
\hat{\phi}_{1}
\Vert_{H^{\frac{1}{2}}(\Gamma_{0}\backslash\{x_{3}=0\})}\right).
\end{eqnarray}

On the other hand, subtracting equations (\ref{sistopt2.1}) written for
$\hat{\textbf{u}}_i,{\boldsymbol
\lambda}_1^i,\lambda_2^i,\hat{\theta}_i, {\boldsymbol \lambda}_3^i,
\ i=1,2,$ we have
\begin{eqnarray}\label{uniq3c}
\gamma_{1} (\text{rot } \hat{\mathbf{u}},\text{rot }
\mathbf{h}_{1})_{L^2(\Omega)}
\!+\!\gamma_{2}(\hat{\mathbf{u}},\mathbf{h}_{1})_{L^2(\Omega)}
-Pr\, a(\mathbf{h}_1,{\boldsymbol
\lambda_1})-c(\hat{\mathbf{u}}_1,\mathbf{h}_{1},{\boldsymbol
\lambda_1})
-c(\hat{\mathbf{u}},\mathbf{h}_{1},{\boldsymbol
\lambda_1^2})-c(\textbf{h}_{1},\hat{\mathbf{u}}_1,{\boldsymbol
\lambda_1})\nonumber\\
-c(\textbf{h}_{1},\hat{\mathbf{u}},{\boldsymbol
\lambda}_1^2)-c_1(\textbf{h}_{1},\hat{\theta}_1,\lambda_2)-c_1(\textbf{h}_{1},\hat{\theta},\lambda_2^2)-\langle
{\boldsymbol
\lambda_{3}},\mathbf{h}_{1}\!\mid_{\Gamma_0}\rangle_{(\widetilde{\mathbf{H}}^{1/2}_{e}(\Gamma_{0}))',\widetilde{\mathbf{H}}^{1/2}_{e}(\Gamma_{0})}=0,\;\;
\forall \textbf{h}_{1}\in \widetilde{\textbf{X}}.
\end{eqnarray}
Taking $\textbf{h}_{1}=\hat{\mathbf{u}}$ in (\ref{uniq3c}) we obtain
\begin{eqnarray}\label{uniq3d}
-Pr\, a(\hat{\mathbf{u}},{\boldsymbol
\lambda_1})-c(\hat{\mathbf{u}}_1,\hat{\mathbf{u}},{\boldsymbol
\lambda_1}) -2 c(\hat{\mathbf{u}},\hat{\mathbf{u}},{\boldsymbol
\lambda_1^2})-c(\hat{\mathbf{u}},\hat{\mathbf{u}}_1,{\boldsymbol
\lambda_1})
-c_1(\hat{\mathbf{u}},\hat{\theta}_1,\lambda_2)\nonumber\\-c_1(\hat{\mathbf{u}},\hat{\theta},\lambda_2^2)-\langle
{\boldsymbol
\lambda_{3}},\hat{\mathbf{u}}\!\mid_{\Gamma_0}\rangle_{(\widetilde{\mathbf{H}}^{1/2}_{e}(\Gamma_{0}))',\widetilde{\mathbf{H}}^{1/2}_{e}(\Gamma_{0})}=-\gamma_{1}\Vert
\text{rot } \hat{\mathbf{u}}\Vert^2_{L^2(\Omega)}
-\gamma_{2}\Vert\hat{\mathbf{u}}\Vert^2_{L^2(\Omega)}.
\end{eqnarray}
Now, subtracting equations (\ref{sistopt3.1})  written for
$\hat{\textbf{u}}_i,{\boldsymbol
\lambda}_1^i,\lambda_2^i,\hat{\theta}_i, \lambda_4^i, \ i=1,2,$ we
get
\begin{eqnarray}\label{uniq3e}
\gamma_{3}(\hat{\theta}, h_{2})_{L^2(\Omega)}
-Pr M\;b_{1}(h_{2},{\boldsymbol\lambda_{1}})+Pr R\;(h_{2},\lambda_{1_{3}})_{L^2(\Omega)}-c_{1}(\hat{\textbf{u}}_1,h_{2},\lambda_{2})-c_{1}(\hat{\textbf{u}},h_{2},\lambda_{2}^2)\nonumber\\
-a_{1}(h_{2},\lambda_{2})-\langle B
h_{2},\lambda_{2}\rangle_{\Gamma_{1}}- \langle
\lambda_{4},h_2\!\mid_{\{x_{3}=0\}}\rangle_{(H^{1/2}_{e}(\{x_{3}=0\}))',H^{1/2}_{e}(\{x_{3}=0\})}=0,\;\;\forall
h_{2}\in H^1(\Omega).
\end{eqnarray}
Taking $h_{2}=\hat{\theta}$ in (\ref{uniq3e}) we have
\begin{eqnarray}\label{uniq3f}
-a_{1}(\hat{\theta},\lambda_{2})-\langle B
\hat{\theta},\lambda_{2}\rangle_{\Gamma_{1}}- \langle
\lambda_{4},\hat{\theta}\!\mid_{\{x_{3}=0\}}\rangle_{(H^{1/2}_{e}(\{x_{3}=0\}))',H^{1/2}_{e}(\{x_{3}=0\})}
-c_{1}(\hat{\textbf{u}}_1,\hat{\theta},\lambda_{2})\nonumber\\
-c_{1}(\hat{\textbf{u}},\hat{\theta},\lambda_{2}^2) -Pr
M\;b_{1}(\hat{\theta},{\boldsymbol\lambda_{1}})+Pr
R\;(\hat{\theta},\lambda_{1_{3}})_{L^2(\Omega)}=-\gamma_{3}\Vert
\hat{\theta} \Vert^2_{L^2(\Omega)}.
\end{eqnarray}
Moreover, setting $\mathbf{v}={\boldsymbol\lambda}_1$ in
(\ref{uniq1a}) and $W=\lambda_{2}$ in (\ref{uniq1b}), we obtain
\begin{equation}
Pr\, a(\hat{\mathbf{u}},{\boldsymbol\lambda_{1}})+Pr M\,
b_{1}(\hat{\theta},{\boldsymbol\lambda_{1}})+c(\hat{\mathbf{u}}_1,\hat{\mathbf{u}},{\boldsymbol\lambda_{1}})+c(\hat{\mathbf{u}},\hat{\mathbf{u}}_2,{\boldsymbol\lambda_{1}})=\int_\Omega
Pr R \hat{\theta} \lambda_{{1}_3},\label{uniq1aq}
\end{equation}
\begin{equation}
c_{1}(\hat{\mathbf{u}}_1,\hat{\theta},\lambda_2)+c_{1}(\hat{\mathbf{u}},\hat{\theta}_2,\lambda_2)+a_{1}(\hat{\theta},\lambda_2)+\left\langle
B\hat{\theta},\lambda_2\right\rangle _{\Gamma_{1}}=\left\langle
\hat{\phi}_{1},\lambda_2\right\rangle
_{\Gamma_{0}\backslash\{x_{3}=0\}}.
\label{uniq1bq}
\end{equation}
Adding (\ref{uniq3d}), (\ref{uniq3f}), (\ref{uniq1aq}),
(\ref{uniq1bq}), and using (\ref{uniq1c}) we have
\begin{eqnarray}
2c(\hat{\mathbf{u}},\hat{\mathbf{u}},{\boldsymbol\lambda^2_{1}})+c(\hat{\mathbf{u}},\hat{\mathbf{u}},{\boldsymbol\lambda_{1}})+c_1(\hat{\mathbf{u}},\hat{\theta},\lambda_2)+2 c_1(\hat{\mathbf{u}},\hat{\theta},\lambda^2_2)\nonumber\\
+\langle {\boldsymbol
\lambda_{3}},\hat{\mathbf{g}}\!\mid_{\Gamma^1_0}\rangle_{(\widetilde{\mathbf{H}}^{1/2}_{e}(\Gamma_{0}))',\widetilde{\mathbf{H}}^{1/2}_{e}(\Gamma_{0})}
+\langle
\lambda_{4},\hat{\phi}_2\!\mid_{\{x_{3}=0\}}\rangle_{(H^{1/2}_{e}(\{x_{3}=0\}))',H^{1/2}_{e}(\{x_{3}=0\})}+\left\langle
\hat{\phi}_{1},\lambda_2\right\rangle
_{\Gamma_{0}\backslash\{x_{3}=0\}}\nonumber\\
=\gamma_{1}\Vert \text{rot } \hat{\mathbf{u}}\Vert^2_{L^2(\Omega)}
+\gamma_{2}\Vert\hat{\mathbf{u}}\Vert^2_{L^2(\Omega)}+\gamma_{3}\Vert \hat{\theta} \Vert^2_{L^2(\Omega)}.\label{hel6}
\end{eqnarray}
Considering the optimality condition (\ref{lagmul1b1})
$\langle\gamma_4
\hat{\mathbf{g}}_i+{\boldsymbol\lambda}_3^i,\mathbf{g}-\hat{\mathbf{g}}_i\rangle
_{(H^{\frac{1}{2}}(\Gamma^1_0))',H^{\frac{1}{2}}(\Gamma^1_0)}\geq
0,$ $i=1,2,$ taking $\mathbf{g}=\hat{\mathbf{g}}_1$ at $i=2$,
$\mathbf{g}=\hat{\mathbf{g}}_2$ at $i=1$, and adding these
relations, we obtain
\begin{eqnarray}\label{fn1}
\langle{\boldsymbol
\lambda}_3,\hat{\mathbf{g}}\vert_{\Gamma_{0}^{1}}\rangle_{(H^{\frac{1}{2}}(\Gamma^1_0))',H^{\frac{1}{2}}(\Gamma^1_0)}
\leq -\gamma_4\Vert
\hat{\mathbf{g}}\Vert^2_{H^{\frac{1}{2}}(\Gamma^1_0)}.
\end{eqnarray}
Analogously, from the optimality conditions (\ref{lagmul1b2})-(\ref{lagmul1b3}) we get
\begin{eqnarray}
\left\langle \hat{\phi}_{1},\lambda_2\right\rangle
_{(H^{\frac{1}{2}}(\Gamma_{0}\setminus\left\{
x_{3}=0\right\}))',H^{\frac{1}{2}}(\Gamma_{0}\setminus\left\{
x_{3}=0\right\})}\leq -\gamma_5\Vert \hat{\phi}_{1}\Vert^2_{H^{\frac{1}{2}}(\Gamma_0\setminus\{x_3=0\})},\label{fn2}\\
\langle\lambda_4,\hat{\phi}_{2}\vert_{\{x_3=0\}}\rangle
_{(H^{\frac{1}{2}}(\left\{
x_{3}=0\right\}))',H^{\frac{1}{2}}(\left\{ x_{3}=0\right\})}\leq
-\gamma_6\Vert \hat{\phi}_{2}\Vert^2_{H^{\frac{1}{2}}(\left\{
x_{3}=0\right\})}.\label{fn2}
\end{eqnarray}
By using (\ref{hel3}),(\ref{hel4}) and (\ref{estimulti}) we can bound the four terms of the left hand side of (\ref{hel6}) as follows
\begin{eqnarray}\label{hel9}
2c(\hat{\mathbf{u}},\hat{\mathbf{u}},{\boldsymbol\lambda^2_{1}})\leq
C(\mathcal{H}_0+1)^{2}\frac{1}{\beta_0^{1/2}}(\mathcal{M}[\hat{\mathbf{u}}_1,\hat{\theta}_1])^{1/2}\left(\Vert
\hat{\mathbf{g}}\Vert^2_{H^{\frac{1}{2}}(\Gamma_0^1)}\!+\!\Vert
\hat{\phi}_{2} \Vert^2_{H^{\frac{1}{2}}(\{x_{3}=0\})}\!+\!\Vert
\hat{\phi}_{1}
\Vert^2_{H^{\frac{1}{2}}(\Gamma_{0}\backslash\{x_{3}=0\})}\right)\!,
\end{eqnarray}
\begin{eqnarray}\label{hel10}
c(\hat{\mathbf{u}},\hat{\mathbf{u}},{\boldsymbol\lambda_{1}})\!\!\!&\leq &\!\!\!C(\mathcal{H}_0+1)^2\frac{1}{\beta_0^{1/2}}((\mathcal{M}[\hat{\mathbf{u}}_1,\hat{\theta}_1])^{1/2}+(\mathcal{M}[\hat{\mathbf{u}}_2,\hat{\theta}_2])^{1/2})\nonumber\\
&&\!\!\!\times \left(\Vert
\hat{\mathbf{g}}\Vert^2_{H^{\frac{1}{2}}(\Gamma_0^1)}+\Vert
\hat{\phi}_{2} \Vert^2_{H^{\frac{1}{2}}(\{x_{3}=0\})}+\Vert
\hat{\phi}_{1}
\Vert^2_{H^{\frac{1}{2}}(\Gamma_{0}\backslash\{x_{3}=0\})}\right),
\end{eqnarray}
\begin{eqnarray}\label{hel11}
c_1(\hat{\mathbf{u}},\hat{\theta},{\lambda_{2}})\!\!\!&\leq &\!\!\!C(\mathcal{H}_0+1)(\mathcal{H}_1+1)\frac{1}{\beta_0^{1/2}}((\mathcal{M}[\hat{\mathbf{u}}_1,\hat{\theta}_1])^{1/2}+(\mathcal{M}[\hat{\mathbf{u}}_2,\hat{\theta}_2])^{1/2})\nonumber\\
&&\!\!\!\times \left(\Vert
\hat{\mathbf{g}}\Vert^2_{H^{\frac{1}{2}}(\Gamma_0^1)}+\Vert
\hat{\phi}_{2} \Vert^2_{H^{\frac{1}{2}}(\{x_{3}=0\})}+\Vert
\hat{\phi}_{1}
\Vert^2_{H^{\frac{1}{2}}(\Gamma_{0}\backslash\{x_{3}=0\})}\right),
\end{eqnarray}
\begin{eqnarray}\label{hel12}
2c_1(\hat{\mathbf{u}},\hat{\theta},{\lambda^2_{2}})\!\!\!&\leq &\!\!\!C(\mathcal{H}_0+1)(\mathcal{H}_1+1)\frac{1}{\beta_0^{1/2}}(\mathcal{M}[\hat{\mathbf{u}}_2,\hat{\theta}_2])^{1/2}\nonumber\\
&&\!\!\!\times \left(\Vert
\hat{\mathbf{g}}\Vert^2_{H^{\frac{1}{2}}(\Gamma_0^1)}+\Vert
\hat{\phi}_{2} \Vert^2_{H^{\frac{1}{2}}(\{x_{3}=0\})}+\Vert
\hat{\phi}_{1}
\Vert^2_{H^{\frac{1}{2}}(\Gamma_{0}\backslash\{x_{3}=0\})}\right).
\end{eqnarray}
From (\ref{hel6}) and taking into account estimates (\ref{fn1})-(\ref{hel12}) we obtain
\begin{eqnarray}
&&\gamma_{1}\Vert \text{rot } \hat{\mathbf{u}}\Vert^2_{L^2(\Omega)}
+\gamma_{2}\Vert\hat{\mathbf{u}}\Vert^2_{L^2(\Omega)}+\gamma_{3}\Vert \hat{\theta} \Vert^2_{L^2(\Omega)}\leq (\mathcal{I} - \min\{\gamma_4,\gamma_5,\gamma_6\})\nonumber\\
&&\ \ \ \times\left(\Vert
\hat{\mathbf{g}}\Vert^2_{H^{\frac{1}{2}}(\Gamma_0^1)}+\Vert
\hat{\phi}_{2} \Vert^2_{H^{\frac{1}{2}}(\{x_{3}=0\})}+\Vert
\hat{\phi}_{1}
\Vert^2_{H^{\frac{1}{2}}(\Gamma_{0}\backslash\{x_{3}=0\})}\right),\label{hel15}
\end{eqnarray}
where
$\mathcal{I}:=C\max\{(\mathcal{H}_0+1)^{2},(\mathcal{H}_0+1)(\mathcal{H}_1+1)\}\frac{1}{\beta_0^{1/2}}((\mathcal{M}[\hat{\mathbf{u}}_1,\hat{\theta}_1])^{1/2}+(\mathcal{M}[\hat{\mathbf{u}}_2,\hat{\theta}_2])^{1/2}).$
By assuming $Pr$ large enough and $R,M$ small enough such that
$\mathcal{H}_0$ be small enough and
$\mathcal{I}<\min\{\gamma_4,\gamma_5,\gamma_6\},$ from (\ref{hel15})
we obtain
$$ \Vert\hat{\mathbf{u}}\Vert_{L^2(\Omega)}+\Vert \hat{\theta} \Vert_{L^2(\Omega)}+\Vert
\hat{\mathbf{g}}\Vert_{H^{\frac{1}{2}}(\Gamma_0^1)}+\Vert
\hat{\phi}_{2} \Vert_{H^{\frac{1}{2}}(\{x_{3}=0\})}+\Vert
\hat{\phi}_{1}
\Vert_{H^{\frac{1}{2}}(\Gamma_{0}\backslash\{x_{3}=0\})}=0,$$ which
implies that $\hat{\mathbf{z}}_1=\hat{\mathbf{z}}_2$. Thus, we have
proved the following theorem:
\begin{theorem}
If $Pr$ is large enough and $R,M$ are small enough, then the optimal
solution of problem (\ref{eq:funcional}) is unique.
\end{theorem}

\end{document}